\documentclass[preprint,12pt]{elsarticle}

\usepackage{amssymb,amsmath}
\usepackage[margin=1in]{geometry}
\usepackage{color}

\newcommand{\2}{\|_{H^2}}

\newcommand{\be}{\begin{equation}}
\newcommand{\ee}{\end{equation}}
\newcommand{\bee}{\begin{equation*}}
\newcommand{\eee}{\end{equation*}}
\newcommand{\bs}{\begin{split}}
\newcommand{\es}{\end{split}}
\newtheorem{thm}{Theorem}[section]

\newtheorem{lem}{Lemma}[section]

\newtheorem{rmk}{Remark}[section]

\numberwithin{equation}{section}
\journal{Journal des Math\'ematiques Pures et Appliqu\'ees}

\begin{document}

\begin{frontmatter}



\title{Hilbert Expansion of the coupled radiation-Euler in the equilibrium regime}

 \author[1]{Qiangchang Ju}
\address[1]{Institute of Applied Physics and Computational Mathematics, Beijing, P. R. China}
\ead{ju\_qiangchang@iapcm.ac.cn}
\author[2]{Lei Li}
\address[2]{School of Mathematics and Statistics, Fuyang Normal University, Fuyang, 236037, P. R. China}
\ead{202410001@fynu.edu.cn}
\author[3]{Zhengce Zhang \corref{*}}
\address[3]{School of Mathematics and Statistics, Xi'an Jiaotong University, Xi'an, 710049, P. R. China}
\ead{zhangzc@mail.xjtu.edu.cn}
\cortext[*]{Corresponding author.}

\begin{abstract}
 In this paper, we establish the validity of the Hilbert expansion for the coupled radiation-Euler model with non-relativistic source term for general initial data on the torus $\mathbb{T}^{3}$, which leads to the simplified equilibrium-diffusion limit and the initial layer corrections for the radiation intensity coupled with the temperature.
\end{abstract}

\begin{keyword}

Hilbert expansion\sep Radiation-Euler model; General initial data\sep Equilibrium-diffusion limit\sep Initial-layer corrections
\end{keyword}

\end{frontmatter}













\section{Introduction and Main Results}

{\subsection{The coupled radiation-Euler model and main results}
The Euler system coupled with a radiative transfer equation can describe the motion of a compressible inviscid fluid under the presence of the radiative field. Following \cite{bu-8} and \cite{bu-9}, we derive the  non-dimensional equations  as follows
\begin{equation}\label{research equations-x-1}\left\{
\begin{split}
&\partial_{t}\rho+\mathrm{div}_{x}(\rho\vec{u})=0,\\
&\partial_{t}(\rho\vec{u})+\mathrm{div}_{x}(\rho\vec{u}\otimes\vec{u})+\nabla_{x}P
=-\mathcal{P}\vec{S}_{F},\\
&\partial_{t}(\rho E)+\mathrm{div}_{x}(\rho E\vec{u}+P\vec{u})
=-\mathcal{P}\mathcal{C}S_{E},\\
&\frac{1}{\mathcal{C}}\partial_{t}f+\vec w\cdot\nabla_{x}f=S.\\
\end{split}\right.
\end{equation}
Here the unknowns $\rho, \vec u=(u^{1}, u^{2}, u^{3})$ and $E=e+\frac{1}{2}|\vec u|^{2}$ denote the density, velocity, and total energy, respectively. $f=f(t, \vec x, \nu, \vec w)$ is the radiative intensity which depends on time $t\geq 0$, the spatial variable $\vec x\in\mathbb{T}^{3}=(\mathbb{R}\setminus(2\pi\mathbb{Z}))^3$, the frequency $\nu\geq 0$ and direction $\vec w\in\mathbb{S}^{2}$ of photons with $\mathbb{S}^{2}\subset\mathbb{R}^{3}$ being the unit sphere. The pressure $P=P(\rho, \theta)$ and the internal energy $e=e(\rho, \theta)$ are smooth functions of $\rho$ and the temperature $\theta$. 

In \eqref{research equations-x-1}, the radiative source term $S$  is given by
\begin{equation}\label{source-term}
S=\mathcal{L}_{a}\frac{\nu_{0}}{\nu}\sigma_{a}(\nu_{0})\Big(\Big(\frac{\nu}{\nu_{0}}\Big)^{3}B(\nu_{0}, \theta)-f\Big)+\mathcal{L}_{s}\frac{\nu_{0}}{\nu}\sigma_{s}(\nu_{0})\Big(\Big(\frac{\nu}{\nu_{0}}\Big)^{3}\frac{1}{4\pi}\int_{\mathbb{S}^{2}}\frac{\nu_{0}}{\nu'}f(\nu', \vec w')\mathrm{d}\vec w'-f\Big),
\end{equation}
where 
\begin{equation}\nonumber
B(\nu_{0}, \theta)=\frac{15\nu_0^3}{4\pi^5}\Big(e^{\frac{\nu_{0}}{\theta}}-1\Big)^{-1},
\end{equation}
\begin{equation}\nonumber
\nu=\nu_{0}\gamma_{L}(1-\vec{w}\cdot\vec u/\mathcal{C}),\;\; \; \nu'=\nu\frac{1-\vec{w}\cdot\vec u/\mathcal{C}}{1-\vec{w'}\cdot\vec u/\mathcal{C}},\;\;\;\gamma_{L}=1/\sqrt{1-|\vec{u}|^{2}/\mathcal{C}^{2}}.
\end{equation}
Furthermore, the radiative flux $\vec{S}_{F}$ and the radiative energy source $S_{E}$ are defined by
\begin{equation}\nonumber
\vec{S}_{F}=\int_{0}^{\infty}\int_{\mathbb{S}^{2}}\vec{w}S\mathrm{d}\vec w\mathrm{d}\nu,\ \ S_{E}=\int_{0}^{\infty}\int_{\mathbb{S}^{2}}S\mathrm{d}\vec w\mathrm{d}\nu.
\end{equation}

In addition, $\sigma_{a}(\nu_0)$ and $\sigma_{s}(\nu_0)$ represent the absorption and the scattering coefficient, respectively. The non-dimensional parameters $\mathcal{L}_{a}$ and $\mathcal{L}_{s}$ describe the dominant effects of absorption and scattering, respectively. $\mathcal{C}$ is always large parameter for a flow non-relativistic. $\mathcal{P}$ measures the ratio of radiative energy over the internal energy.

In what follows, we take $\nu=\nu_{0}=\nu'$ to derive the non-relativistic source term
 \begin{equation}\nonumber
S=\mathcal{L}_{a}\sigma_{a}[B(\nu, \theta)-f(t, \vec x, \nu, \vec w)]+\mathcal{L}_{s}\sigma_{s}\Big(\frac{1}{4\pi}\int_{\mathbb{S}^{2}}f(t, \vec x, \nu, \vec w)\mathrm{d}\vec w-f(t, \vec x, \nu, \vec w)\Big).
\end{equation}
Moreover, we assume that $\sigma_a$ and $\sigma_s$ are independent of the frequency and make the so-called gray hypothesis.  If we denote $\int_{0}^{\infty}f \mathrm{d}\nu$ by $f$, then $\eqref{research equations-x-1}_{4}$ can be written as
\begin{equation}\nonumber
\begin{aligned}
{\mathcal{C}^{-1}}\partial_{t}f+\vec w\cdot\nabla_{x}f=\mathcal{L}_{a}(B(\theta)-f)
+\mathcal{L}_{s}\Big(\frac{1}{4\pi}\int_{\mathbb{S}^{2}}f\mathrm{d}\vec{w}-f\Big),
\end{aligned}
\end{equation}
where $B(\theta)=C\theta^4$ for some positive constant.  

In order to get the equilibrium-diffusion model, we suppose $$\mathcal{L}_{a}=\epsilon^{-1},\;\;\;\mathcal{L}_{s}=\epsilon,\;\;\;\mathcal{P}=1,\;\;\;\mathcal{C}=\epsilon^{-1},$$
where $\epsilon$ is the Knudsen number. We shall focus on the fluids obeying the ideal polytropic gas relations
$$e=C_V\theta\;\;\;\;P=R\rho\theta,$$
where the specific gas constant $R$ and the specific heat at constant volume $C_V$ are positive constants. We take $\sigma_a,\sigma_s, R$ and $C_V$ to be one for convenience. Therefore the rescaled unknowns $(\rho^\epsilon, \vec{u}^\epsilon, \theta^\epsilon,f^\epsilon)$ satisfy
\begin{equation}\label{research equations}\left\{
\begin{split}
&\partial_{t}\rho^{\epsilon}+\mathrm{div}_{x}(\rho^{\epsilon}\vec{u}^{\epsilon})=0,~~\mathrm{in} \ \ (0,T)\times\mathbb{T}^{3},\\
&\partial_{t}(\rho^{\epsilon}\vec{u}^{\epsilon})+\mathrm{div}_{x}(\rho^{\epsilon}\vec{u}^{\epsilon}\otimes\vec{u}^{\epsilon})+\nabla_{x}P^{\epsilon}
=\Big\langle\Big(\frac{1}{\epsilon}+\epsilon\Big)\vec{w}(f^{\epsilon}-\overline {f^{\epsilon}})\Big\rangle,~~\mathrm{in} \ \ (0,T)\times\mathbb{T}^{3},\\
&\partial_{t}(\rho^{\epsilon}E^{\epsilon})+\mathrm{div}_{x}(\rho^{\epsilon}E^{\epsilon}\vec{u}^{\epsilon}+P^{\epsilon}\vec{u}^{\epsilon})
=\frac{1}{\epsilon^{2}}(\overline {f^{\epsilon}}-B(\theta^{\epsilon})),~~\mathrm{in} \ \ (0,T)\times\mathbb{T}^{3}, \\
&\partial_{t}f^{\epsilon}+\frac{1}{\epsilon}\vec w\cdot \nabla_{x}f^{\epsilon}+f^{\epsilon}-\overline {f^{\epsilon}}=\frac{1}{\epsilon^{2}}(B(\theta^{\epsilon})-f^{\epsilon}),~~\mathrm{in} \ \ (0,T)\times\mathbb{T}^{3}\times \mathbb{S}^{2},\\
&\rho^{\epsilon}(0, \vec{x})=\rho^{0}(\vec x),\;\;\vec{u}^{\epsilon}(0, \vec{x})=\vec{u}^{0}(\vec x),\;\;\theta^{\epsilon}(0, \vec{x})=\theta^{0}(\vec{x})~~\;\mathrm{in}  \ \ \mathbb{T}^{3},\\
&f^{\epsilon}(0, \vec{x}, \vec{w})=h(\vec{x}, \vec{w})~~\;\mathrm{in}  \ \ \mathbb{T}^{3}\times\mathbb{S}^{2},\\
\end{split}\right.
\end{equation}
where $E^{\epsilon}=\theta^{\epsilon}+\frac{1}{2}|\vec{u}^{\epsilon}|^{2}$, $P^{\epsilon}=\rho^{\epsilon}\theta^{\epsilon}$ and $\overline {f^{\epsilon}}=\langle f^{\epsilon}\rangle=\frac{1}{4\pi}\int_{\mathbb{S}^{2}}f^{\epsilon}\mathrm{d}\vec w$.

The main objective of this paper is to study the diffusive limit $(\epsilon\rightarrow 0)$ for the general initial data. For the general initial data, the initial layers would arise. By the Hilbert expansion (see Section \ref{san}), we can derive the zeroth order interior system and initial layers as follows:
\begin{equation}\label{equations about f0 epsilon and theta-0}\left\{
\begin{split}
&f_{0}=\overline {f_{0}}=\theta_{0}^{4},\\
&\partial_{t}\rho_{0}+\mathrm{div}(\rho_{0}\vec u_{0})=0,~~\mathrm{in} \ \ (0,T)\times \mathbb{T}^{3},\\
&\partial_{t}(\rho_{0}\vec{u}_{0})+\mathrm{div}_{x}(\rho_{0}\vec{u}_{0}\otimes\vec u_{0})+\nabla_{x}(\rho_{0}\theta_{0})
=-\frac{1}{3}\nabla_{x}\theta_{0}^{4},~~\mathrm{in} \ \ (0,T)\times \mathbb{T}^{3},\\
&\partial_{t}(\rho_{0}E_{0}+\theta_{0}^{4})+\mathrm{div}_{x}(\rho_{0}E_{0}\vec{u}_{0}+\rho_{0}\theta_{0}\vec{u}_{0})
-\frac{1}{3}\Delta_{x}\theta_{0}^{4}=0,~~\mathrm{in} \ \ (0,T)\times\mathbb{T}^{3}, \\
&\rho_{0}(0, \vec{x})=\rho^{0}(\vec x),\;\;\vec{u}_{0}(0, \vec{x})=\vec{u}^{0}(\vec x)~~\mathrm{in}  \ \ \mathbb{T}^{3},\\
&\theta_{0}(0, \vec{x})=\theta_{0}^{0}(\vec{x}),\;\;\overline {f_{0}}(0, \vec{x})=(\theta_{0}^{0})^{4}~~\mathrm{in}  \ \ \mathbb{T}^{3},
\end{split}\right.
\end{equation}
and
\begin{equation}\label{equations about f I0,theta I0-jiajia}\left\{
\begin{split}
&f_{I, 0}=\hat{f}_{0}-(\theta_{0}^{0})^{4}, \theta_{I, 0}=\hat\theta_{0}-\theta_{0}^{0},\\
&\frac{\partial\hat{f}_{0}}{\partial\tau}+\hat{f}_{0}=B(\hat\theta_{0}),\\
&\frac{\partial\hat\theta_{0}}{\partial\tau}=\frac{1}{\rho^{0}}(\overline{\hat{f}_{0}}-B(\hat\theta_{0})),\\
&\hat{f}_{0}(0, \vec{x}, \vec{w})=h(\vec{x}, \vec{w}), \hat\theta_{0}(0, \vec{x})=\theta^{0}(\vec x),\\
&\lim_{\tau\rightarrow\infty}\hat{f}_{0}(\tau, \vec{x}, \vec{w})=(\theta_{0}^{0}(\vec x))^{4},\ \ \lim_{\tau\rightarrow\infty}\hat\theta_{0}(\tau, \vec{x}, \vec{w})=\theta_{0}^{0}(\vec x),
\end{split}\right.
\end{equation}
where $\tau=t/\epsilon^{2}$, $E_{0}=\theta_{0}+\frac{1}{2}|\vec{u}_{0}|^{2}$ and $\theta_{0}^{0}(\vec x)$ is the initial data of $\theta_{0}(t, \vec x)$ and will be determined by the given initial data $\rho^{0}, \theta^{0}$ and $h$ by studying \eqref{equations about f I0,theta I0-jiajia}. Furthermore, the zeroth order initial layers of $\rho^\epsilon$ and $\vec{u}^{\epsilon}$ won't appear, which can be seen in Section \ref{san} for detailed discussions.

The main theorem in the paper is as follows.

{\thm \label{mainthm}
Assume that $h(\vec{x}, \vec{w})\in L^{\infty}(\mathbb{S}^{2}; H^{2N+2}(\mathbb{T}^{3}))$ and $\rho^{0}(\vec x), \vec{u}^{0}(\vec x), \theta^{0}(\vec{x})\in H^{2N+2}\\(\mathbb{T}^{3})$ with $h(\vec{x}, \vec{w})>0, \rho^{0}(\vec x), \theta^{0}(\vec x)\geq a$ and $\|h-(\theta^{0})^{4}\|_{L^{\infty}(\mathbb{S}^{2}; H^{2N+2}(\mathbb{T}^{3}))}\leq \eta$, where $N\geq 17$, $a$ and $\eta$ are positive constants, and $\eta$ is suitably small and independent of $\epsilon$. There exists a $T>0$ independent of $\epsilon$, such that the problem (\ref{research equations}) has a unique solution $(f^{\epsilon}, \rho^{\epsilon}, \vec{u}^{\epsilon}, \theta^{\epsilon})\in C([0, T]; L^{2}(\mathbb{S}^{2}; H^{3}(\mathbb{T}^{3})))\cap L^{\infty}((0, T)\times\mathbb{T}^{3}\times\mathbb{S}^{2})\times C([0, T]; H^{3}(\mathbb{T}^{3}))\times C([0, T]; H^{3}(\mathbb{T}^{3}))\times C([0, T]; H^{3}(\mathbb{T}^{3}))$. Furthermore, the unique solution $(f^{\epsilon}, \rho^{\epsilon}, \vec{u}^{\epsilon}, \theta^{\epsilon})$ satisfies
\begin{equation}\label{jielunbudengshi 1}
\|f^{\epsilon}-\theta_{0}^{4}-f_{I, 0}\|_{L^{\infty}((0, T)\times\mathbb{T}^{3}\times\mathbb{S}^{2})}=O(\epsilon),
\end{equation} 
\begin{equation}\label{jielunbudengshi 2}
\|\theta^{\epsilon}-\theta_{0}-\theta_{I, 0}\|_{C([0, T]; H^{3}(\mathbb{T}^{3}))}=O(\epsilon),
\end{equation}
and
\begin{equation}\label{jielunbudengshi 3}
\|(\rho^{\epsilon}-\rho_{0}, \vec{u}^{\epsilon}-\vec{u}_{0})\|_{C([0, T]; H^{3}(\mathbb{T}^{3}))}
=O(\epsilon),
\end{equation}
where $(\rho_{0}, \vec{u}_{0}, \theta_{0})\in (C([0, T]; H^{2N+2}(\mathbb{T}^{3}))\times (C([0, T]; H^{2N+2}(\mathbb{T}^{3}))\times (C([0, T]; H^{2N+2}(\mathbb{T}^{3}))\cap L^{2}(0, T; H^{2N+3}(\mathbb{T}^{3})))$ satisfies \eqref{equations about f0 epsilon and theta-0} and the zeroth-order initial layer $(f_{I, 0}, \theta_{I, 0})\in (C^{1}([0, \infty); \\ L^{\infty}(\mathbb{S}^{2}; H^{2N+2}(\mathbb{T}^{3})))\cap L^{1}([0, \infty); L^{\infty}(\mathbb{S}^{2}; H^{2N+2}(\mathbb{T}^{3}))))\times (C^{1}([0, \infty); H^{2N+2}(\mathbb{T}^{3}))\cap L^{1}([0, \infty); \\H^{2N+2}(\mathbb{T}^{3})))$ satisfies \eqref{equations about f I0,theta I0-jiajia}.
}
{\rmk The formal derivation of the equilibrium-diffusion limit of the coupled radiative-Euler model with relativistic source terms \eqref{research equations-x-1} was provided in \cite{bu-8}. In this paper we indeed verify the limit for the simplified case of non-relativistic term. In this case, we miss the additional convection terms due to the relativistic source terms in the equilibrium-diffusion limit equation.}

\subsection{Related work} 
A classical theorem for the diffusive limit of the single neutron transport equation in \cite{lionssidawenzhang} states that the solution $f^{\epsilon}$ to the problem as
\begin{equation}\label{steady neutron transport equation in a unit disk}\left\{
\begin{split}
&\epsilon \vec w\cdot \nabla_{x}f^{\epsilon}+f^{\epsilon}-\overline {f^{\epsilon}}=0\ \ \mathrm{in} \ \ \Omega\times \mathbb{S}^{1},\\
&f^{\epsilon}(\vec{x}_{0}, \vec{w})=g(\vec{x}_{0}, \vec{w})\ \ \mathrm{for} \ \ \vec{w}\cdot\vec{n}<0~~\mathrm{and}~~\vec{x}_{0}\in\partial\Omega,\\
\end{split}\right.
\end{equation}
satisfies
\begin{equation}\label{jieguo1}
\|f^{\epsilon}-\overline{f_{0}}-f_{B, 0}\|_{L^{\infty}(\Omega\times \mathbb{S}^{1})}=O(\epsilon),
\end{equation}
where $\overline{f_{0}}$ is the solution of Laplacian equation and $f_{B, 0}$ is the boundary layer correction defined by the Milne problem. However, when the boundary is non-flat, the convergence stated in \eqref{jieguo1} fails, as demonstrated in \cite{geometric-CMP}. In \cite{geometric-CMP}, a geometric correction $\overline{f_{0}^{\epsilon}}$ and $f_{B, 0}^{\epsilon}$ was constructed to deal with the curvature effects. With this modification, $\|f^{\epsilon}-\overline{f_{0}^{\epsilon}}-f_{B, 0}^{\epsilon}\|_{L^{\infty}}$ is shown to converges to zero as $\epsilon\rightarrow 0$ in the 2D unit disk \cite{geometric-CMP}. Subsequent extensions of this result have been established in the annulus \cite{yuanhuanneutrontransport}, in general 2D convex domains \cite{convex-domain-123, unsteady-neutron-transport-convex-domain, in-flow-boundry-neutron-transport}, and in 3D convex domains \cite{3-d-diffusive-limit-convex-domain, 3-d-zuinanqingxing}.

Recently, there are many results \cite{bu-3, bu-2, bu-4, traceth, bu-1} about the following kind of equilibrium-regime coupled model:
\begin{equation}\label{research equations-dayu0-c1-equilibrium-jia}\left\{
\begin{split}
&{\epsilon}^{2}\partial_{t}f^{\epsilon}+\epsilon \vec w\cdot \nabla_{x}f^{\epsilon}+f^{\epsilon}=B(\theta^{\epsilon}),\ \ (t, \vec x, \vec w)\in \mathbb{R}^+\times \Omega\times \mathbb{S}^{d-1},\\
&\epsilon^{2}\partial_{t}\theta^{\epsilon}-\epsilon^{2}\Delta_{x} \theta^{\epsilon}
=\overline {f^{\epsilon}}-B(\theta^{\epsilon}),\ \ \ \ (t, \vec x)\in \mathbb{R}^+\times \Omega.\\
\end{split}\right.
\end{equation}
Klar and Schmeiser \cite{bu-1} proved the well-posedness of system (\ref{research equations-dayu0-c1-equilibrium-jia}) via fixed point theory. They also established a diffusive limit as $\epsilon\rightarrow 0$, with temperature converging weakly in $H^{1}((0, T)\times\Omega)$ and radiation intensity weak-$\ast$ in $L^{\infty}((0, T)\times\Omega\times\mathbb{S}^{2})$ where $\Omega$ is smooth and bounded. The nonlinear Milne problem for the boundary layer was also treated in \cite{bu-1}, yielding an existence result. They also prevent numerical results in different physical situations. Recently, Ghattassi, Huo and Masmoudi \cite{traceth} proved the global existence of weak solutions for (\ref{research equations-dayu0-c1-equilibrium-jia}) and the convergence of the weak solutions to a nonlinear diffusion model under the diffusive limit, which extended the work by Klar and Schmeiser \cite{bu-1}. Ghattassi, Huo and Masmoudi \cite{traceth-1} studied the diffusive limit of the steady state of (\ref{research equations-dayu0-c1-equilibrium-jia}) for non-homogeneous Dirichlet boundary conditions in a bounded domain with flat boundaries. Considering the arising of the boundary layers, a composite approximate solution is constructed using asymptotic analysis by resorting to the results of \cite{traceth-2} on the boundary layer problem to the steady state of (\ref{research equations-dayu0-c1-equilibrium-jia}). Furthermore, Ghattassi, Huo and Masmoudi \cite{traceth-3} studied the diffusive limit of the steady state of (\ref{research equations-dayu0-c1-equilibrium-jia}) in a unit disk with non-flat boundaries by considering the geometric corrections. The models (\ref{research equations-dayu0-c1-equilibrium-jia}) considered in \cite{traceth-1, traceth-2, traceth-3} are the steady case of equilibrium regime. There have been extensive research efforts to derive the diffusive limit of the non-equilibrium regime case of \eqref{research equations-dayu0-c1-equilibrium-jia}, see \cite{yuanmoxing-bu, yuanmoxing-bu-jia2} and the references cited therein. 

All of the above-mentioned works are only concerned on a simplified model of \eqref{research equations-x-1} for the given density and velocity. There are some works \cite{bu-6-jia0, yuanmoxing-bu-jia3jia, yuanmoxing-bu-jia3} about the diffusive limit of the non-equilibrium regime case of \eqref{research equations-dayu0-c1-equilibrium-jia} under the grey assumption or P1 hypothesis. One can also see the study of low Mach number limit of the radiation-hydrodynamic model in \cite{DN1, bujia-4, bujia-3, bujia-5, bujia-1, jia-jia-4, jia-jia-4-jia, bujia-6}. In this paper, we justify rigorously the equilibrium-diffusion limit of the compressible Euler model coupled with a radiative transfer equation arising in radiation hydrodynamics for the general initial data case.

\subsection{Major difficulties and methods.} The main difficulties arise from the $O(\epsilon^{-2})$ singular source terms which are the corresponding $O(1)$ terms in \cite{yuanmoxing-bu-jia3jia}. Motivated by \cite{traceth-1}, we expand $(f^{\epsilon}, \rho^{\epsilon}, \vec{u}^{\epsilon}, \theta^{\epsilon})$ to $N$-th order, $N\geq 17$, to overcome the singularities arising from $\epsilon$. In the following, we prove the well-posedness of the expansion of interior, initial layer and remainder parts, and get the estimates about $\epsilon$ by careful calculations. First, the close initial assumption $\|h-(\theta^{0})^{4}\|_{L^{\infty}(\mathbb{S}^{2}; H^{2N+2}(\mathbb{T}^{3}))}\leq \eta$, where $\eta$ is suitably small, plays an important role in the proof of the global existence and the exponential decay of the strong initial layers of the radiative intensity coupled with temperature. Second, the main step of getting the convergence rates of $(f^{\epsilon}, \rho^{\epsilon}, \vec{u}^{\epsilon}, \theta^{\epsilon})$ is estimating the remainders $(f_{r}, \rho_{r}, \vec{u}_{r}, \theta_{r})$ about $\epsilon$. In order to achieve this purpose, we have to get the following estimate
\begin{equation}\label{mubiaoestimate}
\begin{aligned}
&\|V_{r}\|_{L_{t}^{\infty}H_{\vec x}^{3}}^{2}+\|f_{r}\|_{L_{t}^{\infty}L_{\vec w}^{2}H_{\vec x}^{3}}^{2}+\frac{1}{\epsilon^{2}}\|f_{r}-\overline{f_{r}}\|_{L_{t}^{2}L_{\vec w}^{2}H_{\vec x}^{3}}^{2}+\frac{1}{\epsilon^{2}}\|4(\theta^{N})^{3}\theta_{r}-\overline{f_{r}}\|_{L_{t}^{2}H_{\vec x}^{3}}^{2}
\\\leq& C(t)\Big(\frac{1}{\epsilon^{16}}\|(R_{1}, R_{2}, R_{3}, R_{4}, R)\|_{L_{t}^{2}L_{\vec w}^{2}H_{\vec x}^{3}}^{2}\Big),
\end{aligned}
\end{equation}
in Lemma \ref{le4.2}, where $V_{r}=(\rho_{r}, \vec{u}_{r}, \theta_{r})$, and $R_{1}, R_{2}, R_{3}, R_{4}$ and $R$ can be seen as source terms. 

Our main idea is that we apply properties of the commutator operator and the Young's inequality to derive the following kind of inequality:
\begin{equation}\nonumber
\begin{aligned}
&\|D_{x}^{\gamma}V_{r}(t)\|_{L_{\vec x}^{2}}^{2}+\|D_{x}^{\gamma}f_{r}(t)\|_{L_{\vec w}^{2}L_{\vec x}^{2}}^{2}
+\frac{1}{\epsilon^{2}}\int_{0}^{t}\int_{\mathbb{T}^{3}}\int_{\mathbb{S}^{2}}|D_{x}^{\gamma}f_{r}-D_{x}^{\gamma}\overline{f_{r}}|^{2}
\mathrm{d}\vec w\mathrm{d}\vec x\mathrm{d}s\\&+\frac{1}{\epsilon^{2}}\int_{0}^{t}\int_{\mathbb{T}^{3}}|D_{x}^{\gamma}(4(\theta^{N})^{3}\theta_{r}-\overline{f_{r}})|^{2}\mathrm{d}\vec x\mathrm{d}s
\\\leq&C\Big(\frac{1}{\epsilon^{4}}\|(R_{1}, R_{2}, R_{3}, R_{4}, R)\|_{L_{t}^{2}L_{\vec w}^{2}H_{\vec x}^{3}}^{2}+\frac{1}{\epsilon^{2}}\|f_{r}-\overline{f_{r}}\|_{L_{t}^{2}L_{\vec x}^{2}}^{2}\Big)+C(\|V_{r}\|_{L_{t}^{2}L_{\vec x}^{2}}+\|D_{x}^{\gamma}V_{r}\|_{L_{t}^{2}L_{\vec x}^{2}}\\&+\|f_{r}\|_{L_{t}^{2}L_{\vec x}^{2}L_{\vec w}^{2}}^{2}+\|D_{x}^{\gamma}f_{r}\|_{L_{t}^{2}L_{\vec x}^{2}L_{\vec w}^{2}}^{2})+\frac{C}{\epsilon^{2}}(\|f_{r}-\overline{f_{r}}\|_{L_{t}^{2}L_{\vec w}^{2}L_{\vec x}^{2}}^{2}
+\|D_{x}^{\gamma-1}f_{r}-D_{x}^{\gamma-1}\overline{f_{r}}\|_{L_{t}^{2}L_{\vec w}^{2}L_{\vec x}^{2}}^{2})\\&+\frac{C}{\epsilon^{4}}(\|f_{r}\|_{L_{t}^{2}L_{\vec w}^{2}L_{\vec x}^{2}}^{2}+\|D_{x}^{\gamma-1}f_{r}\|_{L_{t}^{2}L_{\vec w}^{2}L_{\vec x}^{2}}^{2}+\|\theta_{r}\|_{L_{t}^{2}L_{\vec w}^{2}}^{2}+\|D_{x}^{\gamma-1}\theta_{r}\|_{L_{t}^{2}L_{\vec x}^{2}}^{2}).
\end{aligned}
\end{equation}
i.e. the $\gamma$-th order derivative estimates can be controlled by the $(\gamma-1)$-th order ones, where $0\leq\gamma\leq 3$. We note that the singular term $\frac{C}{\epsilon^{2}}\|f_{r}-\overline{f_{r}}\|_{L_{t}^{2}L_{\vec w}^{2}L_{\vec x}^{2}}^{2}$ won't appear when $\gamma=0$ in the calculation process. In Section \ref{sec4.2}, we construct the solutions $\{\rho_{r}^{k+1}, \vec{u}_{r}^{k+1}, \theta_{r}^{k+1}, f_{r}^{k+1}\}$, iteratively, and get the following estimate
\begin{equation}\nonumber
\begin{aligned}
&\|f_{r}^{k+1}\|_{L_{T}^{\infty}H^{3}_{\vec x}L_{\vec w}^{2}}+\|V_{r}^{k+1}\|_{L_{T}^{\infty}H_{\vec x}^{3}}+\frac{1}{\epsilon^{2}}\|f_{r}^{k+1}-4(\theta^{N})^{3}\theta_{r}^{k+1}\|_{L_{T}^{2}L_{\vec w}^{2}H_{\vec x}^{3}}+\frac{1}{\epsilon}\|f_{r}^{k+1}-\overline{f_{r}^{k+1}}\|_{L_{T}^{2}L_{\vec w}^{2}H_{\vec x}^{3}}
\\\leq& \frac{C(T)}{\epsilon^{8}}\|(R_{1}, R_{2}, R_{3}, R_{4}, R)\|_{L_{T}^{2}L_{\vec w}^{2}H_{\vec x}^{3}}
\\\leq& \frac{C(T)}{\epsilon^{8}}(\|\theta_{r}^{k}\|_{L_{T}^{\infty}H^{3}_{\vec x}}^{2}+\|\theta_{r}^{k}\|_{L_{T}^{\infty}H^{3}_{\vec x}}^{4}+\epsilon^{N+1})
\leq \frac{C(T)}{\epsilon^{8}}(\epsilon^{2q}+\epsilon^{N+1})
\leq C(T)(\epsilon^{2q-8}+\epsilon^{N-7}),
\end{aligned}
\end{equation}
where $V_{r}^{k}=(\rho_{r}^{k}, \vec{u}_{r}^{k}, \theta_{r}^{k})$, by assuming $\|f_{r}^{k}\|_{L^{\infty}_{T}L^{2}_{\vec w}H^{3}_{\vec x}}+\|V_{r}^{k}\|_{L^{\infty}_{T}H^{3}_{\vec x}}\leq \epsilon^{q}$ and careful calculations. In order to close the energy, we need that $2q-8>q$, $N-7>q$ and $q, N\in\mathbb{Z}$, i.e. $q\geq 9$ and $N>7+q$. So, $N\geq 17$ is required.
{\rmk In contrast to Ghattassi, Huo and Masmoudi \cite{traceth-1}, where 1D boundary points ensure infinite smoothness via elliptic theory and thus low regularity suffices, our problem requires higher initial regularity, which is depleted by the initial layer and further limited by interior regularity constraints. The expansion is of higher order due to stronger coupling, ruling out classical $L^{2}$-$L^{\infty}$ estimates; hence uniform higher-derivative energy estimates for the remainders are necessary for the final $L^\infty$ convergences.
}}

\vskip 3mm

{ Throughout this paper, $C$ denotes a certain positive constant and $C(\cdot)$ and $C_{1}(\cdot)$ are the positive constants depending on the quantity $"\cdot"$. $D_{x}$ and $D_{t}$ represent the operator $\partial_{x_{i}}, i=1, 2, 3$ and $\partial_{t}$, respectively. $H^{k}(\mathbb{T}^{3})=W^{k, 2}(\mathbb{T}^{3})$ denotes the usual Lebesgue space on $\mathbb{T}^{3}$ with norm $\|\cdot\|_{H^{k}(\mathbb{T}^{3})}$. We denote by $L^{p}(0, T; H^{k}(\mathbb{T}^{3}))$ $(1\leq p\leq\infty)$ the space of $L^{p}$ functions on $(0, T)$ with values in $H^{k}(\mathbb{T}^{3})$, and $C^{0}(I, H^{k}(\mathbb{T}^{3}))$ standards for the space of continuous functions on the interval $I$ with values in $H^{k}(\mathbb{T}^{3})$. For simplicity, we use $L_{T}^{p}, L_{\vec w}^{p}, H_{\vec x}^{k}$ to denote $L^{p}(0, T), L^{p}(\mathbb{S}^{2}), H^{k}(\mathbb{T}^{3})$, where $1\leq p\leq\infty$, respectively. For $k\geq 0$ and $k\in\mathbb{Z}$, $C^{k}([0, \infty)\times\mathbb{T}^{3})$ denotes the set satisfying $D_{x}^{l}D_{t}^{r}f\in C([0, \infty)\times\mathbb{T}^{3})$ where $0\leq l+r\leq k$.

The paper is organized as follows. In Section \ref{er}, we discuss the well-posedness of neutron transport equation. In Section \ref{san}, we give the Hilbert expansion of the system and construct the asymptotic expansion. In Section \ref{si}, we prove Theorem \ref{mainthm}. 
}

\section{Well-posedness of the single radiative transfer equation}\label{er}
In this section, we consider the well-posedness of the single radiative transfer equation, which can be written as
\begin{equation}\label{research equation}\left\{
\begin{split}
&{\epsilon}^{2}\frac{\partial I}{\partial t}+\epsilon \vec{w}\cdot \nabla_{x}I+{\epsilon}^{2}(I-\overline I)+I=F(t, \vec{x}, \vec{w})~~\mathrm{in} \ \ (0, \infty)\times\mathbb{T}^{3}\times \mathbb{S}^{2},\\
&I(0, \vec{x}, \vec{w})=h(\vec{x}, \vec{w})~~\mathrm{in} \ \ \mathbb{T}^{3}\times \mathbb{S}^{2}.\\
\end{split}\right.
\end{equation}
{\lem \label{main result}
Assume $F(t, \vec{x}, \vec{w})\in L^{\infty}((0, \infty)\times\mathbb{T}^{3}\times \mathbb{S}^{2}), h(\vec{x}, \vec{w})\in L^{\infty}(\mathbb{T}^{3}\times \mathbb{S}^{2})$ and $F, h\geq 0$. Then there exists a unique nonnegative solution $I(t, \vec{x}, \vec{w})\in L^{\infty}((0, \infty)\times\mathbb{T}^{3}\times \mathbb{S}^{2})$ of the initial value problem of \eqref{research equation} which satisfies
\begin{equation}\label{estimate result}
\|I\|_{L^{\infty}((0, \infty)\times\mathbb{T}^{3}\times \mathbb{S}^{2})}\leq C(\|h\|_{L^{\infty}(\mathbb{T}^{3}\times \mathbb{S}^{2})}+\|F\|_{L^{\infty}((0, \infty)\times \mathbb{T}^{3}\times \mathbb{S}^{2})}),
\end{equation}
where $C$ is independent of $\epsilon$.
}

{\lem \label{lianxujieguo}
For any $0<T_{0}<\infty$ and $m\geq 0$, if $h\in L^{2}(\mathbb{S}^{2}; H^{m}(\mathbb{T}^{3}))$ and $F\in L^{2}(0, T_{0}; L^{2}(\mathbb{S}^{2}; \\H^{m}(\mathbb{T}^{3})))$ the initial value problem of \eqref{research equation} has a unique solution $I(t, \vec{x}, \vec{w})\in C([0, T_{0}]; L^{2}(\mathbb{S}^{2}; \\H^{m}(\mathbb{T}^{3})))$.
}

The proof of the above lemmas is similar as in \cite{yuanmoxing-bu-jia3jia}.

\section{Hilbert Expansion}\label{san}

\subsection{Interior expansion}
We write the equivalent form of \eqref{research equations} as
\begin{equation}\label{research equations-0}\left\{
\begin{split}
&\partial_{t}\rho^{\epsilon}+\vec{u}^{\epsilon}\cdot\nabla_{x}\rho^{\epsilon}+\rho^{\epsilon}\mathrm{div}_{x}\vec{u}^{\epsilon}=0,~~\mathrm{in} \ \ (0,T)\times\mathbb{T}^{3},\\
&\rho^{\epsilon}\partial_{t}\vec{u}^{\epsilon}+\rho^{\epsilon}\vec{u}^{\epsilon}\cdot\nabla_{x}\vec{u}^{\epsilon}+\rho^{\epsilon}\nabla_{x}\theta^{\epsilon}
+\theta^{\epsilon}\nabla_{x}\rho^{\epsilon}
=\Big\langle\Big(\frac{1}{\epsilon}+\epsilon\Big)\vec{w}(f^{\epsilon}-\overline {f^{\epsilon}})\Big\rangle,~~\mathrm{in} \ \ (0,T)\times\mathbb{T}^{3},\\
&\rho^{\epsilon}\partial_{t}\theta^{\epsilon}+\rho^{\epsilon}\vec{u}^{\epsilon}\cdot\nabla_{x}\theta^{\epsilon}+\rho^{\epsilon}\theta^{\epsilon}\mathrm{div}_{x}\vec{u}^{\epsilon}
=\frac{1}{\epsilon^{2}}(\overline {f^{\epsilon}}-B(\theta^{\epsilon}))\\&\hspace{6.2cm}-\Big(\frac{1}{\epsilon}+\epsilon\Big)\langle\vec{w}(f^{\epsilon}-\overline {f^{\epsilon}})\rangle\cdot\vec{u}^{\epsilon},~~\mathrm{in} \ \ (0,T)\times\mathbb{T}^{3}, \\
&\partial_{t}f^{\epsilon}+\frac{1}{\epsilon}\vec w\cdot \nabla_{x}f^{\epsilon}+f^{\epsilon}-\overline {f^{\epsilon}}=\frac{1}{\epsilon^{2}}(B(\theta^{\epsilon})-f^{\epsilon}),~~\mathrm{in} \ \ (0,T)\times\mathbb{T}^{3}\times \mathbb{S}^{2},\\
&\rho^{\epsilon}(0, \vec{x})=\rho^{0}(\vec x),\;\;\vec{u}^{\epsilon}(0, \vec{x})=\vec{u}^{0}(\vec x),\;\;\theta^{\epsilon}(0, \vec{x})=\theta^{0}(\vec{x}),\;\;f^{\epsilon}(0, \vec{x}, \vec{w})=h(\vec{x}, \vec{w})~~\;\mathrm{in}  \ \ \mathbb{T}^{3}\times\mathbb{S}^{2}.\\
\end{split}\right.
\end{equation}

First, we give the interior expansion of $\theta^{\epsilon}$ as follows:
\begin{equation}\label{expansion of theta epsilon}
\rho^{\epsilon}\sim\sum_{k=0}^{N}\epsilon^{k}\rho_{k}, ~~\vec u^{\epsilon}\sim\sum_{k=0}^{N}\epsilon^{k}\vec{u}_{k}, ~~\theta^{\epsilon}\sim\sum_{k=0}^{N}\epsilon^{k}\theta_{k},~~f^{\epsilon}\sim \sum_{k=0}^{N}{\epsilon}^{k}f_{k},
\end{equation}
where $\rho_{k}=\rho_{k}(t, \vec x), \vec{u}_{k}=\vec{u}_{k}(t, \vec x), \theta_{k}=\theta_{k}(t, \vec x)$ and $f_{k}=f_{k}(t, \vec x, \vec w)$.
Then, we get
\begin{equation}\label{expansion of B(theta epsilon)}
\begin{split}
B(\theta^{\epsilon})&\sim(\theta_{0}+\epsilon\theta_{1}+\cdot\cdot\cdot)^{4}
\\&=B_{0}(\theta_{0})+\epsilon B_{1}(\theta_{0}; \theta_{1})+\epsilon^{2}B_{2}(\theta_{0}; \theta_{2})+\cdot\cdot\cdot
\end{split}
\end{equation}
where 
\begin{equation}\label{defBk}
B_{k}(\theta_{0}; \theta_{k}):=\sum_{\substack{
  i+j+l+m=k, \\
  i,j,l,m\geq 0
}}\theta_{i}\theta_{j}\theta_{l}\theta_{m}.
\end{equation}

Define
\begin{equation}\label{suanzi1}
\mathcal{L}_{1}=\mathcal{L}_{1}(\rho^{\epsilon}, \vec{u}^{\epsilon}):=\partial_{t}\rho^{\epsilon}+\vec{u}^{\epsilon}\cdot\nabla_{x}\rho^{\epsilon}+\rho^{\epsilon}\mathrm{div}_{x}\vec{u}^{\epsilon},
\end{equation}

\begin{equation}\label{suanzi2}
\begin{aligned}
\mathcal{L}_{2}=&\mathcal{L}_{2}(\rho^{\epsilon}, \vec{u}^{\epsilon}, \theta^{\epsilon}, f^{\epsilon})\\:=&\epsilon^{2}\rho^{\epsilon}\partial_{t}\vec{u}^{\epsilon}+\epsilon^{2}\rho^{\epsilon}\vec{u}^{\epsilon}\cdot\nabla_{x}\vec{u}^{\epsilon}+\epsilon^{2}\rho^{\epsilon}\nabla_{x}\theta^{\epsilon}
+\epsilon^{2}\theta^{\epsilon}\nabla_{x}\rho^{\epsilon}
-\Big\langle\Big(\epsilon+\epsilon^{3}\Big)\vec{w}(f^{\epsilon}-\overline {f^{\epsilon}})\Big\rangle\\=&\epsilon^{2}\rho^{\epsilon}\partial_{t}\vec{u}^{\epsilon}+\epsilon^{2}\rho^{\epsilon}\vec{u}^{\epsilon}\cdot\nabla_{x}\vec{u}^{\epsilon}+\epsilon^{2}\rho^{\epsilon}\nabla_{x}\theta^{\epsilon}
+\epsilon^{2}\theta^{\epsilon}\nabla_{x}\rho^{\epsilon}
-\Big\langle\Big(\epsilon+\epsilon^{3}\Big)\vec{w}f^{\epsilon}\Big\rangle,
\end{aligned}
\end{equation}

\begin{equation}\label{suanzi3}
\begin{aligned}
\mathcal{L}_{3}=&\mathcal{L}_{3}(\rho^{\epsilon}, \vec{u}^{\epsilon}, \theta^{\epsilon}, f^{\epsilon})\\:=&\epsilon^{2}\rho^{\epsilon}\partial_{t}\theta^{\epsilon}+\epsilon^{2}\rho^{\epsilon}\vec{u}^{\epsilon}\cdot\nabla_{x}\theta^{\epsilon}+\epsilon^{2}\rho^{\epsilon}\theta^{\epsilon}\mathrm{div}_{x}\vec{u}^{\epsilon}
-(\overline {f^{\epsilon}}-B(\theta^{\epsilon}))+\Big(\epsilon+\epsilon^{3}\Big)\langle\vec{w}(f^{\epsilon}-\overline {f^{\epsilon}})\rangle\cdot\vec{u}^{\epsilon}
\\=&\epsilon^{2}\rho^{\epsilon}\partial_{t}\theta^{\epsilon}+\epsilon^{2}\rho^{\epsilon}\vec{u}^{\epsilon}\cdot\nabla_{x}\theta^{\epsilon}+\epsilon^{2}\rho^{\epsilon}\theta^{\epsilon}\mathrm{div}_{x}\vec{u}^{\epsilon}
-(\overline {f^{\epsilon}}-B(\theta^{\epsilon}))+\Big(\epsilon+\epsilon^{3}\Big)\langle\vec{w}f^{\epsilon}\rangle\cdot\vec{u}^{\epsilon}
\end{aligned}
\end{equation}
and
\begin{equation}\label{suanzi4}
\mathcal{L}_{4}=\mathcal{L}_{4}(f^{\epsilon}, \theta^{\epsilon}):=\epsilon^{2}\partial_{t}f^{\epsilon}+\epsilon \vec w\cdot \nabla_{x}f^{\epsilon}+{\epsilon}^{2}(f^{\epsilon}-\overline {f^{\epsilon}})+f^{\epsilon}-B(\theta^{\epsilon}).
\end{equation}
Plugging \eqref{expansion of theta epsilon} into the above formulas gives
\begin{equation}\label{resi1}
\begin{aligned}
&\mathcal{L}_{1}\Big(\sum_{k=0}^{N}{\epsilon}^{k}\rho_{k}, \displaystyle\sum_{k=0}^{N}\epsilon^{k}\vec{u}_{k}\Big)=\displaystyle\sum_{k=0}^{N}{\epsilon}^{k}(\partial_{t}\rho_{k}+\vec{u}_{k}\cdot\nabla_{x}\rho_{0}+\vec{u}_{0}\cdot\nabla_{x}\rho_{k}+\rho_{0}\mathrm{div}\vec{u}_{k}
+\rho_{k}\mathrm{div}\vec{u}_{0}\\&+\sum_{\substack{
  i+j=k, \\
  i,j\geq 1
}}\vec{u}_{i}\cdot\nabla_{x}\rho_{j}
+\displaystyle\sum_{\substack{
  i+j=k, \\
  i,j\geq 1
}}\rho_{j}\mathrm{div}\vec{u}_{i})+\displaystyle\sum_{k=N+1}^{2N}\epsilon^{k}(\sum_{\substack{
  i+j=k, \\
  0\leq i,j\leq N
}}\vec{u}_{i}\cdot\nabla_{x}\rho_{j}
+\displaystyle\sum_{\substack{
  i+j=k, \\
  0\leq i,j\leq N
}}\rho_{j}\mathrm{div}\vec{u}_{i}),
\end{aligned}
\end{equation}

\begin{equation}\label{resi2}
\begin{aligned}
&\mathcal{L}_{2}\Big(\sum_{k=0}^{N}{\epsilon}^{k}\rho_{k}, \sum_{k=0}^{N}\epsilon^{k}\vec{u}_{k}, \sum_{k=0}^{N}\epsilon^{k}\theta_{k}, \sum_{k=0}^{N}{\epsilon}^{k}f_{k}\Big)\\=&\sum_{k=0}^{N}{\epsilon}^{k}(\rho_{0}\partial_{t}\vec{u}_{k-2}+\rho_{k-2}\partial_{t}\vec{u}_{0}+\rho_{0}\vec{u}_{0}\cdot\nabla_{x}\vec{u}_{k-2}
+\rho_{0}\vec{u}_{k-2}\cdot\nabla_{x}\vec{u}_{0}+\rho_{k-2}\vec{u}_{0}\cdot\nabla_{x}\vec{u}_{0}+\rho_{0}\nabla_{x}\theta_{k-2}
\\&+\rho_{k-2}\nabla_{x}\theta_{0}+\theta_{0}\nabla_{x}\rho_{k-2}+\theta_{k-2}\nabla_{x}\rho_{0}
-\langle\vec w(f_{k-1}+{f_{k-3}})\rangle+\sum_{\substack{
  i+j=k-2, \\
  i,j\geq 1
}}(\rho_{i}\partial_{t}\vec{u}_{j})\\&+\sum_{\substack{
  i+j+l=k-2, \\
  0\leq i,j,l\leq k-3
}}\rho_{i}\vec{u}_{j}\cdot\nabla_{x}\vec{u}_{l}+\sum_{\substack{
  i+j=k-2, \\
  0\leq i,j\leq k-3
}}(\rho_{i}\nabla_{x}\theta_{j}+\theta_{j}\nabla_{x}\rho_{i}))
+\sum_{k=N+1}^{2N+2}\epsilon^{k}\Big(\sum_{\substack{
  i+j=k-2, \\
  0\leq i,j\leq N
}}(\rho_{i}\partial_{t}\vec{u}_{j})\\&+\sum_{\substack{
  i+j+l=k-2, \\
  0\leq i,j, l\leq N
}}(\rho_{i}\vec{u}_{j}\cdot\nabla_{x}\vec{u}_{l})
+\sum_{\substack{
  i+j=k-2, \\
  0\leq i,j\leq N
}}(\rho_{i}\nabla_{x}\theta_{j}+\theta_{j}\nabla_{x}\rho_{i})\Big)-\epsilon^{N+1}\langle\vec w f_{N}\rangle-\epsilon^{N+1}\langle\vec w f_{N-2}\rangle\\&-\epsilon^{N+2}\langle\vec w f_{N-1}\rangle-\epsilon^{N+3}\langle\vec w f_{N}\rangle,
\end{aligned}
\end{equation}

\begin{equation}\label{resi3}
\begin{aligned}
&\mathcal{L}_{3}\Big(\sum_{k=0}^{N}{\epsilon}^{k}\rho_{k}, \sum_{k=0}^{N}\epsilon^{k}\vec{u}_{k}, \sum_{k=0}^{N}\epsilon^{k}\theta_{k}, \sum_{k=0}^{N}{\epsilon}^{k}f_{k}\Big)\\=&\sum_{k=0}^{N}{\epsilon}^{k}\Big(\rho_{0}\partial_{t}\theta_{k-2}+\rho_{k-2}\partial_{t}\theta_{0}+\rho_{0}\vec{u}_{0}\cdot\nabla_{x}\theta_{k-2}
+\rho_{0}\vec{u}_{k-2}\cdot\nabla_{x}\theta_{0}+\rho_{0}\theta_{0}\mathrm{div}_{x}\vec{u}_{k-2}+\rho_{k-2}\theta_{0}\mathrm{div}_{x}\vec{u}_{0}\\&+\rho_{0}\theta_{k-2}\mathrm{div}_{x}\vec{u}_{0}
-(\overline{f_{k}}-B_{k}(\theta_{0}; \theta_{k}))+\langle\vec w(f_{k-1}+{f_{k-3}})\rangle\cdot\vec{u}_{0}
+\sum_{\substack{
  i+j=k-1, \\
  j\geq 1
}}\langle \vec w f_{i}\rangle\vec{u}_{j}+\langle \vec w f_{0}\rangle\cdot\vec{u}_{k-1}
\end{aligned}
\end{equation}
\begin{equation}\nonumber
\begin{aligned}
\\&+\sum_{\substack{
  i+j=k-3, \\
  j\geq 1
}}\langle\vec w f_{i}\rangle\cdot\vec{u}_{j}\Big)
+\sum_{k=N+1}^{2N+2}\epsilon^{k}\sum_{\substack{
  i+j=k-2, \\
  i, j\geq 1
}}(\rho_{i}\partial_{t}\theta_{j}+(\partial_{t}\rho_{i})\theta_{j})+\sum_{k=N+1}^{3N}\epsilon^{k}\sum_{\substack{
  i+j+l=k-2, \\
 0\leq i,j,l\leq k-3
}}(\rho_{i}\vec{u}_{j}\cdot\nabla_{x}\theta_{l}\\&+\rho_{i}\theta_{l}\mathrm{div}_{x}\vec{u}_{j})+\sum_{k=N+1}^{2N+1}\epsilon^{k}\sum_{\substack{
  i+j=k-1, \\
  0\leq i, j\geq N
}}\langle\vec{w}f_{i}\rangle\cdot\vec{u}_{j}
+\sum_{k=N+1}^{2N+3}\epsilon^{k}\sum_{\substack{
  i+j=k-3, \\
  0\leq i, j\geq N
}}\langle\vec{w}f_{i}\rangle\cdot\vec{u}_{j}
\end{aligned}
\end{equation}
and
\begin{equation}\label{resi4}
\begin{aligned}
&\mathcal{L}_{4}\Big(\sum_{k=0}^{N}{\epsilon}^{k}f_{k}, \sum_{k=0}^{N}\epsilon^{k}\theta_{k}\Big)=\sum_{k=0}^{N}{\epsilon}^{k}(\partial_{t}f_{k-2}+\vec{w}\cdot\nabla_{x}f_{k-1}
+f_{k-2}-\overline{f_{k-2}}+f_{k}-B_{k}
\\&+\epsilon^{N+1}\partial_{t}f_{N-1}+\epsilon^{N+2}\partial_{t}f_{N}+\epsilon^{N+1}\vec{w}\cdot\nabla_{x}f_{N}
-\sum_{k=N+1}^{4N}\epsilon^{k}\sum_{\substack{
  i+j+l+m=k, \\
  0\leq i,j,l,m\leq N
}}\theta_{i}\theta_{j}\theta_{l}\theta_{m},
\end{aligned}
\end{equation}
where $(\rho_{k}, \vec{u}_{k}, \theta_{k}, f_{k})$, $k<0$ are taken to be zero.

Upon collecting terms of the same order, we arrive at
\begin{equation}\label{rhok}
\begin{aligned}
&\partial_{t}\rho_{k}+\vec{u}_{k}\cdot\nabla_{x}\rho_{0}+\vec{u}_{0}\cdot\nabla_{x}\rho_{k}+\rho_{0}\mathrm{div}\vec{u}_{k}
+\rho_{k}\mathrm{div}\vec{u}_{0}+\sum_{\substack{
  i+j=k, \\
  i,j\geq 1
}}(\vec{u}_{i}\cdot\nabla_{x}\rho_{j}
+\rho_{j}\mathrm{div}\vec{u}_{i})=0,
\end{aligned}
\end{equation}

\begin{equation}\label{uk}
\begin{aligned}
&\rho_{0}\partial_{t}\vec{u}_{k-2}+\rho_{k-2}\partial_{t}\vec{u}_{0}+\rho_{0}\vec{u}_{0}\cdot\nabla_{x}\vec{u}_{k-2}
+\rho_{0}\vec{u}_{k-2}\cdot\nabla_{x}\vec{u}_{0}+\rho_{0}\nabla_{x}\theta_{k-2}+\rho_{k-2}\nabla_{x}\theta_{0}\\&+\theta_{0}\nabla_{x}\rho_{k-2}+\theta_{k-2}\nabla_{x}\rho_{0}
-\langle\vec w(f_{k-1}+{f_{k-3}})\rangle+\sum_{\substack{
  i+j=k-2, \\
  i,j\geq 1
}}(\rho_{i}\partial_{t}\vec{u}_{j})+\sum_{\substack{
  i+j+l=k-2, \\
 0\leq i,j,l\leq k-3
}}\rho_{i}\vec{u}_{j}\cdot\nabla_{x}\vec{u}_{l}\\&+\sum_{\substack{
  i+j=k-2, \\
  1\leq i,j\leq k-3
}}(\theta_{i}\nabla_{x}\rho_{j}+\rho_{j}\nabla_{x}\theta_{i})=0,
\end{aligned}
\end{equation}

\begin{equation}\label{thetak}
\begin{aligned}
&\rho_{0}\partial_{t}\theta_{k-2}+\rho_{k-2}\partial_{t}\theta_{0}+\rho_{0}\vec{u}_{0}\cdot\nabla_{x}\theta_{k-2}
+\rho_{0}\vec{u}_{k-2}\cdot\nabla_{x}\theta_{0}+\rho_{k-2}\vec{u}_{0}\cdot\nabla_{x}\theta_{0}+\rho_{0}\theta_{0}\mathrm{div}_{x}\vec{u}_{k-2}\\&+\rho_{k-2}\theta_{0}\mathrm{div}_{x}\vec{u}_{0}+\rho_{0}\theta_{k-2}\mathrm{div}_{x}\vec{u}_{0}
-(\overline{f_{k}}-B_{k}(\theta_{0}; \theta_{k}))+\langle\vec w(f_{k-1}+{f_{k-3}})\rangle\cdot\vec{u}_{0}+\langle \vec w f_{1}\rangle\cdot\vec{u}_{k-2}
\\&+\sum_{\substack{
  i+j=k-1, \\
  1\leq j\leq k-2
}}\langle \vec w f_{i}\rangle\cdot\vec{u}_{j}+\langle \vec w f_{0}\rangle\cdot\vec{u}_{k-1}-\sum_{\substack{
  i+j=k-3, \\
  j\geq 1
}}\langle\vec w f_{i}\rangle\cdot\vec{u}_{j}=0
\end{aligned}
\end{equation}
and
\begin{equation}\label{fk}
\begin{aligned}
\partial_{t}f_{k-2}+\vec{w}\cdot\nabla_{x}f_{k-1}
+f_{k-2}-\overline{f_{k-2}}+f_{k}-B_{k}(\theta_{0}; \theta_{k})=0.
\end{aligned}
\end{equation}
For $k\geq 1$, we can rewrite \eqref{rhok} as 
\begin{equation}\label{rhok1}
\begin{aligned}
&\partial_{t}\rho_{k}+\vec{u}_{k}\cdot\nabla_{x}\rho_{0}+\vec{u}_{0}\cdot\nabla_{x}\rho_{k}+\rho_{0}\mathrm{div}\vec{u}_{k}
+\rho_{k}\mathrm{div}\vec{u}_{0}=F_{k-1}^{1},
\end{aligned}
\end{equation}
where
\begin{equation}\label{Fk1}
F_{k-1}^{1}=-\sum_{\substack{
  i+j=k, \\
  i,j\geq 1
}}(\vec{u}_{i}\cdot\nabla_{x}\rho_{j}
+\rho_{j}\mathrm{div}\vec{u}_{i}).
\end{equation}
Combining \eqref{uk} and \eqref{fk}, for $k\geq 1$, we rewrite \eqref{uk} as
\begin{equation}\label{uk-1}
\begin{aligned}
&\rho_{0}\partial_{t}\vec{u}_{k}+\rho_{0}\vec{u}_{0}\cdot\nabla_{x}\vec{u}_{k}
+\rho_{0}\vec{u}_{k}\cdot\nabla_{x}\vec{u}_{0}+\Big(\rho_{0}+\frac{1}{3}\theta_{0}^{3}\Big)\nabla_{x}\theta_{k}+\rho_{k}\nabla_{x}\theta_{0}\\&+\theta_{0}\nabla_{x}\rho_{k}
+\theta_{k}\Big(\nabla_{x}\rho_{0}+\frac{1}{3}\nabla_{x}\theta_{0}^{3}\Big)
+\rho_{k}\partial_{t}\vec{u}_{0}=F_{k-1}^{2},
\end{aligned}
\end{equation}
where
\begin{equation}\nonumber
\begin{aligned}
F_{k-1}^{2}=&-\frac{1}{3}\nabla_{x}\overline{f_{k-2}}+\langle(\vec{w}\otimes\vec{w})\cdot\nabla_{x}(f_{k-2}+\partial_{t}f_{k-2}+\vec{w}\cdot\nabla_{x}f_{k-1})\rangle
+\langle\vec{w}\cdot\nabla_{x}f_{k-1}\rangle\\&-\frac{1}{3}\nabla_{x}\sum_{\substack{
  i+j+l+m=k, \\
  0\leq i,j,l,m\leq k-1
}}(\theta_{i}\theta_{j}\theta_{l}\theta_{m})-\sum_{\substack{
  i+j=k, \\
  i,j\geq 1
}}(\rho_{i}\partial_{t}\vec{u}_{j})-\sum_{\substack{
  i+j+l=k, \\
  0\leq i,j,l\leq k-1
}}\rho_{i}\vec{u}_{j}\cdot\nabla_{x}\vec{u}_{l}\\&-\sum_{\substack{
  i+j=k, \\
 i,j\geq 1
}}(\theta_{i}\nabla_{x}\rho_{j}+\rho_{j}\nabla_{x}\theta_{i}).
\end{aligned}
\end{equation}
Collecting \eqref{thetak} and \eqref{fk}, for $k\geq 1$, we can rewrite \eqref{thetak} as follows
\begin{equation}\label{thetak-1}
\begin{aligned}
&(\rho_{0}+\theta_{0}^{3})\partial_{t}\theta_{k}+\theta_{k}\partial_{t}(\theta_{0}^{3})+\rho_{k}\partial_{t}\theta_{0}+\rho_{0}\vec{u}_{0}\cdot\nabla_{x}\theta_{k}
+\rho_{0}\vec{u}_{k}\cdot\nabla_{x}\theta_{0}+\rho_{k}\vec{u}_{0}\cdot\nabla_{x}\theta_{0}+\rho_{0}\theta_{0}\mathrm{div}_{x}\vec{u}_{k}\\&+\rho_{k}\theta_{0}\mathrm{div}_{x}\vec{u}_{0}+\rho_{0}\theta_{k}\mathrm{div}_{x}\vec{u}_{0}
-4\theta_{0}^{3}\Delta_{x}\theta_{k}-12\theta_{0}^{2}\nabla_{x}\theta_{0}\cdot\nabla_{x}\theta_{k}-4\theta_{k}\Delta_{x}\theta_{0}^{3}-\frac{(\theta_{0})^{3}}{3}\nabla_{x}\theta_{k}\cdot\vec{u}_{0}
\\&-\theta_{k}\theta_{0}^{2}\nabla_{x}\theta_{0}\cdot\vec{u}_{0}+\frac{4\theta_{0}^{3}}{3}\nabla_{x}\theta_{0}\cdot\vec{u}_{k}=F_{k-1}^{3},
\end{aligned}
\end{equation}
where
\begin{equation}\label{Fk-1biaoda}
\begin{aligned}
F_{k-1}^{3}=&-\sum_{\substack{
  i+j+l+m=k, \\
  0\leq i,j,l,m\leq k-1
}}\partial_{t}(\theta_{i}\theta_{j}\theta_{l}\theta_{m})-\partial_{t}^{2}\overline{f_{k-2}}-\langle\vec{w}\cdot\nabla_{x}\partial_{t}f_{k-1}\rangle-\langle(\vec w\cdot\nabla_{x})^{2}(\partial_{t}f_{k-2}\\&+(\vec w\cdot\nabla_{x})f_{k-1}-\overline{f_{k-2}}+f_{k-2})\rangle+\sum_{\substack{
  i+j+l+m=k, \\
  0\leq i,j,l,m\leq k-1
}}\frac{1}{3}\Delta_{x}(\theta_{i}\theta_{j}\theta_{l}\theta_{m})
\\&+\sum_{\substack{
  i+j=k+1, \\
  j\geq 1
}}\langle\vec w (\partial_{t}f_{i-2}+\vec w\cdot\nabla_{x}f_{i-1}-\overline {f_{i-2}}+f_{i-2})\rangle\cdot\vec{u}_{j}+\sum_{\substack{
  i+j=k-1, \\
  j\geq 1
}}\langle\vec w f_{i}\rangle\cdot\vec{u}_{j}
\\&+\langle\vec w\partial_{t}f_{k-1}\rangle\cdot\vec{u}_{0}+\langle\vec w f_{k-1}\rangle\cdot\vec{u}_{0}-\Big(-\frac{1}{3}\nabla_{x}\overline{f_{k-2}}+\langle(\vec{w}\otimes\vec{w})\cdot\nabla_{x}(f_{k-2}+\partial_{t}f_{k-2}
\\&+\vec{w}\cdot\nabla_{x}f_{k-1})\rangle
+\langle\vec{w}\cdot\nabla_{x}f_{k-1}\rangle-\frac{1}{3}\nabla_{x}\sum_{\substack{
  i+j+l+m=k, \\
  0\leq i,j,l,m\leq k-1
}}(\theta_{i}\theta_{j}\theta_{l}\theta_{m})\Big)\cdot\vec{u}_{0}.
\end{aligned}
\end{equation}
For $k=0$, we have
\begin{equation}\label{zero-ini-sys}\left\{
\begin{split}
&f_{0}=\overline {f_{0}}=\theta_{0}^{4},\\
&\partial_{t}\rho_{0}+\mathrm{div}(\rho_{0}\vec u_{0})=0,\\
&\partial_{t}(\rho_{0}\vec{u}_{0})+\mathrm{div}_{x}(\rho_{0}\vec{u}_{0}\otimes\vec u_{0})+\nabla_{x}(\rho_{0}\theta_{0})
=-\frac{1}{3}\nabla_{x}\theta_{0}^{4},\\
&\partial_{t}(\rho_{0}\theta_{0}+\theta_{0}^{4})+\mathrm{div}_{x}(\rho_{0}\theta_{0}\vec{u}_{0})+\rho_{0}\theta_{0}\mathrm{div}_{x}\vec{u}_{0}
-\frac{1}{3}\Delta_{x}\theta_{0}^{4}=\frac{1}{3}\nabla_{x}\theta_{0}^{4}\cdot\vec u_{0},
\end{split}\right.
\end{equation}
which is the equivalent form of $\eqref{equations about f0 epsilon and theta-0}_{1-4}$.

In view of the above equations, the residuals \eqref{resi1}-\eqref{resi4} take the form
\begin{equation}\label{resi1-jia}
\begin{aligned}
\mathcal{L}_{1}\Big(\sum_{k=0}^{N}{\epsilon}^{k}\rho_{k}, \sum_{k=0}^{N}\epsilon^{k}\vec{u}_{k}\Big)=&\sum_{k=N+1}^{2N}\epsilon^{k}\Big(\sum_{\substack{
  i+j=k, \\
  0\leq i,j\leq N
}}\vec{u}_{i}\cdot\nabla_{x}\rho_{j}
+\sum_{\substack{
  i+j=k, \\
  0\leq i,j\leq N
}}\rho_{j}\mathrm{div}\vec{u}_{i}\Big),
\end{aligned}
\end{equation}

\begin{equation}\label{resi2-jia}
\begin{aligned}
&\mathcal{L}_{2}\Big(\sum_{k=0}^{N}{\epsilon}^{k}\rho_{k}, \sum_{k=0}^{N}\epsilon^{k}\vec{u}_{k}, \sum_{k=0}^{N}\epsilon^{k}\theta_{k}, \sum_{k=0}^{N}{\epsilon}^{k}f_{k}\Big)\\=&\sum_{k=N+1}^{2N+2}\epsilon^{k}\Big(\sum_{\substack{
  i+j=k-2, \\
  0\leq i,j\leq N
}}(\rho_{i}\partial_{t}\vec{u}_{j})+\sum_{\substack{
  i+j+l=k-2, \\
  0\leq i,j, l\leq N
}}(\rho_{i}\vec{u}_{j}\cdot\nabla_{x}\vec{u}_{l})+\sum_{\substack{
  i+j=k-2, \\
  0\leq i,j\leq N
}}(\rho_{i}\nabla_{x}\theta_{j}\\&+\theta_{j}\nabla_{x}\rho_{i})\Big)-\epsilon^{N+1}\langle\vec w f_{N}\rangle-\epsilon^{N+1}\langle\vec w f_{N-2}\rangle-\epsilon^{N+2}\langle\vec w f_{N-1}\rangle-\epsilon^{N+3}\langle\vec w f_{N}\rangle,
\end{aligned}
\end{equation}

\begin{equation}\label{resi3-jia}
\begin{aligned}
&\mathcal{L}_{3}\Big(\sum_{k=0}^{N}{\epsilon}^{k}\rho_{k}, \sum_{k=0}^{N}\epsilon^{k}\vec{u}_{k}, \sum_{k=0}^{N}\epsilon^{k}\theta_{k}, \sum_{k=0}^{N}{\epsilon}^{k}f_{k}\Big)\\=&\sum_{k=N+1}^{2N+2}\epsilon^{k}\sum_{\substack{
  i+j=k-2, \\
  i, j\geq 1
}}(\rho_{i}\partial_{t}\theta_{j}+\partial_{t}\rho_{i}\theta_{j})+\sum_{k=N+1}^{3N}\epsilon^{k}\sum_{\substack{
  i+j+l=k-2, \\
 0\leq i,j,l\leq k-3
}}(\rho_{i}\vec{u}_{j}\cdot\nabla_{x}\theta_{l}+\rho_{i}\theta_{l}\mathrm{div}_{x}\vec{u}_{j})
\\&+\sum_{k=N+1}^{2N+1}\epsilon^{k}\sum_{\substack{
  i+j=k-1, \\
  0\leq i, j\geq N
}}\langle\vec{w}f_{i}\rangle\cdot\vec{u}_{j}+\sum_{k=N+1}^{2N+3}\epsilon^{k}\sum_{\substack{
  i+j=k-3, \\
  0\leq i, j\geq N
}}\langle\vec{w}f_{i}\rangle\cdot\vec{u}_{j}
\end{aligned}
\end{equation}
and
\begin{equation}\label{resi4-jia}
\begin{aligned}
\mathcal{L}_{4}\Big(\sum_{k=0}^{N}{\epsilon}^{k}f_{k}, \sum_{k=0}^{N}\epsilon^{k}\theta_{k}\Big)=&\epsilon^{N+1}\partial_{t}f_{N-1}+\epsilon^{N+2}\partial_{t}f_{N}+\epsilon^{N+1}\vec{w}\cdot\nabla_{x}f_{N}
\\&-\sum_{k=N+1}^{4N}\epsilon^{k}\sum_{\substack{
  i+j+l+m=k, \\
  0\leq i,j,l,m\leq N
}}\theta_{i}\theta_{j}\theta_{l}\theta_{m},
\end{aligned}
\end{equation}
where the right hand side are both formally of order $\epsilon^{N+1}$.
\subsection{Initial layer expansion.}
To determine the initial conditions for $\overline {f_{k}}$ and $V_{k}$, we employ the initial-layer expansion. To this end, we introduce a new variable via the scaling transform  $(f^{\epsilon}(t, \vec{x}, \vec w), \rho^{\epsilon}(t, \vec x), \vec{u}^{\epsilon}(t, \vec x), \\\theta^{\epsilon}(t, \vec x))\rightarrow (f^{\epsilon}(\tau, \vec{x}, \vec w), \rho^{\epsilon}(\tau, \vec x), \vec{u}^{\epsilon}(\tau, \vec x), \theta^{\epsilon}(\tau, \vec x))$ with $\tau$ as
$\tau=\frac{t}{\epsilon^{2}},$ where $\tau\in [0, \infty)$. This gives
\begin{equation}\nonumber
\frac{\partial f^{\epsilon}}{\partial t}=\frac{1}{\epsilon^{2}}\frac{\partial f^{\epsilon}}{\partial\tau}, \frac{\partial \rho^{\epsilon}}{\partial t}=\frac{1}{\epsilon^{2}}\frac{\partial \rho^{\epsilon}}{\partial\tau}, \frac{\partial \vec{u}^{\epsilon}}{\partial t}=\frac{1}{\epsilon^{2}}\frac{\partial \vec{u}^{\epsilon}}{\partial\tau}, \frac{\partial \theta^{\epsilon}}{\partial t}=\frac{1}{\epsilon^{2}}\frac{\partial \theta^{\epsilon}}{\partial\tau}.
\end{equation}

Under this change of variables, the system \eqref{research equations} is transformed into
\begin{equation}\label{research equation in the initial layer}\left\{
\begin{split}
&\frac{1}{\epsilon^{2}}\frac{\partial \rho^{\epsilon}}{\partial\tau}+\vec{u}^{\epsilon}\cdot\nabla_{x}\rho^{\epsilon}+\rho^{\epsilon}\mathrm{div}_{x}\vec{u}^{\epsilon}=0,\\
&\frac{\rho^{\epsilon}}{\epsilon^{2}}\frac{\partial \vec u^{\epsilon}}{\partial\tau}+\rho^{\epsilon}\vec{u}^{\epsilon}\cdot\nabla_{x}\vec{u}^{\epsilon}+\rho^{\epsilon}\nabla_{x}\theta^{\epsilon}
+\theta^{\epsilon}\nabla_{x}\rho^{\epsilon}=\Big\langle\Big(\frac{1}{\epsilon}+\epsilon\Big)\vec{w}(f^{\epsilon}-\overline {f^{\epsilon}})\Big\rangle,\\
&\frac{\rho^{\epsilon}}{\epsilon^{2}}\frac{\partial \theta^{\epsilon}}{\partial\tau}+\rho^{\epsilon}\vec{u}^{\epsilon}\cdot\nabla_{x}\theta^{\epsilon}+\rho^{\epsilon}\theta^{\epsilon}\mathrm{div}_{x}\vec{u}^{\epsilon}
=\frac{1}{\epsilon^{2}}(\overline {f^{\epsilon}}-B(\theta^{\epsilon}))-\Big(\frac{1}{\epsilon}+\epsilon\Big)\langle\vec{w}(f^{\epsilon}-\overline {f^{\epsilon}})\rangle\cdot\vec u^{\epsilon},\\
&\frac{\partial f^{\epsilon}}{\partial\tau}+\epsilon \vec w\cdot \nabla_{x}f^{\epsilon}+\epsilon^{2}(f^{\epsilon}-\overline {f^{\epsilon}})+f^{\epsilon}=B(\theta^{\epsilon}).\\
\end{split}\right.
\end{equation}

We define the initial layer expansion as follows:
\begin{equation}\label{expansion about the initial layer}
f_{I}^{\epsilon}\thicksim \sum_{k=0}^{N}\epsilon^{k}f_{I, k}(\tau, \vec{x}, \vec{w}),\;\;\rho_{I}^{\epsilon}\thicksim\sum_{k=0}^{N}\epsilon^{k}\rho_{I, k}(\tau, \vec{x}),\;\;\vec u_{I}^{\epsilon}\thicksim\sum_{k=0}^{N}\epsilon^{k}\vec u_{I, k}(\tau, \vec{x}),\;\;
\theta_{I}^{\epsilon}\thicksim\sum_{k=0}^{N}\epsilon^{k}\theta_{I, k}(\tau, \vec{x}).
\end{equation}
Plugging the expansions $f^{\epsilon}\thicksim\sum_{k=0}^{N}\epsilon^{k}(f_{k}+f_{I, k})$, $\rho^{\epsilon}\thicksim\sum_{k=0}^{N}\epsilon^{k}(\rho_{k}+\rho_{I, k})$, $\vec{u}^{\epsilon}\thicksim\sum_{k=0}^{N}\epsilon^{k}(\vec{u}_{k}+\vec{u}_{I, k})$ and $\theta^{\epsilon}\thicksim\sum_{k=0}^{N}\epsilon^{k}(\theta_{k}+\theta_{I, k})$ into \eqref{research equation in the initial layer} and comparing $\epsilon$-orders gives $\rho_{I, k}, \vec u_{I, k}, \theta_{I, k}$ and $f_{I, k}$. We thus have
\begin{equation}\label{resi1-jia-ini}
\begin{aligned}
&\mathcal{L}_{1}\Big(\sum_{k=0}^{N}\epsilon^{k}(\rho_{k}+\rho_{I, k}), \sum_{k=0}^{N}\epsilon^{k}(\vec{u}_{k}+\vec{u}_{I, k})\Big)\\=&\mathcal{L}_{1}\Big(\sum_{k=0}^{N}\epsilon^{k}\rho_{k}, \sum_{k=0}^{N}\epsilon^{k}\vec{u}_{k}\Big)+\sum_{k=0}^{N}\epsilon^{k}(\partial_{\tau}\rho_{I, k}+\sum_{\substack{
  i+j=k-2, \\
  0\leq i,j\leq k-2
}}\{(\vec{E}_{i}+\vec{u}_{I, i})\cdot\nabla_{x}(A_{j}+\rho_{I, j})
\\&-\vec{E}_{i}\cdot\nabla_{x}A_{j}\}+\sum_{\substack{
  i+j=k-2, \\
  0\leq i,j\leq k-2
}}\{(A_{j}+\rho_{I, j})\mathrm{div}_{x}(\vec{E}_{i}+\vec{u}_{I, i})-A_{j}\mathrm{div}_{x}\vec{E}_{i}\})+Er_{1}+O(\epsilon^{N+1}),
\end{aligned}
\end{equation}

\begin{equation}\label{resi2-jia-ini}
\begin{aligned}
&\mathcal{L}_{2}\Big(\sum_{k=0}^{N}{\epsilon}^{k}(\rho_{k}+\rho_{I, k}), \sum_{k=0}^{N}\epsilon^{k}(\vec{u}_{k}+\vec{u}_{I, k}), \sum_{k=0}^{N}\epsilon^{k}(\theta_{k}+\theta_{I, k}), \sum_{k=0}^{N}{\epsilon}^{k}(f_{k}+f_{I, k})\Big)\\=&\mathcal{L}_{2}\Big(\sum_{k=0}^{N}{\epsilon}^{k}\rho_{k}, \sum_{k=0}^{N}\epsilon^{k}\vec{u}_{k}, \sum_{k=0}^{N}\epsilon^{k}\theta_{k}, \sum_{k=0}^{N}{\epsilon}^{k}f_{k}\Big)+\sum_{k=0}^{N}\epsilon^{k}\Big((A_{0}+\rho_{I, 0})\partial_{\tau}(\vec{E}_{k}+\vec{u}_{I, k})-A_{0}\partial_{\tau}\vec{E}_{k}
\\&+(A_{k}+\rho_{I, k})\partial_{\tau}(\vec {E}_{0}+\vec{u}_{I, 0})-A_{k}\partial_{\tau}\vec{E}_{0}+\sum_{\substack{
  i+j=k-2, \\
  0\leq i,j\leq k-2
}}\{(A_{i}+\rho_{I, i})\partial_{\tau}(\vec{E}_{j}+\vec{u}_{I, j})-A_{i}\partial_{\tau}\vec{E}_{j}\}
\end{aligned}
\end{equation}
\begin{equation}\nonumber
\begin{aligned}
\\&+\sum_{\substack{
  i+j+l=k-2, \\
  0\leq i, j, l\leq k-2
}}\{(A_{i}+\rho_{I, i})(\vec{E}_{j}+\vec{u}_{I, j})\cdot\nabla_{x}(\vec{E}_{l}+\vec{u}_{I, l})-A_{i}\vec{E}_{j}\cdot\nabla_{x}\vec{E}_{l}\}+\sum_{\substack{
  i+j=k-2, \\
  0\leq i,j\leq k-2
}}\{(A_{i}
\\&+\rho_{I, i})\nabla_{x}(D_{j}+\theta_{I, j})-A_{i}\nabla_{x}D_{j}\}+\sum_{\substack{
  i+j=k-2, \\
  0\leq i,j\leq k-2
}}\{(D_{j}+\theta_{I, j})\cdot\nabla_{x}(A_{i}+\rho_{I, i})-D_{j}\cdot\nabla_{x}A_{i}\}
\\&-\langle\vec{w}f_{I, k-1}\rangle-\langle\vec{w}f_{I, k-3}\rangle\Big)+Er_{2}
+O(\epsilon^{N+1}),\hspace{6.4 cm}
\end{aligned}
\end{equation}

\begin{equation}\label{resi3-jia-ini}
\begin{aligned}
&\mathcal{L}_{3}\Big(\sum_{k=0}^{N}{\epsilon}^{k}(\rho_{k}+\rho_{I, k}), \sum_{k=0}^{N}\epsilon^{k}(\vec{u}_{k}+\vec{u}_{I, k}), \sum_{k=0}^{N}\epsilon^{k}(\theta_{k}+\theta_{I, k}), \sum_{k=0}^{N}{\epsilon}^{k}(f_{k}+f_{I, k})\Big)\\=&\mathcal{L}_{3}\Big(\sum_{k=0}^{N}{\epsilon}^{k}\rho_{k}, \sum_{k=0}^{N}\epsilon^{k}\vec{u}_{k}, \sum_{k=0}^{N}\epsilon^{k}\theta_{k}, \sum_{k=0}^{N}{\epsilon}^{k}f_{k}\Big)+\sum_{k=0}^{N}\epsilon^{k}\Big((A_{0}+\rho_{I, 0})\partial_{\tau}(D_{k}+\theta_{I, k})-A_{0}\partial_{\tau}D_{k}\\&+(A_{k}+\rho_{I, k})\partial_{\tau}(D_{0}+\theta_{I, 0})-A_{k}\partial_{\tau}D_{0}+\sum_{\substack{
  i+j=k, \\
  0\leq i, j\leq k-1
}}((A_{i}+\rho_{I, i})\partial_{\tau}(D_{j}+\theta_{I, j})-A_{i}\partial_{\tau}D_{j})\\&+\sum_{\substack{
  i+j+l=k-2, \\
  0\leq i, j, l\leq k-2
}}((A_{i}+\rho_{I, i})(\vec{E}_{j}+\vec{u}_{I, j})\cdot\nabla_{x}(D_{l}+\theta_{I, l})-A_{i}\vec{E}_{j}\cdot\nabla_{x}D_{l})+\sum_{\substack{
  i+j+l=k-2, \\
  0\leq i, j, l\leq k-2
}}((A_{i}+\rho_{I, i})
\\&\times(D_{l}+\theta_{I, l})\mathrm{div}_{x}(\vec{E}_{j}+\vec{u}_{I, j})-A_{i}D_{l}\mathrm{div}_{x}\vec{E}_{j})-\overline{f_{I, k}}
+\sum_{\substack{
  i+j+l+m=k, \\
  i,j,l,m\geq 0
}}((D_{i}+\theta_{I, i})(D_{j}+\theta_{I, j})
\\&\times(D_{l}+\theta_{I, l})(D_{m}+\theta_{I, m})-D_{i}D_{j}D_{l}D_{m})
+\sum_{\substack{
  i+j=k-1, \\
  0\leq i, j\leq k-1
}}(\langle\vec{w}(J_{i}+f_{I, i})\rangle\cdot(\vec{E}_{j}+\vec{u}_{I, j})
\\&-\langle\vec{w}J_{i}\rangle\cdot\vec{E}_{j})
+\sum_{\substack{
  i+j=k-3, \\
  0\leq i, j\leq k-3
}}(\langle\vec{w}(J_{i}+f_{I, i})\rangle\cdot(\vec{E}_{j}+\vec{u}_{I, j})-\langle\vec{w}J_{i}\rangle\cdot\vec{E}_{j})\Big)+Er_{3}
+O(\epsilon^{N+1})
\end{aligned}
\end{equation}
and
\begin{equation}\label{resi4-zaijia-xin}
\begin{aligned}
&\mathcal{L}_{4}\Big(\sum_{k=0}^{N}{\epsilon}^{k}(f_{k}+f_{I, k}), \sum_{k=0}^{N}\epsilon^{k}(\theta_{k}+\theta_{I, k})\Big)
\\=&\mathcal{L}_{4}\Big(\sum_{k=0}^{N}{\epsilon}^{k}f_{k}, \sum_{k=0}^{N}\epsilon^{k}\theta_{k}\Big)
+\sum_{k=0}^{N}(\partial_{\tau}f_{I, k}+\vec{w}\cdot\nabla_{x}f_{I, k-1}+f_{I, k}+f_{I, k-2}-\overline{f_{I, k-2}}
\\&-(B_{k}(D_{0}+\theta_{I, 0};D_{k}+\theta_{I, k})-B_{k}(D_{0};D_{k})))
+Er_{4}+O(\epsilon^{N+1}),
\end{aligned}
\end{equation}
where
\begin{equation}\label{error1}
\begin{aligned}
Er_{1}=&\sum_{k=0}^{N}\epsilon^{k}\Big(\sum_{\substack{
  i+j=k-2, \\
  0\leq i,j\leq k-2
}}\{(\vec{u}_{i}+\vec{u}_{I, i})\cdot\nabla_{x}(\rho_{j}+\rho_{I, j})-\vec{u}_{i}\cdot\nabla_{x}\rho_{j}\}
\end{aligned}
\end{equation}
\begin{equation}\nonumber
\begin{aligned}
\\&+\sum_{\substack{
  i+j=k-2, \\
  0\leq i,j\leq k-2
}}\{(\rho_{j}+\rho_{I, j})\mathrm{div}_{x}(\vec{u}_{i}+\vec{u}_{I, i})-\rho_{j}\mathrm{div}_{x}\vec{u}_{i}\}
-\sum_{\substack{
  i+j=k-2, \\
  0\leq i,j\leq k-2
}}\{(\vec{E}_{i}+\vec{u}_{I, i})\cdot\nabla_{x}(A_{j}+\rho_{I, j})
\\&-\vec{E}_{i}\cdot\nabla_{x}\rho_{j}\}-\sum_{\substack{
  i+j=k-2, \\
  0\leq i,j\leq k-2
}}\{(A_{j}+\rho_{I, j})\mathrm{div}_{x}(\vec{E}_{i}+\vec{u}_{I, i})-A_{j}\mathrm{div}_{x}\vec{E}_{i}\}\Big),
\end{aligned}
\end{equation}

\begin{equation}\label{error2}
\begin{aligned}
&Er_{2}\\=&\sum_{k=0}^{N}\epsilon^{k}\Big((\rho_{0}+\rho_{I, 0})\partial_{\tau}(\vec{u}_{k}+\vec{u}_{I, k})-\rho_{0}\partial_{\tau}\vec{u}_{k}+(\rho_{k}+\rho_{I, k})\partial_{\tau}(\vec u_{0}+\vec{u}_{I, 0})-\rho_{k}\partial_{\tau}\vec{u}_{0}-((A_{0}+\rho_{I, 0})\\&\partial_{\tau}(\vec{E}_{k}+\vec{u}_{I, k})+A_{0}\partial_{\tau}\vec{E}_{k}+(A_{k}+\rho_{I, k})\partial_{\tau}(\vec {E}_{0}+\vec{u}_{I, 0})-A_{k}\partial_{\tau}\vec{E}_{0})+\sum_{\substack{
  i+j=k-2, \\
  0\leq i,j\leq k-2
}}\{(\rho_{i}+\rho_{I, i})\partial_{\tau}(\vec{u}_{j}\\&+\vec{u}_{I, j})-\rho_{i}\partial_{\tau}\vec{u}_{j}\}
+\sum_{\substack{
  i+j+l=k-2, \\
  0\leq i, j, l\leq k-2
}}\{(\rho_{i}+\rho_{I, i})(\vec{u}_{j}
+\vec{u}_{I, j})\cdot\nabla_{x}(\vec{u}_{l}+\vec{u}_{I, l})-\rho_{i}\vec{u}_{j}\cdot\nabla_{x}\vec{u}_{l}\}
+\sum_{\substack{
  i+j=k-2, \\
  0\leq i,j\leq k-2
}}\\&\{(\rho_{i}+\rho_{I, i})\nabla_{x}(\theta_{j}+\theta_{I, j})-\rho_{i}\nabla_{x}\theta_{j}\}
+\sum_{\substack{
  i+j=k-2, \\
  0\leq i,j\leq k-2
}}
\{(\theta_{j}+\theta_{I, j})\cdot\nabla_{x}(\rho_{i}+\rho_{I, i})-\theta_{j}\cdot\nabla_{x}\rho_{i}\}\Big)-\sum_{k=0}^{N}\\&\epsilon^{k}\Big((A_{0}+\rho_{I, 0})\partial_{\tau}(\vec{E}_{k}+\vec{u}_{I, k})
-A_{0}\partial_{\tau}\vec{E}_{k}+(A_{k}+\rho_{I, k})\partial_{\tau}(\vec {E}_{0}+\vec{u}_{I, 0})-A_{k}\partial_{\tau}\vec{E}_{0}
+\sum_{\substack{
  i+j=k-2, \\
  0\leq i,j\leq k-2
}}\{(A_{i}\\&+\rho_{I, i})\partial_{\tau}(\vec{E}_{j}+\vec{u}_{I, j})
-A_{i}\partial_{\tau}\vec{E}_{j}\}
+\sum_{\substack{
  i+j+l=k-2, \\
  0\leq i, j, l\leq k-2
}}\{(A_{i}+\rho_{I, i})(\vec{E}_{j}+\vec{u}_{I, j})\cdot\nabla_{x}(\vec{E}_{l}+\vec{u}_{I, l})
-A_{i}\vec{E}_{j}\\&\cdot\nabla_{x}\vec{E}_{l}\}
+\sum_{\substack{
  i+j=k-2, \\
  0\leq i,j\leq k-2
}}\{(A_{i}+\rho_{I, i})\nabla_{x}(D_{j}+\theta_{I, j})-A_{i}\nabla_{x}D_{j}\}
+\sum_{\substack{
  i+j=k-2, \\
  0\leq i,j\leq k-2
}}\{(D_{j}+\theta_{I, j})\cdot\nabla_{x}(A_{i}\\&+\rho_{I, i})-D_{j}\cdot\nabla_{x}A_{i}\})\Big),
\end{aligned}
\end{equation}

\begin{equation}\label{error3}
\begin{aligned}
&Er_{3}\\=&\sum_{k=0}^{N}\epsilon^{k}
\Big(((\rho_{0}+\rho_{I, 0})
\partial_{\tau}(\theta_{k}+\theta_{I, k})-\rho_{0}\partial_{\tau}\theta_{k}-(A_{0}+\rho_{I, 0})\partial_{\tau}(D_{k}+\theta_{I, k})+A_{0}\partial_{\tau}D_{k}
+((\rho_{k}+\rho_{I, k})\\&\partial_{\tau}(\theta_{0}+\theta_{I, 0})-\rho_{k}\partial_{\tau}\theta_{0}-(A_{k}+\rho_{I, k})\partial_{\tau}(D_{0}+\theta_{I, 0})+A_{k}\partial_{\tau}D_{0})
+\sum_{\substack{
  i+j=k, \\
  0\leq i, j\leq k-1
}}((\rho_{i}+\rho_{I, i})\partial_{\tau}(\theta_{j}+\theta_{I, j})\\&-\rho_{i}\partial_{\tau}\theta_{j}-(A_{i}+\rho_{I, i})\partial_{\tau}(D_{j}+\theta_{I, j})+A_{i}\partial_{\tau}D_{j})
+\sum_{\substack{
  i+j+l=k, \\
  0\leq i, j, l\leq k
}}((\rho_{i}+\rho_{I, i})(\vec{u}_{j}+\vec{u}_{I, j})\cdot\nabla_{x}(\theta_{l}+\theta_{I, l})
\end{aligned}
\end{equation}
\begin{equation}\nonumber
\begin{aligned}
\\&-\rho_{i}\vec{u}_{j}\cdot\nabla_{x}\theta_{l}-(A_{i}+\rho_{I, i})(\vec{E}_{j}+\vec{u}_{I, j})\cdot\nabla_{x}(D_{l}+\theta_{I, l})+A_{i}\vec{B}_{j}\cdot\nabla_{x}D_{l})+\sum_{\substack{
  i+j+l=k, \\
  0\leq i, j, l\leq k
}}((\rho_{i}+\rho_{I, i})\\&(\theta_{l}+\theta_{I, l})\mathrm{div}_{x}(\vec{u}_{j}+\vec{u}_{I, j})-\rho_{i}\theta_{l}\mathrm{div}_{x}\vec{u}_{j}
-(A_{i}+\rho_{I, i})(D_{l}+\theta_{I, l})\mathrm{div}_{x}(\vec{E}_{j}+\vec{u}_{I, j})+A_{i}D_{l}\mathrm{div}_{x}\vec{E}_{j})
\\&+\sum_{\substack{
  i+j+l+m=k, \\
  i,j,l,m\geq 0
}}((\theta_{i}+\theta_{I, i})(\theta_{j}+\theta_{I, j})(\theta_{l}+\theta_{I, l})(\theta_{m}+\theta_{I, m})-\theta_{i}\theta_{j}\theta_{l}\theta_{m}-(D_{i}+\theta_{I, i})(D_{j}+\theta_{I, j})\\&(D_{l}+\theta_{I, l})(D_{m}+\theta_{I, m})
+D_{i}D_{j}D_{l}D_{m})
+\sum_{\substack{
  i+j=k-1, \\
  0\leq i, j\leq k-1
}}(\langle\vec{w}(f_{i}+f_{I, i})\rangle\cdot(\vec{u}_{j}+\vec{u}_{I, j})-\langle\vec{w}f_{i}\rangle\cdot\vec{u}_{j}\\&-\langle\vec{w}(J_{i}+f_{I, i})\rangle\cdot(\vec{E}_{j}+\vec{u}_{I, j})
+\langle\vec{w}J_{i}\rangle\cdot\vec{E}_{j})
+\sum_{\substack{
  i+j=k-3, \\
  0\leq i, j\leq k-3
}}(\langle\vec{w}(f_{i}+f_{I, i})\rangle\cdot(\vec{u}_{j}+\vec{u}_{I, j})-\langle\vec{w}f_{i}\rangle\cdot\vec{u}_{j}\\&-\langle\vec{w}(J_{i}+f_{I, i})\rangle\cdot(\vec{E}_{j}+\vec{u}_{I, j})+\langle\vec{w}J_{i}\rangle\cdot\vec{E}_{j})\Big),
\end{aligned}
\end{equation}

\begin{equation}\label{error4}
\begin{aligned}
Er_{4}=&-\sum_{k=0}^{N}{\epsilon}^{k}(B_{k}(\theta_{0}+\theta_{I, 0};\theta_{k}+\theta_{I, k})-B_{k}(\theta_{0};\theta_{k}))
+\sum_{k=0}^{N}{\epsilon}^{k}(B_{k}(D_{0}+\theta_{I, 0};D_{k}+\theta_{I, k})\\&-B_{k}(D_{0};D_{k})),
\end{aligned}
\end{equation}
and $O(\epsilon^{N+1})$ denotes the higher-order remainder and is well-defined, $A_{k}=A_{k}(\tau, \vec x)$, $\vec{E}_{k}=\vec{E}_{k}(\tau, \vec x)$, $D_{k}=D_{k}(\tau, \vec x)$ and $J_{k}=J_{k}(\tau, \vec x, \vec w)$ $k=0, 1, \cdot\cdot\cdot, N$ are obtained by Taylor-expanding $\rho_{k}(\epsilon^{2}\tau, \vec x), \vec{u}_{k}(\epsilon^{2}\tau, \vec x), \theta_{k}(\epsilon^{2}\tau, \vec x), f_{k}(\epsilon^{2}\tau, \vec x, \vec w)$, respectively, at $\tau=0$:
\begin{equation}\label{formula1}
A_{k}(\tau, \vec x)=\sum_{l=0}^{k}\epsilon^{l}\frac{\tau^{l}}{l!}\frac{\partial^{l}}{\partial t^{l}}\rho_{k-l}(0, \vec x),\ \ \vec{E}_{k}(\tau, \vec x)=\sum_{l=0}^{k}\epsilon^{l}\frac{\tau^{l}}{l!}\frac{\partial^{l}}{\partial t^{l}}\vec{u}(0, \vec x),
\end{equation}
and
\begin{equation}\label{formula2}
D_{k}(\tau, \vec x)=\sum_{l=0}^{k}\epsilon^{l}\frac{\tau^{l}}{l!}\frac{\partial^{l}}{\partial t^{l}}\theta_{k-l}(0, \vec x),\ \ J_{k}(\tau, \vec x, \vec w)=\sum_{l=0}^{k}\epsilon^{l}\frac{\tau^{l}}{l!}\frac{\partial^{l}}{\partial t^{l}}f_{k-l}(0, \vec x, \vec w).
\end{equation}

{\rmk To ensure that the coefficients in the initial layer system remain independent of $\epsilon$, we adopt the definitions in \eqref{formula1} and \eqref{formula2}. These formulations are motivated by the analogous construction in (2.12) of \cite{traceth-1}; however, the additional factor $\epsilon^{l}$ appearing in \eqref{formula1} and \eqref{formula2} arises from the distinct time scaling $\tau=\frac{t}{\epsilon^{2}}$, in contrast to the spatial scaling $\eta=\frac{x_{1}}{\epsilon}$ used in \cite{traceth-1}. Consequently, the operator $A_{k}$ may be interpreted as a perturbation of $A_{0}$ by terms of order $O(\epsilon)$, and hence is well-defined; the same reasoning applies mutatis mutandis to $E_{k}, D_{k}$ and $J_{k}$. Moreover, in view of the definition of $D_{k}(\tau, \vec x)$ and the presence of $\Delta_{x}\theta_{k}$ in the definition of $\theta_{k}$ in \eqref{thetak-1}, the requirement that the initial data belong to $H^{2N+2}(\mathbb{T}^{3})$ rather than merely $H^{N+2}(\mathbb{T}^{3})$, becomes necessary.}

Collecting terms of the same order in \eqref{resi1-jia-ini}-\eqref{resi4-zaijia-xin}, we obtain
\begin{equation}\label{biaoda1}
\begin{aligned}
&\partial_{\tau}\rho_{I, k}+\sum_{\substack{
  i+j=k-2, \\
  0\leq i,j\leq k-2
}}\{(\vec{E}_{i}+\vec{u}_{I, i})\cdot\nabla_{x}(A_{j}+\rho_{I, j})-\vec{E}_{i}\cdot\nabla_{x}A_{j}\}\\&+\sum_{\substack{
  i+j=k-2, \\
  0\leq i,j\leq k-2
}}\{(A_{j}+\rho_{I, j})\mathrm{div}_{x}(\vec{E}_{i}+\vec{u}_{I, i})-A_{j}\mathrm{div}_{x}\vec{E}_{i}\}=0,
\end{aligned}
\end{equation}

\begin{equation}\label{biaoda2}
\begin{aligned}
&(A_{0}+\rho_{I, 0})\partial_{\tau}(\vec{E}_{k}+\vec{u}_{I, k})-A_{0}\partial_{\tau}\vec{E}_{k}
+(A_{k}+\rho_{I, k})\partial_{\tau}(\vec {E}_{0}+\vec{u}_{I, 0})-A_{k}\partial_{\tau}\vec{E}_{0}+\sum_{\substack{
  i+j=k-2, \\
  0\leq i,j\leq k-2
}}\\&\{(A_{i}+\rho_{I, i})\partial_{\tau}(\vec{E}_{j}+\vec{u}_{I, j})-A_{i}\partial_{\tau}\vec{E}_{j}\}
+\sum_{\substack{
  i+j+l=k-2, \\
  0\leq i, j, l\leq k-2
}}\{(A_{i}+\rho_{I, i})(\vec{E}_{j}+\vec{u}_{I, j})\cdot\nabla_{x}(\vec{E}_{l}+\vec{u}_{I, l})
\\&-A_{i}\vec{E}_{j}\cdot\nabla_{x}\vec{E}_{l}\}+\sum_{\substack{
  i+j=k-2, \\
  0\leq i,j\leq k-2
}}\{(A_{i}+\rho_{I, i})\nabla_{x}(D_{j}+\theta_{I, j})-A_{i}\nabla_{x}D_{j}\}+\sum_{\substack{
  i+j=k-2, \\
  0\leq i,j\leq k-2
}}
\\&\{(D_{j}+\theta_{I, j})\cdot\nabla_{x}(A_{i}+\rho_{I, i})-D_{j}\cdot\nabla_{x}A_{i}\}-\langle\vec{w}f_{I, k-1}\rangle-\langle\vec{w}f_{I, k-3}\rangle=0,\hspace{3.4 cm}
\end{aligned}
\end{equation}
\begin{equation}\label{biaoda3}
\begin{aligned}
&(A_{0}+\rho_{I, 0})\partial_{\tau}(D_{k}+\theta_{I, k})-A_{0}\partial_{\tau}D_{k}
+(A_{k}+\rho_{I, k})\partial_{\tau}(D_{0}+\theta_{I, 0})-A_{k}\partial_{\tau}D_{0}+\sum_{\substack{
  i+j=k, \\
  0\leq i, j\leq k-1
}}
\\&((A_{i}+\rho_{I, i})\partial_{\tau}(D_{j}+\theta_{I, j})-A_{i}\partial_{\tau}D_{j})
+\sum_{\substack{
  i+j+l=k, \\
  0\leq i, j, l\leq k
}}((A_{i}+\rho_{I, i})(\vec{E}_{j}+\vec{u}_{I, j})\cdot\nabla_{x}(D_{l}+\theta_{I, l})
-\\&A_{i}\vec{E}_{j}\cdot\nabla_{x}D_{l})
+\sum_{\substack{
  i+j+l=k, \\
  0\leq i, j, l\leq k
}}((A_{i}+\rho_{I, i})(D_{l}+\theta_{I, l})\mathrm{div}_{x}(\vec{E}_{j}+\vec{u}_{I, j})-A_{i}D_{l}\mathrm{div}_{x}\vec{E}_{j})-\overline{f_{I, k}}
+\\&\sum_{\substack{
  i+j+l+m=k, \\
  i,j,l,m\geq 0
}}((D_{i}+\theta_{I, i})(D_{j}+\theta_{I, j})(D_{l}+\theta_{I, l})(D_{m}+\theta_{I, m})-D_{i}D_{j}D_{l}D_{m})
+\sum_{\substack{
  i+j=k-1, \\
  0\leq i, j\leq k-1
}}(\langle\vec{w}(J_{i}\\&+f_{I, i})\rangle\cdot(\vec{E}_{j}+\vec{u}_{I, j})-\langle\vec{w}J_{i}\rangle\cdot\vec{E}_{j})
+\sum_{\substack{
  i+j=k-3, \\
  0\leq i, j\leq k-3
}}(\langle\vec{w}(J_{i}+f_{I, i})\rangle\cdot(\vec{E}_{j}+\vec{u}_{I, j})-\langle\vec{w}J_{i}\rangle\cdot\vec{E}_{j})=0
\end{aligned}
\end{equation}
and
\begin{equation}\label{biaoda4}
\begin{aligned}
\partial_{\tau}f_{I, k}+\vec{w}\cdot\nabla_{x}f_{I, k-1}+f_{I, k}+f_{I, k-2}-\overline{f_{I, k-2}}-(B_{k}(D_{0}+\theta_{I, 0};D_{k}+\theta_{I, k})-B_{k}(D_{0}; D_{k})=0.
\end{aligned}
\end{equation}

\subsection{Construction of asymptotic expansion.}\label{construction}
The interior solution and the initial-layer solution are matched through the initial condition of \eqref{research equations}. To begin with, we have the following relations:
\begin{equation}\nonumber
\rho_{0}(0, \vec{x})+\rho_{I, 0}(0, \vec{x})=\rho^{0}(\vec{x}),\ \ \vec{u}_{0}(0, \vec{x})+\vec{u}_{I, 0}(0, \vec{x})=\vec{u}^{0}(\vec{x}),\ \ \theta_{0}(0, \vec{x})+\theta_{I, 0}(0, \vec{x})=\theta^{0}(\vec{x}),\hspace{-0.2 cm}
\end{equation}
\begin{equation}\nonumber
f_{0}(0, \vec{x}, \vec{w})+f_{I, 0}(0, \vec{x}, \vec{w})=h(\vec{x}, \vec{w}),\hspace{0.4 cm}
\end{equation}
\begin{equation}\nonumber
\rho_{k}(0, \vec{x})+\rho_{I, k}(0, \vec{x})=0,\ \ \vec{u}_{k}(0, \vec{x})+\vec{u}_{I, k}(0, \vec{x})=0,\ \ \theta_{k}(0, \vec{x})+\theta_{I, k}(0, \vec{x})=0,\ \ k\geq 1,\hspace{-0.9 cm}
\end{equation}
\begin{equation}\nonumber
f_{k}(0, \vec{x}, \vec{w})+f_{I, k}(0, \vec{x}, \vec{w})=0,\ \ k\geq 1.\hspace{0.0cm}
\end{equation}

The construction of $\rho_{k}, \rho_{I, k}$, $\vec u_{k}, \vec u_{I, k}$, $\theta_{k}, \theta_{I, k}$ and $f_{k}, f_{I, k}$ are as follows:

\noindent\textbf{Step 1.}Construction of zeroth-order terms.

On the basis of \eqref{biaoda1} and \eqref{biaoda2}, the equations for $\rho_{I, 0}, \vec{u}_{I, 0}$ are defined by
\begin{equation}\label{equation about theta I0}\left\{
\begin{split}
&\frac{\partial \rho_{I, 0}}{\partial\tau}=0,\;\;\frac{\partial \vec u_{I, 0}}{\partial\tau}=0,\\
&(\rho_{I, 0}(0, \vec x), \vec u_{I, 0}(0, \vec x))=(\rho^{0}(\vec x)-\rho_{0}(0, \vec x), \vec {u}^{0}(\vec x)-\vec u_{0}(0, \vec x)),\\
&\lim_{\tau\rightarrow\infty}(\rho_{I, 0}(\tau, \vec x), \vec u_{I, 0}(\tau, \vec x))=(0, 0).\\
\end{split}\right.
\end{equation}

A direct calculation yields $(\rho_{I, 0}, \vec{u}_{I, 0})=(0, 0)$. Hence, it follows that $\rho_{0}(0, \vec x)=\rho^{0}(\vec x)$ and $\vec u_{0}(0, \vec x)=\vec {u}^{0}(\vec x)$.

Henceforth, we impose the assumption $\theta_{0}(0, \vec x)=\theta_{0}^{0}(\vec x)$, with $\theta_{0}^{0}(\vec x)$ to be determined subsequently. In conjunction with $\eqref{zero-ini-sys}_{1}$, \eqref{biaoda3}, \eqref{biaoda4}, and the condition $\rho_{0}(0, \vec x)=\rho^{0}(\vec x)$, the zeroth-order initial-layer solution $(f_{I, 0}, \theta_{I, 0})$ is then defined by
\begin{equation}\label{equations about f I0,theta I0}\left\{
\begin{split}
&\frac{\partial f_{I, 0}}{\partial\tau}+f_{I, 0}=B(\theta_{0}^{0}+\theta_{I, 0})-B(\theta_{0}^{0}),\\
&\frac{\partial \theta_{I, 0}}{\partial\tau}=\frac{1}{\rho^{0}}(\overline{f_{I, 0}}-B(\theta_{0}^{0}+\theta_{I, 0})+B(\theta_{0}^{0})),\\
&f_{I, 0}(0, \vec{x}, \vec{w})=h(\vec{x}, \vec{w})-(\theta_{0}^{0}(\vec x))^{4}, \quad\theta_{I, 0}(0, \vec{x})=\theta^{0}(\vec x)-\theta_{0}^{0}(\vec x),\\
&\lim_{\tau\rightarrow\infty}f_{I, 0}(\tau, \vec{x}, \vec{w})=0,\ \ \lim_{\tau\rightarrow\infty}\theta_{I, 0}(\tau, \vec{x}, \vec{w})=0.
\end{split}\right.
\end{equation}

Let $\hat{f}_{0}(\tau, \vec w, \vec w)=f_{I, 0}(\tau, \vec w, \vec w)+\overline{f_{0}}(0, \vec x)=f_{I, 0}+(\theta_{0}^{0})^{4}$ and $\hat\theta_{0}(\tau, \vec x)=\theta_{I, 0}(\tau, \vec w)+\theta_{0}(0, \vec x)=\theta_{I, 0}+\theta_{0}^{0}$. Then the system \eqref{equations about f I0,theta I0} can be rewritten as
\begin{equation}\label{equations about f I0,theta I0-jia}\left\{
\begin{split}
&f_{I, 0}=\hat{f}_{0}-(\theta_{0}^{0})^{4}, \quad\theta_{I, 0}=\hat\theta_{0}-\theta_{0}^{0},\\
&\frac{\partial\hat{f}_{0}}{\partial\tau}+\hat{f}_{0}=B(\hat\theta_{0}),\\
&\frac{\partial\hat\theta_{0}}{\partial\tau}=\frac{1}{\rho^{0}}(\overline{\hat{f}_{0}}-B(\hat\theta_{0})),\\
&\hat{f}_{0}(0, \vec{x}, \vec{w})=h(\vec{x}, \vec{w}), \quad\hat\theta_{0}(0, \vec{x})=\theta^{0}(\vec x),\\
&\lim_{\tau\rightarrow\infty}\hat{f}_{0}(\tau, \vec{x}, \vec{w})=(\theta_{0}^{0}(\vec x))^{4},\ \ \lim_{\tau\rightarrow\infty}\hat\theta_{0}(\tau, \vec{x}, \vec{w})=\theta_{0}^{0}(\vec x).
\end{split}\right.
\end{equation}
{\lem\label{ini-0}
Assume that $h(\vec{x}, \vec{w})\in L^{\infty}(\mathbb{S}^{2}; H^{2N+2}(\mathbb{T}^{3}))$ and $\rho^{0}(\vec x), \vec{u}^{0}(\vec x), \theta^{0}(\vec{x})\in H^{2N+2}(\mathbb{T}^{3})$ with $h(\vec{x}, \vec{w})>0, \rho^{0}(\vec x), \theta^{0}(\vec x)\geq a$ and $\|h-(\theta^{0})^{4}\|_{L^{\infty}(\mathbb{S}^{2}; H^{2N+2}(\mathbb{T}^{3}))}\leq \eta$, where $N\geq 17$, $a$ and $\eta$ are positive constants and $\eta$ is suitably small and independent of $\epsilon$. The problem \eqref{equations about f I0,theta I0} has a unique solution $(f_{I, 0}, \theta_{I, 0})\in (C^{1}([0, \infty); L^{\infty}(\mathbb{S}^{2}; H_{\vec x}^{2N+2}))\cap L^{1}([0, \infty); L^{\infty}(\mathbb{S}^{2}; H_{\vec x}^{2N+2})))\times (C^{1}([0, \infty); H_{\vec x}^{2N+2})\cap L^{1}([0, \infty); H_{\vec x}^{2N+2}))$. Furthermore, there exists a suitably small constant $\sigma_{0}>0$ such that
\begin{equation}\label{biaoda6}
\|\theta_{0}^{0}(\vec x)\|_{H_{\vec x}^{2N+2}}\leq C,
\end{equation}
\begin{equation}\label{ineq-theta00-jia1}
\|e^{\sigma_{0}\tau}f_{I, 0}\|_{L^{\infty}([0, \infty); L^{\infty}_{\vec w}H_{\vec x}^{2N+2})}+\|e^{\sigma_{0}\tau}\theta_{I, 0}\|_{L^{\infty}([0, \infty); H_{\vec x}^{2N+2})}\leq C\eta
\end{equation}
and
\begin{equation}\label{ineq-theta00-jia2}
\|e^{\sigma_{0}\tau}f_{I, 0}\|_{L^{1}([0, \infty); L^{\infty}_{\vec w}H_{\vec x}^{2N+2})}+\|e^{\sigma_{0}\tau}\theta_{I, 0}\|_{L^{1}([0, \infty); H_{\vec x}^{2N+2})}\leq C\eta.
\end{equation}\nonumber}

\noindent\textbf{Proof.} From the equations $\eqref{equations about f I0,theta I0}_{1}$ and $\eqref{equations about f I0,theta I0}_{2}$, we can derive 
\begin{equation}\nonumber
\partial_{\tau}(\overline {f_{I, 0}}+\rho^{0}\theta_{I, 0})=0.
\end{equation}
By the fact that $f_{I, 0}, \theta_{I, 0}\rightarrow 0$ as $\tau\rightarrow\infty$, we can obtain
\begin{equation}\label{daoshuweiling}
\overline {f_{I, 0}}+\rho^{0}\theta_{I, 0}\equiv 0,
\end{equation}
which further implies
\begin{equation}\label{biaoda5}
(\theta^{0}_{0})^{4}+\rho^{0}\theta_{0}^{0}=\overline h+\rho^{0}\theta^{0}:=l_{0}(\vec x).
\end{equation}

Define $H(\theta_{0}^{0})=(\theta^{0}_{0})^{4}+\rho^{0}\theta_{0}^{0}:\mathbb{R}^{+}\rightarrow\mathbb{R}^{+}$. Since $\rho^{0}\geq a>0$, we have $H'=4(\theta_{0}^{0})^{3}+\rho^{0}>0$ for any $\theta_{0}^{0}\in\mathbb{R}^{+}$. So, $H$ is a one-to-one function. By the fact that $l_{0}\geq a^{2}>0$ and $H$ is a $C^\infty$ function about $\theta_{0}^{0}$, then the inverse function $\theta_{0}^{0}(\vec x)=H^{-1}(l_{0}(\vec x))\geq b>0$ where $b=b(a)$, $H^{-1}(l_{0}(\vec x))$ is a $C^\infty$ function about $l_{0}$. Since $l_{0}(\vec x)\in H_{\vec x}^{2N+2}$, we have $\theta_{0}^{0}(\vec x)=H^{-1}(l_{0}(\vec x))\in H_{\vec x}^{2N+2}$. Then, \eqref{biaoda6} follows directly with some constant $C$ depending on $\rho^{0}, \overline h$ and $\theta^{0}$.

According to \eqref{biaoda5} and the definition of $l_{0}(\vec x)$, we have
\begin{equation}\label{biaoshi1}
(\theta^{0}_{0})^{4}+\rho^{0}\theta_{0}^{0}=\overline h+\rho^{0}\theta^{0},
\end{equation}
which can be further written as 
\begin{equation}\label{biaoshi2}
(\theta^{0}_{0})^{4}-(\theta^{0})^{4}+\rho^{0}(\theta_{0}^{0}-\theta^{0})=\overline h-(\theta^{0})^{4}.
\end{equation}
Then,
\begin{equation}\label{biaoshi3}
\theta_{0}^{0}-\theta^{0}=\frac{\overline h-(\theta^{0})^{4}}{((\theta_{0}^{0})^{2}+(\theta^{0})^{2})(\theta_{0}^{0}+\theta^{0})+\rho^{0}}.
\end{equation}
So,
\begin{equation}\label{biaoshi4}
\|\theta_{0}^{0}-\theta^{0}\|_{H_{\vec x}^{2N+2}(\mathbb{T}^{3})}\leq C\eta.
\end{equation}
Furthermore, we have
\begin{equation}\label{biaoshi5}
\|h-(\theta_{0}^{0})^{4}\|_{H_{\vec x}^{2N+2}(\mathbb{T}^{3})}\leq \|h-(\theta^{0})^{4}\|_{H_{\vec x}^{2N+2}(\mathbb{T}^{3})}+\|(\theta_{0}^{0})^{4}-(\theta^{0})^{4}\|_{H_{\vec x}^{2N+2}(\mathbb{T}^{3})}\leq C\eta.
\end{equation}
According to the equality \eqref{daoshuweiling}, we can write $\eqref{equations about f I0,theta I0}_{2}$ in the following form
\begin{equation}
\frac{\partial \theta_{I, 0}}{\partial\tau}=\frac{1}{\rho^{0}}(-\rho^{0}\theta_{I, 0}-B(\theta_{0}^{0}+\theta_{I, 0})+B(\theta_{0}^{0})),
\end{equation}
which can be further rewritten as
\begin{equation}\label{wenduini}
\frac{\partial \theta_{I, 0}}{\partial\tau}+\Big(1+\frac{4}{\rho^{0}}(\theta_{0}^{0})^{3}\Big)\theta_{I, 0}=\frac{1}{\rho^{0}}(-6(\theta_{0}^{0})^{2}\theta_{I, 0}^{2}-4\theta_{0}^{0}\theta_{I, 0}^{3}).
\end{equation}
Then, we have
\begin{equation}\label{jifenfangcheng0}
\begin{aligned}
\theta_{I, 0}(\tau)=e^{-\Big(1+\frac{4}{\rho^{0}}(\theta_{0}^{0})^{3}\Big)\tau}(\theta^{0}-\theta_{0}^{0})-\int_{0}^{\tau}e^{-\Big(1+\frac{4}{\rho^{0}}(\theta_{0}^{0})^{3}\Big)(\tau-s)}\frac{1}{\rho^{0}}(6(\theta_{0}^{0})^{2}\theta_{I, 0}^{2}+4\theta_{0}^{0}\theta_{I, 0}^{3})(s)\mathrm{d}s.
\end{aligned}
\end{equation}

We define the energy radius $E_{13}=\{\theta_{I, 0}:\|\theta_{I, 0}\|_{L^{\infty}(0, \infty; H_{\vec x}^{2N+2})}\leq 2\tilde\eta\}$, where $\tilde\eta=\\2\|e^{-(1+\frac{4}{\rho^{0}}(\theta_{0}^{0})^{3})\tau}(\theta^{0}-\theta_{0}^{0})\|_{L^{\infty}(0, \infty; H_{\vec x}^{2N+2})}=O(\eta)$.

We give the linearized form of \eqref{jifenfangcheng0} as follows:
\begin{equation}\label{jifenfangcheng00}
\begin{aligned}
\theta_{I, 0}^{k+1}(\tau)=e^{-\Big(1+\frac{4}{\rho^{0}}(\theta_{0}^{0})^{3}\Big)\tau}(\theta^{0}-\theta_{0}^{0})
-\int_{0}^{\tau}e^{-(1+4(\theta_{0}^{0})^{3})(\tau-s)}\frac{1}{\rho^{0}}(6(\theta_{0}^{0})^{2}(\theta_{I, 0}^{k})^{2}+\frac{4}{\rho^{0}}\theta_{0}^{0}(\theta_{I, 0}^{k})^{3})(s)\mathrm{d}s,
\end{aligned}
\end{equation}
where we choose $\theta_{I, 0}^{0}=0$. Assume $\|\theta_{I, 0}^{k}\|_{L^{\infty}(0, \infty; H_{\vec x}^{2N+2})}\leq\tilde\eta,$ then,
\begin{equation}\label{eneinitialo}
\begin{aligned}
\|\theta_{I, 0}^{k+1}\|_{L^{\infty}(0, \infty; H_{\vec x}^{2N+2})}\leq\tilde\eta+C_{1}\tilde\eta^{2}+C_{2}\tilde\eta^{3}\leq 2\tilde\eta,
\end{aligned}
\end{equation}
for sufficiently small $\eta$.

Furthermore, we set $\tilde{\theta}_{I, 0}^{k+1}=\theta_{I, 0}^{k+1}-\theta_{I, 0}^{k}$. Then, we can derive
\begin{equation}\label{jifenfangcheng1}
\begin{aligned}
\tilde{\theta}_{I, 0}^{k+1}(\tau)=&-\int_{0}^{\tau}e^{-\Big(1+\frac{4}{\rho^{0}}(\theta_{0}^{0})^{3}\Big)(\tau-s)}\frac{1}{\rho^{0}}(6(\theta_{0}^{0})^{2}(\theta_{I, 0}^{k}+\theta_{I, 0}^{k-1})\tilde{\theta}_{I, 0}^{k}+4\theta_{0}^{0}(\theta_{I, 0}^{k}+\theta_{I, 0}^{k-1})\\&\times((\theta_{I, 0}^{k})^{2}+\theta_{I, 0}^{k}\theta_{I, 0}^{k-1}+(\theta_{I, 0}^{k-1})^{2})\tilde{\theta}_{I, 0}^{k}(s)\mathrm{d}s,
\end{aligned}
\end{equation}
which further implies
\begin{equation}\label{contraction-ini}
\begin{aligned}
\|\tilde{\theta}_{I, 0}^{k+1}\|_{L^{\infty}(0, \infty; H_{\vec x}^{2N+2})}\leq (C_{1}'\tilde\eta+C_{2}'\tilde\eta^{2})\|\tilde{\theta}_{I, 0}^{k}\|_{L^{\infty}(0, \infty; H_{\vec x}^{2N+2})}.
\end{aligned}
\end{equation}
For sufficiently small $\eta$, we have $0<(C_{1}'\tilde\eta+C_{2}'\tilde\eta^{2})<1$. Then, we can claim that $\{\theta_{I, 0}^{k}\}$ is a contraction sequence.

According to \eqref{eneinitialo} and \eqref{contraction-ini}, we have $\theta_{I, 0}^{k}\rightarrow \theta_{I, 0}$ in $E_{13}$ as $k\rightarrow\infty$, by the Banach fixed point theorem.

According to $\eqref{equations about f I0,theta I0}_{1}$ and $\eqref{equations about f I0,theta I0}_{3}$, we can write $f_{I, 0}$ in the following form
\begin{equation}\label{fI0biaodashi}
f_{I, 0}=e^{-\tau}(h-(\theta_{0}^{0})^{4})+\int_{0}^{\tau}e^{s-\tau}(B(\theta_{0}^{0}+\theta_{I, 0})-B(\theta_{0}^{0}))\mathrm{d}s,
\end{equation}
which further implies the existence solution $f_{I, 0}\in C^{1}([0, \infty); L^{\infty}_{\vec w}H_{\vec x}^{2N+2})$. Moreover, we have
\begin{equation}
\|f_{I, 0}\|_{L^{\infty}(0, \infty; H_{\vec x}^{2N+2})}\leq C\eta.
\end{equation}

In the following, let us prove that $(f_{I, 0}, \theta_{I, 0})\in L^{1}([0, \infty); L^{\infty}_{\vec w}H_{\vec x}^{2N+2})\times L^{1}([0, \infty); H_{\vec x}^{2N+2})$.

According to \eqref{fI0biaodashi}, we have 
\begin{equation}\label{gron1}
\begin{aligned}
\|f_{I, 0}(\tau)\|_{L^{\infty}_{\vec w}H_{\vec x}^{2N+2}}\leq& e^{-\tau}\|h-(\theta_{0}^{0})^{4}\|_{L^{\infty}_{\vec w}H_{\vec x}^{2N+2}}+C\eta\int_{0}^{\tau}\|\theta_{I, 0}\|_{H_{\vec x}^{2N+2}}\mathrm{d}s
\\\leq& C\eta e^{-\tau}+C\eta\int_{0}^{\tau}\|\theta_{I, 0}\|_{H_{\vec x}^{2N+2}}\mathrm{d}s.
\end{aligned}
\end{equation}
According to \eqref{jifenfangcheng0}, we have
\begin{equation}\label{jifenfangcheng0-jia}
\begin{aligned}
\theta_{I, 0}(\tau)=e^{-\Big(1+\frac{4}{\rho^{0}}(\theta_{0}^{0})^{3}\Big)\tau}(\theta^{0}-\theta_{0}^{0})
-\int_{0}^{\tau}e^{-\Big(1+\frac{4}{\rho^{0}}(\theta_{0}^{0})^{3}\Big)(\tau-s)}(6(\theta_{0}^{0})^{2}\theta_{I, 0}^{2}+4\theta_{0}^{0}\theta_{I, 0}^{3})(s)\mathrm{d}s,
\end{aligned}
\end{equation}
So,
\begin{equation}\label{gron2}
\begin{aligned}
\|\theta_{I, 0}(\tau)\|_{H_{\vec x}^{2N+2}}\leq C\eta e^{-\kappa\tau}+C\eta\int_{0}^{\tau}e^{-\kappa(\tau-s)}\|\theta_{I, 0}\|_{H_{\vec x}^{2N+2}}\mathrm{d}s,
\end{aligned}
\end{equation}
where $\kappa\in (0, 1)$ is a suitably small positive constant. Add \eqref{gron1} and \eqref{gron2} together, we have
\begin{equation}\label{gron3}
\begin{aligned}
&\|f_{I, 0}(\tau)\|_{L^{\infty}_{\vec w}H_{\vec x}^{2N+2}}+\|\theta_{I, 0}(\tau)\|_{H_{\vec x}^{2N+2}}\\\leq& C\eta e^{-\kappa\tau}+C\eta\int_{0}^{\tau}e^{-\kappa(\tau-s)}(\|f_{I, 0}\|_{L^{\infty}_{\vec w}H_{\vec x}^{2N+2}}+\|\theta_{I, 0}\|_{H_{\vec x}^{2N+2}})\mathrm{d}s,
\end{aligned}
\end{equation}
which further implies 
\begin{equation}\label{gron4}
\begin{aligned}
&\|f_{I, 0}\|_{L^{1}([0, \infty); L^{\infty}_{\vec w}H_{\vec x}^{2N+2})}+\|\theta_{I, 0}(\tau)\|_{L^{1}([0, \infty); H_{\vec x}^{2N+2})}\leq C\eta,
\end{aligned}
\end{equation}
combining with the Gronwall's inequality.

Finally, we will prove the exponential decay of the above solutions.

For suitably small $0<\sigma_{0}<1/2$, multiplying $e^{\sigma_{0}\tau}$ on both sides of \eqref{equations about f I0,theta I0}, we have
\begin{equation}\label{equations about f I0,theta I0-zhishu}\left\{
\begin{split}
&\frac{\partial (e^{\sigma_{0}\tau}f_{I, 0})}{\partial\tau}+(1-\sigma_{0})e^{\sigma_{0}\tau}f_{I, 0}=e^{\sigma_{0}\tau}(B(\theta_{0}^{0}+\theta_{I, 0})-B(\theta_{0}^{0})),\\
&\partial_{\tau}(e^{\sigma_{0}\tau}\theta_{I, 0})+\Big(\frac{4}{\rho^{0}}(\theta_{0}^{0})^{3}-\sigma_{0}\Big)e^{\sigma_{0}\tau}\theta_{I, 0}=e^{\sigma_{0}\tau}\frac{1}{\rho^{0}}(\overline{f_{I, 0}}-6(\theta_{0}^{0})^{2}\theta_{I, 0}^{2}-4\theta_{0}^{0}\theta_{I, 0}^{3}-\theta_{I, 0}^{4}),\\
&f_{I, 0}(0, \vec{x}, \vec{w})=h(\vec{x}, \vec{w})-\overline {f_{0}}(0, \vec x),\ \ \theta_{I, 0}(0, \vec{x}, \vec{w})=\theta^{0}(\vec x)-\theta_{0}^{0}(\vec x),\\
&\lim_{\tau\rightarrow\infty}f_{I, 0}(\tau, \vec{x}, \vec{w})=0,\ \ \lim_{\tau\rightarrow\infty}\theta_{I, 0}(\tau, \vec{x}, \vec{w})=0.
\end{split}\right.
\end{equation}
Then, we can write 
\begin{equation}\label{fI0biaodashi-zhishu}
e^{\sigma_{0}\tau}f_{I, 0}=e^{-(1-\sigma_{0})\tau}(h-(\theta_{0}^{0})^{4})+\int_{0}^{\tau}e^{-(1-\sigma_{0})(\tau-s)}(B(\theta_{0}^{0}+\theta_{I, 0})-B(\theta_{0}^{0}))\mathrm{d}s
\end{equation}
and
\begin{equation}\nonumber
\begin{aligned}
e^{\sigma_{0}\tau}\theta_{I, 0}(\tau)=&\exp\Big\{-\Big(1+\frac{4}{\rho^{0}}(\theta_{0}^{0})^{3}-\sigma_{0}\Big)\tau\Big\}(\theta^{0}-\theta^{0}_{0})
\\&+\int_{0}^{\tau}\exp\Big\{-\Big(1+\frac{4}{\rho^{0}}(\theta_{0}^{0})^{3}-2\sigma_{0}\Big)(\tau-s)\Big\}(-6(\theta_{0}^{0})^{2}\theta_{I, 0}^{2}-4\theta_{0}^{0}\theta_{I, 0}^{3}-\theta_{I, 0}^{4})\mathrm{d}s.
\end{aligned}
\end{equation}
For suitably small $\sigma_{0}$, we can derive
\begin{equation}\nonumber
\|e^{\sigma_{0}\tau}f_{I, 0}\|_{L^{\infty}([0, \infty); L^{\infty}_{\vec w}H_{\vec x}^{2N+2})}+\|e^{\sigma_{0}\tau}\theta_{I, 0}\|_{L^{\infty}([0, \infty); H_{\vec x}^{2N+2})}\leq C\eta
\end{equation}
and
\begin{equation}\nonumber
\|e^{\sigma_{0}\tau}f_{I, 0}\|_{L^{1}([0, \infty); L^{\infty}_{\vec w}H_{\vec x}^{2N+2})}+\|e^{\sigma_{0}\tau}\theta_{I, 0}\|_{L^{1}([0, \infty); H_{\vec x}^{2N+2})}\leq C\eta
\end{equation}
by using similar calculations as above and Gronwall's inequality.

{\rmk In view of the definitions $\hat{f}_{0}(\tau, \vec w, \vec w)=f_{I, 0}+\overline{f_{0}}(0, \vec x)=f_{I, 0}+(\theta_{0}^{0})^{4}$, $\hat\theta_{0}(\tau, \vec x)=\theta_{I, 0}+\theta_{0}^{0}$, it follows that both $(\hat{f}_{0}(\tau, \vec w, \vec w), \hat\theta_{0}(\tau, \vec x))$ are well-defined.}

After imposing the initial datas of \eqref{zero-ini-sys}, we can derive
\begin{equation}\label{equations about f0 epsilon and theta}\left\{
\begin{split}
&f_{0}=\overline {f_{0}}=\theta_{0}^{4},\\
&\partial_{t}\rho_{0}+\mathrm{div}(\rho_{0}\vec u_{0})=0,~~\mathrm{in} \ \ (0,T)\times \mathbb{T}^{3},\\
&\partial_{t}(\rho_{0}\vec{u}_{0})+\mathrm{div}_{x}(\rho_{0}\vec{u}_{0}\otimes\vec u_{0})+\nabla_{x}(\rho_{0}\theta_{0})
=-\frac{1}{3}\nabla_{x}\theta_{0}^{4},~~\mathrm{in} \ \ (0,T)\times \mathbb{T}^{3},\\
&\partial_{t}(\rho_{0}\theta_{0}+\theta_{0}^{4})+\mathrm{div}_{x}(\rho_{0}\theta_{0}\vec{u}_{0})+\rho_{0}\theta_{0}\mathrm{div}_{x}\vec{u}_{0}
-\frac{1}{3}\Delta_{x}\theta_{0}^{4}=\frac{1}{3}\nabla_{x}\theta_{0}^{4}\cdot\vec u_{0},~~\mathrm{in} \ \ (0,T)\times\mathbb{T}^{3}, \\
&\rho_{0}(0, \vec{x})=\rho^{0}(\vec x),\;\;\vec{u}_{0}(0, \vec{x})=\vec{u}^{0}(\vec x),\ \ \theta_{0}(0, \vec{x})=\theta_{0}^{0}(\vec{x}),\;\;\overline {f_{0}}(0, \vec{x})=(\theta_{0}^{0})^{4}~~\mathrm{in}  \ \ \mathbb{T}^{3}.
\end{split}\right.
\end{equation}

{\lem \label{zerothinterior}
For the given data $\rho^{0}(\vec x), \vec{u}^{0}(\vec x), \theta_{0}^{0}(\vec{x})\in H_{\vec x}^{2N+2}$ and $\rho^{0}(\vec x)\geq a, \theta_{0}^{0}\geq b$ where $a$ and $b$ are given positive constants, the problem \eqref{equations about f0 epsilon and theta} has a unique local solution on $[0, T]$, and $(\rho_{0}, \vec u_{0}, \theta_{0}, \overline{f_{0}})\in C([0, T]; H_{\vec x}^{2N+2})\times C([0, T]; H_{\vec x}^{2N+2})\times C([0, T]; H_{\vec x}^{2N+2})\cap L^{2}(0, T; H_{\vec x}^{2N+3})\times C([0, T]; H_{\vec x}^{2N+2})\cap L^{2}(0, T; H_{\vec x}^{2N+3})$. Furthermore, we have
\begin{equation}\label{ineq-theta00}
\|(\rho_{0}, \vec u_{0}, \theta_{0}, \overline{f_{0}})\|_{L^{\infty}(0, T; H_{\vec x}^{2N+2})}+\|(\theta_{0}, \overline{f_{0}})\|_{L^{2}(0, T; H_{\vec x}^{2N+3})}\leq C
\end{equation}
and $\rho_{0}, \theta_{0}\geq \frac{a}{2}$.}

\noindent\textbf{Proof.} We can write \eqref{equations about f0 epsilon and theta} in the following form
\begin{equation}\label{equations about f0 epsilon and theta-1}\left\{
\begin{split}
&f_{0}=\overline {f_{0}}=\theta_{0}^{4},\\
&\partial_{t}\rho_{0}+\vec u_{0}\cdot\nabla_{x}\rho_{0}+\rho_{0}\mathrm{div}\vec u_{0}=0,~~\mathrm{in} \ \ (0,T)\times \mathbb{T}^{3},\\
&\partial_{t}\vec{u}_{0}+\vec{u}_{0}\cdot\nabla_{x}\vec u_{0}+\frac{\theta_{0}}{\rho_{0}}\nabla_{x}\rho_{0}
=-\Big(\frac{4\theta_{0}^{3}}{3\rho_{0}}+1\Big)\nabla_{x}\theta_{0},~~\mathrm{in} \ \ (0,T)\times \mathbb{T}^{3},\\
&\partial_{t}\theta_{0}-\frac{4\theta_{0}^{3}}{(\rho_{0}+4\theta_{0}^{3})}\Delta_{x}\theta_{0}=-\frac{\rho_{0}}{(\rho_{0}
+4\theta_{0}^{3})}\vec{u}_{0}\cdot\nabla_{x}\theta_{0}-\frac{\rho_{0}\theta_{0}}{(\rho_{0}+4\theta_{0}^{3})}\mathrm{div}_{x}\vec{u}_{0}
\\&+\frac{4\theta_{0}^{2}}{3(\rho_{0}+4\theta_{0}^{3})}|\nabla_{x}\theta_{0}|^{2}+\frac{4\theta_{0}^{3}}{3(\rho_{0}+4\theta_{0}^{3})}\nabla_{x}\theta_{0}\cdot\vec u_{0},~~\mathrm{in} \ \ (0,T)\times\mathbb{T}^{3}, \\
&\rho_{0}(0, \vec{x})=\rho^{0}(\vec x),\;\;\vec{u}_{0}(0, \vec{x})=\vec{u}^{0}(\vec x),\ \ \theta_{0}(0, \vec{x})=\theta_{0}^{0}(\vec{x}),\;\;\overline {f_{0}}(0, \vec{x})=(\theta_{0}^{0})^{4}~~\mathrm{in}  \ \ \mathbb{T}^{3}.
\end{split}\right.
\end{equation}
We introduce the vector
\begin{equation}\nonumber
V_{0}=(\rho_{0}, u_{01}, u_{02}, u_{03})^{t}
\end{equation}
and the matrix
\begin{equation}\nonumber
\tilde{A}_{j}(V_{0})=\{\tilde{a}_{mn}\}_{4\times4}
\end{equation}
whose nonzero entries are given by
\begin{equation}\nonumber
\tilde{a}_{ii}=u_{0j}\ (i=1,2,3,4),\qquad
\tilde{a}_{1(j+1)}=\rho_{0},\qquad
\tilde{a}_{(j+1)1}=\frac{\theta_{0}}{\rho_{0}}.
\end{equation}
Additionally, we define $H(V_{0}, \theta_{0})=(g_{0}, g_{1}, g_{2}, g_{3})^{t}$ with
\begin{equation}\nonumber
g_{0}=0,\qquad g_{j}=-\left(\frac{4\theta_{0}^{3}}{3\rho_{0}}+1\right)\partial_{j}\theta_{0},\quad j=1,2,3.
\end{equation}
Therefore, we can rewrite $(\ref{equations about f0 epsilon and theta-1})_{1\sim 2}$ as
\begin{equation}\label{research equations-3-1}
\frac{\partial V_{0}}{\partial t}+\sum_{j=1}^{3}\tilde{A}_{j}(V_{0})\frac{\partial V_{0}}{\partial x_{j}}=H(V_{0}, \theta_{0}).
\end{equation}
We shall study the Cauchy problem for $\eqref{equations about f0 epsilon and theta-1}_{3}$ and \eqref{research equations-3-1} together with the initial data
\begin{equation}\label{research equations-3-3}
V_{0}(0, \vec x)=V^{0}(\vec x)=(\rho^{0}(\vec x), \vec u^{0}(\vec x)),~~~~~~\theta_{0}(0, \vec x)=\theta_{0}^{0}(\vec x).
\end{equation}
First, we symmetrize \eqref{research equations-3-1} by multiplying it by the symmetrizing matrix
\begin{equation}\label{xishujuzhen-1-xin}
A_{0}(V_{0})
=
\left[
\begin{array}{ccc}
(\rho_{0})^{-1}&0\\
0&\frac{\rho_{0}}{\theta_{0}}\mathbb{I}_{3\times3}
\end{array}
\right].
\end{equation}
Therefore, we shall prove the first part of Theorem \ref{zerothinterior} for the quasilinear symmetric system:
\begin{equation}\label{research equations-3-2-xin}
A_{0}(V_{0})\frac{\partial V_{0}}{\partial t}+\sum_{j=1}^{3}A_{j}(V_{0})\frac{\partial V_{0}}{\partial x_{j}}=F(V_{0}, \theta_{0})
\end{equation}
coupled to $\eqref{equations about f0 epsilon and theta-1}_{4}$, where $A_{j}(V_{0}):=A_{0}(V_{0})\tilde{A}_{j}(V_{0})$ is symmetric, and
\begin{equation}\nonumber
F(V_{0}, \theta_{0})=A_{0}(V_{0})H(V_{0}, \theta_{0})=(l_{0}, l_{1}, l_{2}, l_{3})^{t}
\end{equation}
with $l_{0}=0$, and for $j=1, 2, 3$,
\begin{equation}\nonumber
l_{j}=-\Big(\frac{4\theta_{0}^{2}}{3}+\frac{\rho_{0}}{\theta_{0}}\Big)\partial_{j}\theta_{0}.
\end{equation}

For the remainder of the proof concerning local existence, uniqueness, and uniform estimates for \eqref{equations about f0 epsilon and theta}, we refer to \cite{yuanmoxing-bu-jia3jia}. It should be noted that while the local existence result in \cite{yuanmoxing-bu-jia3jia} was derived for the system \eqref{research equations} in the non-equilibrium regime, the underlying method carries over to our present setting.

\noindent\textbf{Step 2.} Construction of first-order terms.

Henceforth, let $\rho_{1}^{0}(\vec x)$, $\vec{u}_{1}^{0}(\vec x)$, $\theta_{1}^{0}(\vec x)$ and $f_{1}^{0}(\vec x, \vec w)$ denote the respective initial data of $\rho_{1}(t, \vec x)$, $\vec{u}_{1}(t, \vec x)$, $\theta_{1}(t, \vec w)$ and $f_{1}(t, \vec w, \vec w)$. The equations for $\rho_{I, 1}, \vec{u}_{I, 1}$ are then defined as follows:
\begin{equation}\label{equation about theta I1}\left\{
\begin{split}
&\frac{\partial \rho_{I, 1}}{\partial\tau}=0,~~\frac{\partial \vec u_{I, 1}}{\partial\tau}=\frac{1}{\rho^{0}}\langle\vec w(f_{I, 0}-\overline {f_{I, 0}})\rangle,\\
&(\rho_{I, 1}(0, \vec x), \vec u_{I, 1}(0, \vec x))=(-\rho_{1}^{0}(\vec x), -\vec u_{1}^{0}(\vec x)),\\
&\lim_{\tau\rightarrow\infty}(\rho_{I, 1}(\tau, \vec x), \vec{u}_{I, 1}(\tau, \vec x))=(0, 0),\\
\end{split}\right.
\end{equation}
combined \eqref{biaoda1}, \eqref{biaoda2} and the fact that $\rho_{I, 0}=0$ and $A_{0}=\rho^{0}$.

{\lem\label{rhoulayer1}
The problem \eqref{equation about theta I1} has a unique solution $(\rho_{I, 1}, \vec{u}_{I, 1})\in (C^{1}([0, \infty); L_{\vec w}^{\infty}H_{\vec x}^{2N+2})\cap L^{1}([0, \infty); L_{\vec w}^{\infty}H_{\vec x}^{2N+2}))\times (C^{1}([0, \infty); H_{\vec x}^{2N+2})\cap L^{1}([0, \infty); H_{\vec x}^{2N+2}))$. Furthermore, there exists a suitably small constant $\sigma_{11}>0$ such that
\begin{equation}\label{ineq-theta01}
\|e^{\sigma_{11}\tau}\rho_{I, 1}\|_{L^{\infty}([0, \infty); H_{\vec x}^{2N+2})}+\|e^{\sigma_{11}\tau}\vec{u}_{I, 1}\|_{L^{\infty}([0, \infty); H_{\vec x}^{2N+2})}+\|\vec{u}_{1}^{0}(\vec x)\|_{H_{\vec x}^{2N+2}}\leq C.
\end{equation}}
\noindent\textbf{Proof.} After a careful calculation, we have $\rho_{I, 1}=0$ and $\rho_{1}(0, \vec x)=\rho_{1}^{0}(\vec x)=0$. By $\eqref{equation about theta I1}_{2}$, we have
\begin{equation}\label{uI1}
\vec{u}_{I, 1}(\tau, \vec x)=-\vec u_{1}^{0}(\vec x)+\int_{0}^{\tau}\frac{1}{\rho^{0}}\langle\vec w(f_{I, 0}-\overline {f_{I, 0}})\rangle\mathrm{d}s.
\end{equation}
Letting $\tau\rightarrow 0$, we have
$\vec u_{1}^{0}(\vec x)=\int_{0}^{\infty}\frac{1}{\rho^{0}}\langle\vec w(f_{I, 0}-\overline {f_{I, 0}})\rangle\mathrm{d}s\in H_{\vec x}^{2N+2}$
and $\|\vec{u}_{I, 1}\|_{L^{\infty}([0, \infty); H_{\vec x}^{2N+2})}+\|\vec{u}_{1}^{0}(\vec x)\|_{H_{\vec x}^{2N+2}}\leq C\eta.$
Substituting \eqref{fI0biaodashi} into \eqref{uI1}, we derive
$\vec{u}_{I, 1}(\tau, \vec x)=-\vec u_{1}^{0}(\vec x)+\frac{1}{\rho^{0}}(1-e^{-\tau})\langle\vec w h\rangle$
and
$\vec u_{1}^{0}(\vec x)=\frac{1}{\rho^{0}}\langle\vec w h\rangle.$
So,
\begin{equation}\label{uI1-jia}
\vec{u}_{I, 1}(\tau, \vec x)=-\frac{e^{-\tau}}{\rho^{0}}\langle\vec w h\rangle,
\end{equation}
which implies the desired exponential decay estimate.

The first-order initial layer $(f_{I, 1}, \theta_{I, 1})$ is defined as
\begin{equation}\label{equations about f I1,theta I1}\left\{
\begin{split}
&\vec w\cdot\nabla_{x}f_{I, 0}+\frac{\partial f_{I, 1}}{\partial\tau}+f_{I, 1}=B_{1}(\theta_{0}^{0}+\theta_{I, 0}; \theta_{1}^{0}+\theta_{I, 1})-B_{1}(\theta_{0}^{0}; \theta_{1}^{0}),\\
&\frac{\partial \theta_{I, 1}}{\partial\tau}=\frac{1}{\rho^{0}}(\overline{f_{I, 1}}-B_{1}(\theta_{0}^{0}+\theta_{I, 0}; \theta_{1}^{0}+\theta_{I, 1})+B_{1}(\theta_{0}^{0}; \theta_{1}^{0})),\\
&f_{I, 1}(0, \vec{x}, \vec w)=-f_{1}^{0}(\vec x, \vec w),\ \ \theta_{I, 1}(0, \vec{x})=-\theta_{1}^{0}(\vec x),\\
&\lim_{\tau\rightarrow\infty}f_{I, 1}(\tau, \vec{x}, \vec{w})=0,\ \ \lim_{\tau\rightarrow\infty}\theta_{I, 1}(\tau, \vec{x})=0,
\end{split}\right.
\end{equation}
combining \eqref{biaoda3}, \eqref{biaoda4} and the fact that $\rho_{I, 0}=0$ and $A_{0}=\rho^{0}$.

{\lem\label{firstinitiallayer}
The problem \eqref{equations about f I1,theta I1} has a unique solution $(f_{I, 1}, \theta_{I, 1})\in (C^{1}([0, \infty); L_{\vec w}^{\infty}H_{\vec x}^{2N+1})\cap L^{1}([0, \infty); L_{\vec w}^{\infty}H_{\vec x}^{2N+1}))\times (C^{1}([0, \infty); H_{\vec x}^{2N+1})\cap L^{1}([0, \infty); H_{\vec x}^{2N+1}))$. Furthermore, there exists a suitably small constant $\sigma_{12}>0$ such that
\begin{equation}\label{ineq-theta11}
\|e^{\sigma_{12}\tau}f_{I, 1}\|_{L^{\infty}([0, \infty); L_{\vec w}^{\infty}H_{\vec x}^{2N+1})}+\|e^{\sigma_{12}\tau}\theta_{I, 1}\|_{L^{\infty}([0, \infty); H_{\vec x}^{2N+1})}+\|\theta_{1}^{0}(\vec x)\|_{H_{\vec x}^{2N+1}}\leq C
\end{equation}
and
\begin{equation}\label{ineq-theta12}
\|e^{\sigma_{12}\tau}f_{I, 1}\|_{L^{1}([0, \infty); L_{\vec w}^{\infty}H_{\vec x}^{2N+1})}+\|e^{\sigma_{12}\tau}\theta_{I, 1}\|_{L^{1}([0, \infty); H_{\vec x}^{2N+1})}\leq C.
\end{equation}}

\noindent\textbf{Proof.} From the equations $\eqref{equations about f I1,theta I1}_{1}$ and $\eqref{equations about f I1,theta I1}_{2}$, we can derive 
\begin{equation}\nonumber
\partial_{\tau}(\overline {f_{I, 1}}+\rho^{0}\theta_{I, 1})=-\langle\vec w\cdot\nabla_{x}f_{I, 0}\rangle.
\end{equation}
Then, 
\begin{equation}\nonumber
\overline {f_{I, 1}}(\tau, \vec x)+\rho^{0}(\vec x)\theta_{I, 1}(\tau, \vec x)=-\overline{f_{1}^{0}}(\vec x)-\rho^{0}(\vec x)\theta_{1}^{0}(\vec x)-\int_{0}^{\tau}\langle\vec w\cdot\nabla_{x}f_{I, 0}\rangle\mathrm{d}s
\end{equation}
By the fact that $f_{I, 1}, \theta_{I, 1}\rightarrow 0$ as $\tau\rightarrow\infty$, we can obtain
\begin{equation}\label{daoshuweiling-1}
\overline {f_{I, 1}}+\rho^{0}\theta_{I, 1}\equiv 0,
\end{equation}
which further implies
\begin{equation}\nonumber
\overline{f_{1}^{0}}(\vec x)+\rho^{0}(\vec x)\theta_{1}^{0}(\vec x)=-\int_{0}^{\infty}\langle\vec w\cdot\nabla_{x}f_{I, 0}\rangle\mathrm{d}s:=l_{1}(\vec x).
\end{equation}
In view of \eqref{fk} for $k=1$ and the fact that $B_{1}(\theta_{0}^{0}; \theta_{1}^{0})=4(\theta^{0}_{0})^{3}\theta_{1}^{0}(\vec x)$, we derive the following equality:
\begin{equation}\nonumber
4(\theta_{0}^{0})^{3}\theta_{1}^{0}+\rho^{0}\theta_{1}^{0}=l_{1}(\vec x).
\end{equation}
Then, we can directly get 
\begin{equation}\nonumber
\theta_{1}^{0}(\vec x)=\frac{1}{4(\theta_{0}^{0})^{3}+\rho^{0}}l_{1}(\vec x)\in H_{\vec x}^{2N+1}.
\end{equation}
Simultaneously, we can get $f_{1}^{0}\in L^{\infty}_{\vec w}H_{\vec x}^{2N+1}$ by \eqref{fk} for $k=1$.

We can write $\eqref{equations about f I1,theta I1}_{2}$ in the following form
\begin{equation}\nonumber
\frac{\partial \theta_{I, 1}}{\partial\tau}+\Big(1+\frac{4}{\rho^{0}}(\theta_{0}^{0}+\theta_{I, 0})^{3}\Big)\theta_{I, 1}=-\frac{4}{\rho^{0}}(\theta_{0}^{0}+\theta_{I, 0})^{3}\theta_{1}^{0}+\frac{4}{\rho^{0}}(\theta_{0}^{0})^{3}\theta_{1}^{0}.
\end{equation}
Then, we have
\begin{equation}\nonumber
\theta_{I, 1}=e^{-\Big(1+\frac{4}{\rho^{0}}(\theta_{0}^{0}+\theta_{I, 0})^{3}\Big)\tau}(-\theta_{1}^{0})+\int_{0}^{\tau}e^{-\Big(1+\frac{4}{\rho^{0}}(\theta_{0}^{0}+\theta_{I, 0})^{3}\Big)(\tau-s)}(-\frac{4}{\rho^{0}}(\theta_{0}^{0}+\theta_{I, 0})^{3}\theta_{1}^{0}+\frac{4}{\rho^{0}}(\theta_{0}^{0})^{3}\theta_{1}^{0})\mathrm{d}s,
\end{equation}
which implies the global existence of $\theta_{I, 1}$ and $\theta_{I, 1}\in C^{1}([0, \infty); H_{\vec x}^{2N+1})$. 

According to $\eqref{equations about f I0,theta I0}_{1}$ and $\eqref{equations about f I0,theta I0}_{3}$, we can write $f_{I, 1}$ in the following form
\begin{equation}\nonumber
f_{I, 1}=-e^{-\tau}f_{1}^{0}+\int_{0}^{\tau}e^{s-\tau}(B_{1}(\theta_{0}^{0}+\theta_{I, 0}; \theta_{1}^{0}+\theta_{I, 1})-B_{1}(\theta_{0}^{0}; \theta_{1}^{0})-\vec w\cdot\nabla_{x}f_{I, 0})\mathrm{d}s,
\end{equation}
which further implies the existence solution $f_{I, 1}\in C^{1}([0, \infty); L^{\infty}_{\vec w}H_{\vec x}^{2N+1}))$. \hfill$\Box$

Following the same line of reasoning as in Lemma \ref{ini-0}, we arrive at the estimates \eqref{ineq-theta11} and \eqref{ineq-theta12}.

We write the equations about $\rho_{1}, \vec{u}_{1}, \theta_{1}$ and $f_{1}$ as follows
\begin{equation}\label{equations about theta 1 overline f epsilon 1-jia}\left\{
\begin{split}
&f_{1}(t, x, w)=B_{1}(\theta_{0}; \theta_{1})-\vec w\cdot\nabla_{x}B(\theta_{0}),\\
&\frac{\partial\rho_{1}}{\partial t}+\vec{u}_{0}\cdot\nabla_{x}\rho_{1}+\rho_{0}\mathrm{div}_{x}\vec{u}_{1}=-\rho_{1}\mathrm{div}_{x}\vec{u}_{0}-\vec{u}_{1}\cdot\nabla_{x}\rho_{0}\ \ \mathrm{in}\ \ (0, T]\times\mathbb{T}^{3},\\
&\frac{\partial\vec{u}_{1}}{\partial t}+\vec{u}_{0}\cdot\nabla_{x}\vec{u}_{1}+\Big(1+\frac{\theta_{0}^{3}}{3\rho_{0}}\Big)\nabla_{x}\theta_{1}+\frac{\theta_{0}}{\rho_{0}}\nabla_{x}\rho_{1}
=-\frac{\rho_{1}}{\rho_{0}}\frac{\partial \vec{u}_{0}}{\partial t}-\vec{u}_{1}\cdot\nabla_{x}\vec{u}_{0}-\frac{\rho_{1}}{\rho_{0}}\vec{u}_{0}\cdot\nabla_{x}\vec{u}_{0}
\\&-\frac{\rho_{1}}{\rho_{0}}\nabla_{x}\theta_{0}-\frac{\theta_{1}}{\rho_{0}}(\nabla_{x}\rho_{0}+\theta_{0}^{2}\nabla_{x}\theta_{0})-\frac{1}{3}\nabla_{x}\theta_{0}^{4}\ \ \mathrm{in}\ \ (0, T]\times\mathbb{T}^{3},\\
&(\rho_{0}+\theta_{0}^{3})\frac{\partial\theta_{1}}{\partial t}+\theta_{1}\partial_{t}(\theta_{0}^{3})+\rho_{1}\partial_{t}\theta_{0}+\rho_{0}\vec{u}_{0}\cdot\nabla_{x}\theta_{1}+\rho_{0}\vec{u}_{1}\cdot\nabla_{x}\theta_{0}
+\rho_{1}\vec{u}_{0}\cdot\nabla_{x}\theta_{0}+\rho_{0}\theta_{0}\mathrm{div}_{x}\vec{u}_{1}\\&+\rho_{1}\theta_{0}\mathrm{div}_{x}\vec{u}_{0}+\rho_{0}\theta_{1}\mathrm{div}_{x}\vec{u}_{0}-4\theta_{0}^{3}\Delta_{x}\theta_{1}
-12\theta_{0}^{2}\nabla_{x}\theta_{0}\cdot\nabla_{x}\theta_{1}-4\theta_{1}\Delta_{x}\theta_{0}^{3}
-\frac{\theta_{0}^{3}}{3}\nabla_{x}\theta_{1}\cdot\vec{u}_{0}-\\&\theta_{1}\theta_{0}^{2}\nabla_{x}\theta_{0}\cdot\vec{u}_{0}
+\frac{4\theta_{0}^{3}}{3}\nabla_{x}\theta_{0}\cdot\vec{u}_{1}=-\partial_{t}\theta_{0}^{4}+\frac{1}{3}\Delta_{x}\theta_{0}^{4}+\frac{1}{3}\nabla_{x}\overline{f_{0}}\cdot\vec{u}_{1}
-\frac{1}{3}\nabla_{x}\theta_{0}^{4}\cdot\vec{u}_{0}\ \ \mathrm{in}\ \ (0, T]\times\mathbb{T}^{3},\\
&\rho_{1}(0, \vec{x})=0,\;\;\vec{u}_{1}(0, \vec{x})=\vec u_{1}^{0}(\vec x)~~\mathrm{in}  \ \ \mathbb{T}^{3},\\
&\theta_{1}(0, \vec{x})=\theta_{1}^{0}(\vec x),\;\;f_{1}(0, \vec{x}, \vec w)=B_{1}(\theta_{0}^{0}; \theta_{1}^{0})-\vec w\cdot\nabla_{x}B(\theta_{0}^{0})~~\mathrm{in}  \ \ \mathbb{T}^{3}.
\end{split}\right.
\end{equation}

{\lem\label{firstinterior}
The problem \eqref{equations about theta 1 overline f epsilon 1-jia} has a unique solution on $[0, T]$, and $(\rho_{1}, \vec u_{1}, \theta_{1}, \overline{f_{1}})\in C([0, T]; H_{\vec x}^{2N+1})\times C([0, T]; H_{\vec x}^{2N+1})\times C([0, T]; H_{\vec x}^{2N+1})\cap L^{2}(0, T; H_{\vec x}^{2N+2})\times C([0, T]; L_{\vec w}^{\infty}H_{\vec x}^{2N+1})\\\cap L^{2}(0, T; L_{\vec w}^{\infty}H_{\vec x}^{2N+2})$. Furthermore, we have
\begin{equation}\label{ineq-theta1}
\begin{aligned}
&\|(\rho_{1}, \vec u_{1}, \theta_{1})\|_{L^{\infty}(0, T; H_{\vec x}^{2N+1})}+\|\theta_{1}\|_{L^{2}(0, T; H_{\vec x}^{2N+2}))}+\|f_{1}\|_{L^{\infty}(0, T; L_{\vec w}^{\infty}H_{\vec x}^{2N+1})}\\&+\|f_{1}\|_{L^{2}(0, T; L_{\vec w}^{\infty}H_{\vec x}^{2N+2})}\leq C.
\end{aligned}
\end{equation}}

Note that $B_{1}(\theta_{0}; \theta_{1})=4\theta_{0}^{3}\theta_{1}$. Similar to the proof of the first part of Lemma \ref{zerothinterior}, we can write system (\ref{equations about theta 1 overline f epsilon 1-jia}) as follows
\begin{equation}\label{equations about theta 1 overline f epsilon 1 jia1}\left\{
\begin{split}
&f_{1}(t, x, w)=B_{1}(\theta_{0}; \theta_{1})-\vec w\cdot\nabla_{x}B(\theta_{0}),\\
&A_{0}(V_{0})\frac{\partial V_{1}}{\partial t}+\sum_{j=1}^{3}A_{j}(V_{0})\frac{\partial V_{1}}{\partial x_{j}}+D(V_{0})V_{1}=F(V_{0}, \theta_{1})\ \ \mathrm{in}\ \ (0, T]\times\mathbb{T}^{3},\\
&\frac{\partial \theta_{1}}{\partial t}-\frac{4\rho_{0}\theta_{0}^{3}}{3(\rho_{0}+4\theta_{0}^{3})}\Delta_{x}\theta_{1}=\frac{\rho_{0}}{\rho_{0}+4\theta_{0}^{3}}G(V_{0}, V_{1}, \theta_{0}, \theta_{1})\ \ \mathrm{in}\ \ (0, T]\times\mathbb{T}^{3},\\
&V_{1}(0, \vec x)=(0, \vec{u}_{1}^{0})^{t},\ \ \theta_{1}(0, \vec x)=\theta_{1}^{0}(\vec x),~~\mathrm{in}  \ \ \mathbb{T}^{3},\\
\end{split}\right.
\end{equation}
where 
\begin{equation}\nonumber
V_{0}=(\rho_{0}, \vec{u}_{0}), V_{1}=(\rho_{1}, \vec{u}_{1}), B_{1}=B_{1}(\theta_{0}, \theta_{1}), F(V_{0}, \theta_{1})=\frac{\rho_{0}}{\theta_{0}}(0, 4\nabla_{x}(\theta_{0}^{3}\theta_{1}))^{t},
\end{equation}
\begin{equation}\nonumber
\begin{aligned}
G(V_{0}, V_{1}, \theta_{0}, \theta_{1})=&-\theta_{1}\partial_{t}(\theta_{0}^{3})-\rho_{1}\partial_{t}\theta_{0}-\rho_{0}\vec{u}_{0}\cdot\nabla_{x}\theta_{1}-\rho_{0}\vec{u}_{1}\cdot\nabla_{x}\theta_{0}
-\rho_{0}\theta_{0}\mathrm{div}_{x}\vec{u}_{1}-\rho_{1}\theta_{0}\mathrm{div}_{x}\vec{u}_{0}\\&-\rho_{0}\theta_{1}\mathrm{div}_{x}\vec{u}_{0}
-3\theta_{0}^{2}\nabla_{x}\theta_{0}\cdot\nabla_{x}\theta_{1}-\theta_{1}\Delta_{x}\theta_{0}^{3}
-\frac{\theta_{0}^{3}}{3}\nabla_{x}\theta_{1}\cdot\vec{u}_{0}-\theta_{1}\theta_{0}^{2}\nabla_{x}\theta_{0}\cdot\vec{u}_{0}
\\&-\frac{4\theta_{0}^{3}}{3}\nabla_{x}\theta_{0}\cdot\vec{u}_{1}-\partial_{t}\theta_{0}^{4}+\frac{1}{3}\Delta_{x}\theta_{0}^{4}+\frac{1}{3}\nabla_{x}\overline{f_{0}}\cdot\vec{u}_{1}
-\frac{1}{3}\nabla_{x}\theta_{0}^{4}\cdot\vec{u}_{0},
\end{aligned}
\end{equation}
$D(V_{0})$ is a $4\times 4$ matrix and belongs to $C([0, T]; H_{\vec x}^{2N+1})$, $A_{0}(V_{0})$ defined the same as \eqref{xishujuzhen-1-xin} and $A_{j}$ is a symmetric matrix defined similarly as in (\ref{research equations-3-2-xin}).

The local existence of smooth solutions can be established by an iterative method as in Lemma \ref{zerothinterior}. To close Lemma \ref{firstinterior}, one only needs to establish the a priori estimates.

Define
\begin{equation}\label{def Et1}
E_{1}(t):=\frac{1}{2}\int_{\mathbb{T}^{3}}A_{0}(V_{0})D_{x}^{\gamma}V_{1}\cdot D_{x}^{\gamma}V_{1}\mathrm{d}\vec x
\end{equation}
and
\begin{equation}\label{def Et2}
E_{2}(t):=\frac{1}{2}\int_{\mathbb{T}^{3}}D_{x}^{\gamma}\theta_{1}\cdot D_{x}^{\gamma}\theta_{1}\mathrm{d}\vec x,
\end{equation}
where $0\leq\gamma\leq 2N+1$.

Upon applying the operator to both sides of \eqref{equations about theta 1 overline f epsilon 1 jia1}{2} and \eqref{equations about theta 1 overline f epsilon 1 jia1}{3}, and then taking the
$L^{2}((0, t)\times\mathbb{T}^{3})$-inner product of the resulting equations with $D_{x}^{\gamma}V_{1}$ and $D_{x}^{\gamma}\theta_{1}$, respectively, we arrive at
\begin{equation}\label{gr-1}
\begin{aligned}
E_{1}(t)-E_{1}(0)\leq C(a, b, r^{0})(\|V_{1}\|_{L^{2}(0, t; H_{\vec x}^{2N+1})}^{2}+\|\theta_{1}\|_{L^{2}(0, t; H_{\vec x}^{2N+1})}^{2}+\delta\|D_{x}\theta_{1}\|_{L^{2}(0, t; H_{\vec x}^{2N+1})}^{2})
\end{aligned}
\end{equation}
and
\begin{equation}\label{gr-2}
\begin{aligned}
E_{2}(t)-E_{2}(0)+\underline{C}\|\nabla_{x}\theta_{1}\|_{L^{2}(0, t; H_{\vec x}^{2N+1})}^{2}\leq& C(a, b, r^{0})(\|V_{1}\|_{L^{2}(0, t; H_{\vec x}^{2N+1})}^{2}+\|\theta_{1}\|_{L^{2}(0, t; H_{\vec x}^{2N+1})}^{2}\\&+\delta\|D_{x}\theta_{1}\|_{L^{2}(0, t; H_{\vec x}^{2N+1})}^{2}),
\end{aligned}
\end{equation}
where $\underline{C}=\underline{C}(a, V_{0}, \theta_{0})$ is some small positive constant.

Choosing $\delta$ such that $C(a, V_{0}, \theta_{0})\delta=\underline{C}/4$ and summing up from $\gamma=0$ to $2N+1$, we have
\begin{equation}\label{gr-3}
\begin{aligned}
&\|V_{1}(t)\|_{H_{\vec x}^{2N+1}}^{2}+\|\theta_{1}(t)\|_{H_{\vec x}^{2N+1}}^{2}+\|D_{x}\theta_{1}\|_{L^{2}(0, t; H_{\vec x}^{2N+1})}^{2}\\\leq& C(a, V_{0}, \theta_{0})(\|V_{1}(0)\|^{2}_{H_{\vec x}^{2N+1}}+\|\theta_{1}(0)\|^{2}_{H_{\vec x}^{2N+1}}+\|V_{1}\|_{L^{2}(0, t; H_{\vec x}^{2N+1})}^{2}+\|\theta_{1}\|_{L^{2}(0, t; H_{\vec x}^{2N+1})}^{2}).
\end{aligned}
\end{equation}
Then, we have the following priori estimates
\begin{equation}\label{gr-2}
\begin{aligned}
&\|V_{1}(t)\|_{L^{\infty}(0, t; H_{\vec x}^{2N+1})}+\|\theta_{1}(t)\|_{L^{\infty}(0, t; H_{\vec x}^{2N+1})}+
\|D_{x}\theta_{1}\|_{L^{2}(0, t; H_{\vec x}^{2N+1})}\leq C(a, V_{0}, \theta_{0}, \theta_{1}^{0}, \vec u_{1}^{0}, t),
\end{aligned}
\end{equation}
by the Gronwall's inequality. Since $0<t\leq T\leq 1$, we derive the fact $C(a, V_{0}, \theta_{0}, \theta_{1}^{0}, \vec u_{1}^{0},t)\\\leq \tilde C$, where $\tilde C$ is a positive constant independent of $t$.

So, the local existence of \eqref{equations about theta 1 overline f epsilon 1 jia1} can be extended to the interval $[0, T]$ by the continuation method.

\noindent\textbf{Step 3.} Construction of kth-order terms.

In the following, we use $\rho_{k}^{0}(\vec x)$, $\vec{u}_{k}^{0}(\vec x)$, $\theta_{k}^{0}(\vec x)$ and $f_{k}^{0}(\vec x, \vec w)$ to denote the initial data of $\rho_{k}(t, \vec x)$, $\vec{u}_{k}(t, \vec x)$, $\theta_{k}(t, \vec w)$ and $f_{k}(t, \vec w, \vec w)$. We define the equation about $\rho_{I, k}, \vec{u}_{I, k}$ as follows
\begin{equation}\label{inilayer-k1}\left\{
\begin{split}
&\partial_{\tau}\rho_{I, k}=G_{k}^{1},\partial_{\tau}\vec{u}_{I, k}=G_{k}^{2},\\
&(\rho_{I, k}(0, \vec x), \vec u_{I, k}(0, \vec x))=(-\rho_{k}^{0}(\vec x), -\vec u_{k}^{0}(\vec x)),\\
&\lim_{\tau\rightarrow\infty}(\rho_{I, k}(\tau, \vec x), \vec{u}_{I, k}(\tau, \vec x))=(0, 0),\\
\end{split}\right.
\end{equation}
where
\begin{equation}\nonumber
\begin{aligned}
G_{k}^{1}=-&\sum_{\substack{
  i+j=k-2, \\
  0\leq i,j\leq k-2
}}\{(\vec{E}_{i}+\vec{u}_{I, i})\cdot\nabla_{x}(A_{j}+\rho_{I, j})-\vec{E}_{i}\cdot\nabla_{x}A_{j}\}\\-&\sum_{\substack{
  i+j=k-2, \\
  0\leq i,j\leq k-2
}}\{(A_{j}+\rho_{I, j})\mathrm{div}_{x}(\vec{E}_{i}+\vec{u}_{I, i})-A_{j}\mathrm{div}_{x}\vec{E}_{i}\}
\end{aligned}
\end{equation}
and
\begin{equation}\nonumber
\begin{aligned}
G_{k}^{2}=&-\frac{1}{\rho^{0}}\Big(\sum_{\substack{
  i+j=k-2, \\
  0\leq i,j\leq k-2
}}\{(A_{i}+\rho_{I, i})\partial_{\tau}(\vec{E}_{j}+\vec{u}_{I, j})-A_{i}\partial_{\tau}\vec{E}_{j}\}
-\sum_{\substack{
  i+j+l=k-2, \\
  0\leq i, j, l\leq k-2
}}\{(A_{i}+\rho_{I, i})(\vec{E}_{j}+\vec{u}_{I, j})\\&\cdot\nabla_{x}(\vec{E}_{l}+\vec{u}_{I, l})-A_{i}\vec{E}_{j}\cdot\nabla_{x}\vec{E}_{l}\}-\sum_{\substack{
  i+j=k-2, \\
  0\leq i,j\leq k-2
}}\{(A_{i}+\rho_{I, i})\nabla_{x}(D_{j}+\theta_{I, j})-A_{i}\nabla_{x}D_{j}\}\\&-\sum_{\substack{
  i+j=k-2, \\
  0\leq i,j\leq k-2
}}\{(D_{j}+\theta_{I, j})\cdot\nabla_{x}(A_{i}+\rho_{I, i})-D_{j}\cdot\nabla_{x}A_{i}\}+\langle\vec{w}f_{I, k-1}\rangle+\langle\vec{w}f_{I, k-3}\rangle\Big),
\end{aligned}
\end{equation}
combining \eqref{biaoda1}, \eqref{biaoda2} and the facts that $\vec{u}_{I, 0}=0, \rho_{I, 0}=0, A_{0}=\rho^{0}$ and $\partial_{\tau}\vec{B}_{0}=0$.

{\lem\label{kthinitiallayer}
The problem \eqref{inilayer-k1} has a unique solution $(\rho_{I, k}, \vec{u}_{I, k})\in (C^{1}([0, \infty); H_{\vec x}^{2N+4-k})\cap L^{1}([0, \infty); H_{\vec x}^{2N+4-k}))\times (C^{1}([0, \infty); H_{\vec x}^{2N+4-2k})\cap L^{1}([0, \infty); H_{\vec x}^{2N+4-2k}))$. Furthermore, there exists a suitably small constant $\sigma_{k1}>0$ such that
\begin{equation}\label{ineq-thetak-jia}
\|e^{\sigma_{k1}\tau}\rho_{I, k}\|_{L^{\infty}([0, \infty); H_{\vec x}^{2N+4-2k})}+\|e^{\sigma_{k1}\tau}\vec{u}_{I, k}\|_{L^{\infty}([0, \infty); H_{\vec x}^{2N+4-2k})}+\|\rho_{k}^{0}\|_{H_{\vec x}^{2N+4-2k}}+\|\vec{u}_{k}^{0}\|_{H_{\vec x}^{2N+4-2k}}\leq C.
\end{equation}}

\noindent\textbf{Proof.} Assume 
\begin{equation}\label{ineq-thetak-1}
\|e^{\sigma_{i1}\tau}\rho_{I, i}\|_{L^{\infty}([0, \infty); H_{\vec x}^{2N+4-2i})}+\|e^{\sigma_{i1}\tau}\vec{u}_{I, i}\|_{L^{\infty}([0, \infty); H_{\vec x}^{2N+4-2i})}+\|\rho_{i}^{0}\|_{H_{\vec x}^{2N+4-2i}}+\|\vec{u}_{i}^{0}\|_{H_{\vec x}^{2N+4-2i}}\leq C
\end{equation}
where $1\leq i\leq k-1$, $\sigma_{i1}>0$, suitably small.

Then, we can verify that $G_{k}^{1}\in C([0, \infty); H_{\vec x}^{2N+4-2k})$, $G_{k}^{2}\in C([0, \infty); H_{\vec x}^{2N+6-2k})$and
\begin{equation}\label{expo1}
\|e^{\sigma_{k1}\tau}G_{k}^{1}\|_{L^{\infty}([0, \infty); H_{\vec x}^{2N+4-k})}, \|e^{\sigma_{k1}\tau}G_{k}^{2}\|_{L^{\infty}([0, \infty); H_{\vec x}^{2N+6-2k})}\leq C,
\end{equation}
where $0<\sigma_{k}<\sigma_{k-1}$.

According to \eqref{inilayer-k1}, we have
\begin{equation}\label{111-1}
\rho_{I, k}=\int_{\infty}^{\tau}G_{k}^{1}\mathrm{d}s=O(e^{-\sigma_{k1}\tau}),\ \
\vec{u}_{I, k}=\int_{\infty}^{\tau}G_{k}^{2}\mathrm{d}s=O(e^{-\sigma_{k1}\tau}),
\end{equation}
which implies the existence of system \eqref{inilayer-k1}. Furthermore, we also have
\begin{equation}\label{111-1}
\rho_{k}^{0}=\int_{0}^{\infty}G_{k}^{1}\mathrm{d}s,\ \
\vec{u}_{k}^{0}=\int_{0}^{\infty}G_{k}^{2}\mathrm{d}s.
\end{equation}
Then, the estimate \eqref{ineq-thetak-jia} can be verified.

The kth-order initial layer $(f_{I, k}, \theta_{I, k})$, $k\geq 2$, is defined as
\begin{equation}\label{equations about f Ik,theta Ik}\left\{
\begin{split}
&\rho^{0}\partial_{\tau}\theta_{I, k}-\overline{f_{I, k}}+\sum_{\substack{
  i+j+l+m=k, \\
  i,j,l,m\geq 0
}}((D_{i}+\theta_{I, i})(D_{j}+\theta_{I, j})(D_{l}+\theta_{I, l})(D_{m}+\theta_{I, m})-D_{i}D_{j}D_{l}D_{m})
=G_{k}^{3},\\
&\partial_{\tau}f_{I, k}+f_{I, k}-\sum_{\substack{
  i+j+l+m=k, \\
  i,j,l,m\geq 0
}}((D_{i}+\theta_{I, i})(D_{j}+\theta_{I, j})(D_{l}+\theta_{I, l})(D_{m}+\theta_{I, m})-D_{i}D_{j}D_{l}D_{m}))=G_{k}^{4},\\
&f_{I, k}(0, \vec{x}, \vec w)=-f_{k}^{0}(\vec x, \vec w),\theta_{I, k}(0, \vec{x})=-\theta_{k}^{0}(\vec x),\\
&\lim_{\tau\rightarrow\infty}f_{I, k}(\tau, \vec{x}, \vec{w})=0,\ \ \lim_{\tau\rightarrow\infty}\theta_{I, k}(\tau, \vec{x})=0,
\end{split}\right.
\end{equation}
where
\begin{equation}\nonumber
\begin{aligned}
G_{k}^{3}=&-((A_{k}+\rho_{I, k})\partial_{\tau}\theta_{I, 0}+\sum_{\substack{
  i+j=k, \\
  0\leq i, j\leq k-1
}}((A_{i}+\rho_{I, i})\partial_{\tau}(D_{j}+\theta_{I, j})-A_{i}\partial_{\tau}D_{j})+\sum_{\substack{
  i+j+l=k-2, \\
  0\leq i, j, l\leq k-2
}}((A_{i}\\&+\rho_{I, i})(\vec{E}_{j}+\vec{u}_{I, j})\cdot\nabla_{x}(D_{l}+\theta_{I, l})-A_{i}\vec{E}_{j}\cdot\nabla_{x}D_{l})
+\sum_{\substack{
  i+j+l=k-2, \\
  0\leq i, j, l\leq k-2
}}((A_{i}+\rho_{I, i})(D_{l}+\theta_{I, l})\\&\mathrm{div}_{x}(\vec{E}_{j}+\vec{u}_{I, j})-A_{i}D_{l}\mathrm{div}_{x}\vec{E}_{j})
+\sum_{\substack{
  i+j=k-1, \\
  0\leq i, j\leq k-1
}}(\langle\vec{w}(J_{i}+f_{I, i})\rangle\cdot(\vec{E}_{j}+\vec{u}_{I, j})-\langle\vec{w}J_{i}\rangle\cdot\vec{E}_{j})
\\&+\sum_{\substack{
  i+j=k-3, \\
  0\leq i, j\leq k-3
}}(\langle\vec{w}(J_{i}+f_{I, i})\rangle\cdot(\vec{E}_{j}+\vec{u}_{I, j})-\langle\vec{w}J_{i}\rangle\cdot\vec{E}_{j}))
\end{aligned}
\end{equation}
and
\begin{equation}\nonumber
\begin{aligned}
G_{k}^{4}=&-(\vec{w}\cdot\nabla_{x}f_{I, k-1}+f_{I, k-2}-\overline{f_{I, k-2}}),
\end{aligned}
\end{equation}
combining \eqref{biaoda3}, \eqref{biaoda4} and the fact that $\rho_{I, 0}=0$ and $A_{0}=\rho^{0}$.

{\lem\label{kthinterior}
The problem \eqref{equations about f Ik,theta Ik} has a unique solution $(f_{I, k}, \theta_{I, k})\in (C^{1}([0, \infty); L_{\vec w}^{\infty}H_{\vec x}^{2N+4-2k}))\\\cap L^{1}([0, \infty); L_{\vec w}^{\infty}H_{\vec x}^{2N+4-2k}))\times (C^{1}([0, \infty); H_{\vec x}^{2N+4-2k}))\cap L^{1}([0, \infty); H_{\vec x}^{2N+4-2k}))$. Furthermore, there exists a suitably small constant $\sigma_{k2}>0$ such that
\begin{equation}\label{ineq-thetak0}
\|\theta_{k}^{0}(\vec x)\|_{H_{\vec x}^{2N+4-2k}}+\|f_{k}^{0}(\vec x, \vec w)\|_{L_{\vec w}^{\infty}H_{\vec x}^{2N+4-2k}}\leq C
\end{equation}
and
\begin{equation}\label{ineq-thetak1}
\|e^{\sigma_{k2}\tau}f_{I, k}\|_{L^{\infty}([0, \infty); L_{\vec w}^{\infty}H_{\vec x}^{2N+4-2k})}+\|e^{\sigma_{k2}\tau}\theta_{I, k}\|_{L^{\infty}([0, \infty); H_{\vec x}^{2N+4-2k})}\leq C,
\end{equation}
where $\sigma_{k}>0$, suitably small.}

\noindent\textbf{Proof.} Assume 
\begin{equation}\label{ineq-thetak-1-jia}
\begin{aligned}
&\|e^{\sigma_{i2}\tau}(\rho_{I, i}, \vec{u}_{I, i}, \theta_{I, i})\|_{L^{\infty}([0, \infty); H_{\vec x}^{2N+4-2i})}+\|e^{\sigma_{i2}\tau}f_{I, i}\|_{L^{\infty}([0, \infty); L_{\vec w}^{\infty}H_{\vec x}^{2N+4-2i})}\\&+\|\rho_{i}^{0}, \vec{u}_{i}^{0}, \theta_{i}^{0}\|_{H_{\vec x}^{2N+4-2i}}+\|f_{i}^{0}\|_{L_{\vec w}^{\infty}H_{\vec x}^{2N+4-2i}}\leq C,
\end{aligned}
\end{equation}
where $2\leq i\leq k-1$, $\sigma_{i2}>0$, suitably small. It is worthy to note that, we can derive the existence of the second order initial layers and estimates 
\begin{equation}\label{ineq-thetak-1-jia}
\begin{aligned}
&\|e^{\sigma_{22}\tau}(\rho_{I, 2}, \vec{u}_{I, 2}, \theta_{I, 2})\|_{L^{\infty}([0, \infty); H_{\vec x}^{2N})}+\|e^{\sigma_{22}\tau}f_{I, 2}\|_{L^{\infty}([0, \infty); L_{\vec w}^{\infty}H_{\vec x}^{2N})}\\&+\|\rho_{2}^{0}, \vec{u}_{2}^{0}, \theta_{2}^{0}\|_{H_{\vec x}^{2N}}+\|f_{2}^{0}\|_{L_{\vec w}^{\infty}H_{\vec x}^{2N}}\leq C,
\end{aligned}
\end{equation}
by similar methods as lemma \ref{rhoulayer1} and \ref{firstinitiallayer}. Then, we can verify that $G_{k}^{3}\in C([0, \infty); H_{\vec x}^{2N+4-2k})$, $G_{k}^{4}\in C([0, \infty); L_{\vec w}^{\infty}H_{\vec x}^{2N+4-2k})$, and
\begin{equation}\label{expo1-jia}
\|e^{\sigma_{k}\tau}G_{k}^{3}\|_{L^{\infty}([0, \infty); H_{\vec x}^{2N+4-2k})}+\|e^{\sigma_{k}\tau}G_{k}^{4}\|_{L^{\infty}([0, \infty); L_{\vec w}^{\infty}H_{\vec x}^{2N+4-2k})}\leq C,
\end{equation}
where $0<\sigma_{k2}<\sigma_{k-1,2}$.

From the equations $\eqref{equations about f Ik,theta Ik}_{1}$ and $\eqref{equations about f Ik,theta Ik}_{2}$, we can derive 
\begin{equation}\nonumber
\partial_{\tau}(\overline {f_{I, k}}+\rho^{0}\theta_{I, k})=G_{k}^{3}+\overline{G_{k}^{4}}.
\end{equation}
Then, 
\begin{equation}\nonumber
\overline {f_{I, k}}(\tau, \vec x)+\rho^{0}(\vec x)\theta_{I, k}(\tau, \vec x)=-\overline{f_{k}^{0}}(\vec x)-\rho^{0}(\vec x)\theta_{k}^{0}(\vec x)+\int_{0}^{\tau}(G_{k}^{3}+\overline{G_{k}^{4}})(s)\mathrm{d}s.
\end{equation}
By the fact that $f_{I, 1}, \theta_{I, 1}\rightarrow 0$ as $\tau\rightarrow\infty$, we can obtain
\begin{equation}\label{daoshuweiling-1}
\overline {f_{I, k}}+\rho^{0}\theta_{I, k}\equiv 0,
\end{equation}
which further implies
$\overline{f_{k}^{0}}(\vec x)+\rho^{0}(\vec x)\theta_{k}^{0}(\vec x)=\int_{0}^{\infty}(G_{k}^{3}+\overline{G_{k}^{4}})(s)\mathrm{d}s:=l_{k}(\vec x).$
According to \eqref{fk}, we can derive the following equality
\begin{equation}\nonumber
\begin{aligned}
4(\theta_{0}^{0})^{3}\theta_{k}^{0}+\rho^{0}\theta_{k}^{0}=&l_{k}+\langle\partial_{t}f_{k-2}|_{t=0}+\vec{w}\cdot\nabla_{x}f_{k-1}^{0}+f_{k-2}^{0}-\overline{f_{k-2}^{0}}\rangle-\Big(\sum_{\substack{
  i+j+l+m=k, \\
  0\leq i, j, l, m\leq k-1
}}\theta_{i}\theta_{j}\theta_{l}\theta_{m}\Big)|_{t=0}\\:=&l_{k}'(\vec x).
\end{aligned}
\end{equation}
Then, we can directly get 
\begin{equation}\nonumber
\theta_{k}^{0}(\vec x)=\frac{1}{4(\theta_{0}^{0})^{3}+\rho^{0}}l_{k}'(\vec x)\in H_{\vec x}^{2N+4-2k}.
\end{equation}
Simultaneously, we can get $f_{k}^{0}\in L^{\infty}_{\vec w}H_{\vec x}^{2N+4-2k}$ by \eqref{fk}. Furthermore, the estimate \eqref{ineq-thetak0} can be verified.

We can write $\eqref{equations about f Ik,theta Ik}_{1}$ in the following form
\begin{equation}\nonumber
\begin{aligned}
\frac{\partial \theta_{I, k}}{\partial\tau}+\Big(1+\frac{4}{\rho^{0}}(\theta_{0}^{0}+\theta_{I, 0})^{3}\Big)\theta_{I, k}=&\frac{1}{\rho^{0}}\Big(-\sum_{\substack{
  i+j+l+m=k, \\
 0\leq i,j,l,m\leq k-1
}}(D_{i}+\theta_{I, i})(D_{j}+\theta_{I, j})(D_{l}+\theta_{I, l})(D_{m}+\theta_{I, m})\\&+\sum_{\substack{
  i+j+l+m=k, \\
 i,j,l,m\geq 0
}}D_{i}D_{j}D_{l}D_{m}+G_{k}^{3}\Big):=\tilde{l}_{k}(\tau, \vec x).
\end{aligned}
\end{equation}
Then, we have
\begin{equation}\nonumber
\theta_{I, k}=e^{-\Big(1+\frac{4}{\rho^{0}}(\theta_{0}^{0}+\theta_{I, 0})^{3}\Big)\tau}(-\theta_{k}^{0})+\int_{0}^{\tau}e^{-\Big(1+\frac{4}{\rho^{0}}(\theta_{0}^{0}+\theta_{I, 0})^{3}\Big)(\tau-s)}\tilde{l}_{k}\mathrm{d}s,
\end{equation}
which implies the global existence of $\theta_{I, k}$ and $\theta_{I, k}\in C^{1}([0, \infty); H_{\vec x}^{2N+4-2k})$. 

According to $\eqref{equations about f Ik,theta Ik}_{1}$ and $\eqref{equations about f Ik,theta Ik}_{3}$, we can write $f_{I, k}$ in the following form
\begin{equation}\nonumber
\begin{aligned}
f_{I, k}=&-e^{-\tau}f_{k}^{0}+\int_{0}^{\tau}e^{s-\tau}\Big(\sum_{\substack{
  i+j+l+m=k, \\
  i,j,l,m\geq 0
}}((D_{i}+\theta_{I, i})(D_{j}+\theta_{I, j})(D_{l}+\theta_{I, l})(D_{m}+\theta_{I, m})\\&-D_{i}D_{j}D_{l}D_{m})\Big)+G_{k}^{4})\mathrm{d}s,
\end{aligned}
\end{equation}
which further implies the existence solution $f_{I, k}\in C^{1}([0, \infty); L^{\infty}_{\vec w}H_{\vec x}^{2N+4-2k}))$. \hfill$\Box$

Using the similar method to Lemma \ref{ini-0}, we can get the estimates \eqref{ineq-thetak0} and \eqref{ineq-thetak1}.

Collecting \eqref{rhok1}-\eqref{Fk-1biaoda}, we write the equations about $\rho_{k}, \vec{u}_{k}, \theta_{k}$ and $f_{k}$, $k\geq 2$, as follows
\begin{equation}\label{equations about theta k overline f epsilon k}\left\{
\begin{split}
&f_{k}=B_{k}(\theta_{0};\theta_{k})-(\partial_{t}f_{k-2}+\vec{w}\cdot\nabla_{x}f_{k-1}
+f_{k-2}-\overline{f_{k-2}})\\
&\partial_{t}\rho_{k}+\vec{u}_{k}\cdot\nabla_{x}\rho_{0}+\vec{u}_{0}\cdot\nabla_{x}\rho_{k}+\rho_{0}\mathrm{div}\vec{u}_{k}
+\rho_{k}\mathrm{div}\vec{u}_{0}=F_{k-1}^{1},\ \ \mathrm{in}\ \ (0, T]\times\mathbb{T}^{3},\\
&\rho_{0}\partial_{t}\vec{u}_{k}+\rho_{0}\vec{u}_{0}\cdot\nabla_{x}\vec{u}_{k}
+\rho_{0}\vec{u}_{k}\cdot\nabla_{x}\vec{u}_{0}+\Big(\rho_{0}+\frac{1}{3}\theta_{0}^{3}\Big)\nabla_{x}\theta_{k}+\rho_{k}\nabla_{x}\theta_{0}\\&+\theta_{0}\nabla_{x}\rho_{k}
+\theta_{k}\nabla_{x}\Big(\rho_{0}+\frac{1}{3}\theta_{0}^{3}\Big)
+\rho_{k}\partial_{t}\vec{u}_{0}=F_{k-1}^{2},\ \ \mathrm{in}\ \ (0, T]\times\mathbb{T}^{3},\\
&(\rho_{0}+\theta_{0}^{3})\partial_{t}\theta_{k}+\theta_{k}\partial_{t}(\theta_{0}^{3})+\rho_{k}\partial_{t}\theta_{0}+\rho_{0}\vec{u}_{0}\cdot\nabla_{x}\theta_{k}
+\rho_{0}\vec{u}_{k}\cdot\nabla_{x}\theta_{0}+\rho_{k}\vec{u}_{0}\cdot\nabla_{x}\theta_{0}+\rho_{0}\theta_{0}\mathrm{div}_{x}\vec{u}_{k}\\&+\rho_{k}\theta_{0}\mathrm{div}_{x}\vec{u}_{0}+\rho_{0}\theta_{k}\mathrm{div}_{x}\vec{u}_{0}
-4\theta_{0}^{3}\Delta_{x}\theta_{k}-12\theta_{0}^{2}\nabla_{x}\theta_{0}\cdot\nabla_{x}\theta_{k}-4\theta_{k}\Delta_{x}\theta_{0}^{3}-\frac{(\theta_{0})^{3}}{3}\nabla_{x}\theta_{k}\cdot\vec{u}_{0}
\\&-\theta_{k}\theta_{0}^{2}\nabla_{x}\theta_{0}\cdot\vec{u}_{0}+\frac{4\theta_{0}^{3}}{3}\nabla_{x}\theta_{0}\cdot\vec{u}_{k}=F_{k-1}^{3},\ \ \mathrm{in}\ \ (0, T]\times\mathbb{T}^{3},\\
&\rho_{k}(0, \vec{x})=\rho_{k}^{0}(\vec x),\;\;\vec{u}_{k}(0, \vec{x})=\vec u_{k}^{0}(\vec x),\;\;\theta_{k}(0, \vec{x})=\theta_{2}^{0}(\vec x),\;\;f_{k}(0, \vec{x}, \vec w)=f_{k}^{0},\ \ \mathrm{in}  \ \ \mathbb{S}^{2}\times\mathbb{T}^{3}.
\end{split}\right.
\end{equation}
\begin{lem}\label{kth}
The problem \eqref{equations about theta k overline f epsilon k} has a unique solution on $[0, T]$, and $(\rho_{k}, \vec u_{k}, \theta_{k}, f_{k})\in C^{0}([0, T]; H_{\vec x}^{2N+4-2k})\times C([0, T]; H_{\vec x}^{2N+4-2k})\times C([0, T]; H_{\vec x}^{2N+4-2k})\cap L^{2}(0, T; H_{\vec x}^{2N+5-2k})\times C([0, T]; \\L_{\vec w}^{\infty}H_{\vec x}^{2N+4-2k})\cap L^{2}(0, T; L_{\vec w}^{\infty}H_{\vec x}^{2N+5-2k})$, where $k\geq 2$. Furthermore, we have
\begin{equation}\label{ineq-thetak}
\begin{aligned}
&\|(\rho_{k}, \vec u_{k}, \theta_{k})\|_{L^{\infty}(0, T; H_{\vec x}^{2N+4-2k})}+\|\theta_{k}\|_{L^{2}(0, T; H_{\vec x}^{2N+5-2k}))}+\|f_{k}\|_{L^{\infty}(0, T; L_{\vec w}^{\infty}H_{\vec x}^{2N+4-2k})}\\&+\|f_{k}\|_{L^{2}(0, T; L_{\vec w}^{\infty}H_{\vec x}^{2N+5-2k})}\leq C.
\end{aligned}
\end{equation}
\end{lem}
The proof process is similar to Lemma \ref{firstinterior}, so, we omit the details here.

\section{Diffusive Limit}\label{si}

In this section, we prove Theorem \ref{mainthm} by estimating the difference between the solution $(\rho^{\epsilon}, \vec{u}^{\epsilon}, \theta^{\epsilon}, f^{\epsilon})$ to system \eqref{research equations} and the constructed approximate solution $(\rho^{N}, \vec{u}^{N}, \theta^{N}, f^{N})$ where 
\begin{equation}\label{expansion of rho u epsilon}
\rho^{N}=\sum_{k=0}^{N}\epsilon^{k}(\rho_{k}+\rho_{I, k}), \ \ \vec{u}^{N}=\sum_{k=0}^{N}\epsilon^{k}(\vec{u}_{k}+\vec{u}_{I, k}), \ \ \theta^{N}=\sum_{k=0}^{N}\epsilon^{k}(\theta_{k}+\theta_{I, k})
\end{equation}
and
\begin{equation}\label{expansion of f theta epsilon}
f^{N}=\sum_{k=0}^{N}\epsilon^{k}(f_{k}+f_{I, k}).
\end{equation} 

The remainder can be defined as
\begin{equation}\nonumber
\begin{split}
\rho_{r}=\rho^{\epsilon}-\rho^{N},\ \ \vec{u}_{r}=\vec{u}^{\epsilon}-\vec{u}^{N},\ \ \theta_{r}=\theta^{\epsilon}-\theta^{N},\ \ f_{r}=f^{\epsilon}-f^{N},
\end{split}
\end{equation}
functions $(\rho_{r}, \vec{u}_{r}, f_{r}, \theta_{r})$ then satisfy
\begin{equation}\label{reeq1}
\partial_{t}\rho_{r}+\vec{u}_{r}\cdot\nabla_{x}\rho^{N}+\vec{u}^{\epsilon}\cdot\nabla_{x}\rho_{r}+\rho_{r}\mathrm{div}\vec{u}^{N}+\rho^{\epsilon}\mathrm{div}\vec{u}_{r}=-\mathcal{L}_{1}(\rho^{N}, \vec{u}^{N}),
\end{equation}
\begin{equation}\label{reeq2}
\begin{aligned}
&\epsilon^{2}\rho_{r}\partial_{t}\vec{u}^{N}+\epsilon^{2}\rho^{\epsilon}\partial_{t}\vec{u}_{r}+\epsilon^{2}(\rho^{\epsilon}\vec{u}^{\epsilon}\cdot\nabla_{x}\vec{u}_{r}
+\rho_{r}\vec{u}^{\epsilon}\nabla_{x}\vec{u}^{N}+\rho^{N}\vec{u}_{r}\cdot\nabla_{x}\vec{u}^{N})
+\epsilon^{2}(\rho_{r}\nabla_{x}\theta^{N}\\&+\rho^{\epsilon}\nabla_{x}\theta_{r})+\epsilon^{2}(\theta_{r}\nabla_{x}\rho^{N}+\theta^{\epsilon}\nabla_{x}\rho_{r})
-\langle(\epsilon+\epsilon^{3})\vec w(f_{r}-\overline{f_{r}})\rangle=-\mathcal{L}_{2}(\rho^{N}, \vec{u}^{N}, \theta^{N}, f^{N}),
\end{aligned}
\end{equation}
\begin{equation}\label{reeq3}
\begin{aligned}
&\epsilon^{2}(\rho^{\epsilon}\partial_{t}\theta_{r}+\rho_{r}\partial_{t}\theta^{N})+\epsilon^{2}(\rho^{\epsilon}\vec{u}^{\epsilon}\cdot\nabla_{x}\theta_{r}+\rho_{r}\vec{u}^{\epsilon}\cdot\nabla_{x}\theta^{N}+\rho^{N}\vec{u}_{r}\cdot\nabla_{x}\theta^{N})
+\epsilon^{2}(\rho^{\epsilon}\theta^{\epsilon}\mathrm{div}\vec {u}_{r}\\&+\rho_{r}\theta^{N}\mathrm{div}\vec {u}^{N}+\rho^{N}\theta_{r}\mathrm{div}\vec {u}^{N})+(\theta^{N}+\theta_{r})^{4}-(\theta^{N})^{4}-\overline{f_{r}}=-\mathcal{L}_{3}(\rho^{N}, \vec{u}^{N}, \theta^{N}, f^{N})
\end{aligned}
\end{equation}
and
\begin{equation}\label{reeq4}
\epsilon^{2}\partial_{t}f_{r}+\epsilon\vec{w}\cdot\nabla_{x}f_{r}+\epsilon^{2}(f_{r}-\overline{f_{r}})+f_{r}
-(\theta^{N}+\theta_{r})^{4}+(\theta^{N})^{4}=-\mathcal{L}_{4}(f^{N}, \theta^{N}),
\end{equation}
with initial conditions
\begin{equation}\label{rein}
\rho_{r}(0, \vec x)=0,\ \ \vec{u}_{r}(0, \vec x)=0,\ \ \theta_{r}(0, \vec x)=0,\ \ f_{r}(0, \vec x, \vec w)=0,\ \ \mathrm{for}\ \ (\vec x, \vec w)\in\mathbb{T}^{3}\times\mathbb{S}^{2}.
\end{equation}

{\thm\label{errorestimates}
Assume $N\geq 4$ and $\sigma=\min_{0\leq k\leq N, i=1, 2}\sigma_{ki}$ where $\sigma_{0i}=\sigma_{0}$. The composite approximate solution $(\rho^{N}, \vec{u}^{N}, \theta^{N}, f^{N})$ constructed in subsection \ref{construction}, satisfies \eqref{reeq1}-\eqref{reeq2} with initial conditions \eqref{rein}. Moreover, the error terms $\mathcal{L}_{1}(\rho^{N}, \vec{u}^{n}), \mathcal{L}_{2}(\rho^{N}, \vec{u}^{N}, \theta^{N}, f^{N}), \mathcal{L}_{3}(\rho^{N}, \vec{u}^{N}, \\\theta^{N}, f^{N})$ and $\mathcal{L}_{4}(\theta^{N}, f^{N})$ satisfy
\begin{equation}\label{r12}
\|\mathcal{L}_{1}(\rho^{N}, \vec{u}^{N})\|_{L_{T}^{\infty}L_{\vec w}^{\infty}H_{\vec x}^{3}}, \|\mathcal{L}_{2}(\rho^{N}, \vec{u}^{N}, \theta^{N}, f^{N})\|_{L_{T}^{\infty}L_{\vec w}^{\infty}H_{\vec x}^{3}}\leq C\epsilon^{N+1}
\end{equation}
and
\begin{equation}\label{r34}
\|\mathcal{L}_{3}(\rho^{N}, \vec{u}^{N}, \theta^{N}, f^{N})\|_{L_{T}^{\infty}L_{\vec w}^{\infty}H_{\vec x}^{3}}, \|\mathcal{L}_{4}(\theta^{N}, f^{N})\|_{L_{T}^{\infty}L_{\vec w}^{\infty}H_{\vec x}^{3}}\leq C\epsilon^{N+1},
\end{equation}
where $C>0$ is a positive constant independent of $\epsilon$.}

\noindent\textbf{Proof.} We first consider $\mathcal{L}_{1}(\rho^{N}, \vec{u}^{N})$. From \eqref{resi1-jia-ini}, we have
\begin{equation}\label{esl1}
\begin{aligned}
&\mathcal{L}_{1}(\rho^{N}, \vec{u}^{N})\\=&\mathcal{L}_{1}\Big(\sum_{k=0}^{N}\epsilon^{k}\rho_{k}, \sum_{k=0}^{N}\epsilon^{k}\vec{u}_{k}\Big)+E_{r1}+O(\epsilon^{N+1})
:=\tilde{\mathcal{L}}_{1}+E_{r1},
\end{aligned}
\end{equation}
where
\begin{equation}\nonumber
\begin{aligned}
\tilde{\mathcal{L}}_{1}=\mathcal{L}_{1}\Big(\sum_{k=0}^{N}\epsilon^{k}\rho_{k}, \sum_{k=0}^{N}\epsilon^{k}\vec{u}_{k}\Big)+O(\epsilon^{N+1}),
\end{aligned}
\end{equation}
and $E_{r1}$ defined in \eqref{error1} and we rewrite it for clarity as follows
\begin{equation}\nonumber
\begin{aligned}
E_{r1}=&\sum_{k=0}^{N}\epsilon^{k}\Big(\sum_{\substack{
  i+j=k-2, \\
  0\leq i,j\leq k-2
}}\{(\vec{u}_{i}+\vec{u}_{I, i})\cdot\nabla_{x}(\rho_{j}+\rho_{I, j})-\vec{u}_{i}\cdot\nabla_{x}\rho_{j}\}
+\sum_{\substack{
  i+j=k-2, \\
  0\leq i,j\leq k-2
}}\{(\rho_{j}+\rho_{I, j})\mathrm{div}_{x}(\vec{u}_{i}\\&+\vec{u}_{I, i})-\rho_{j}\mathrm{div}_{x}\vec{u}_{i}\}
-\sum_{\substack{
  i+j=k-2, \\
  0\leq i,j\leq k-2
}}\{(\vec{E}_{i}+\vec{u}_{I, i})\cdot\nabla_{x}(A_{j}+\rho_{I, j})-\vec{E}_{i}\cdot\nabla_{x}A_{j}\}-\sum_{\substack{
  i+j=k-2, \\
  0\leq i,j\leq k-2
}}\{(A_{j}\\&+\rho_{I, j})\mathrm{div}_{x}(\vec{E}_{i}+\vec{u}_{I, i})-A_{j}\mathrm{div}_{x}\vec{E}_{i}\}\Big).
\end{aligned}
\end{equation}
Collecting Lemma \ref{ini-0}-\ref{kth}, we have
\begin{equation}\label{eq1}
\|\tilde{\mathcal{L}}_{1}\|_{L_{T}^{\infty}L_{\vec w}^{\infty}H_{\vec x}^{3}}\leq C\epsilon^{N+1}.
\end{equation}

For $\mathcal{L}_{12}$, after careful calculation, we have
\begin{equation}\nonumber
\begin{aligned}
E_{r1}=&\sum_{k=0}^{N}\epsilon^{k}\sum_{\substack{
  i+j=k-2, \\
  0\leq i,j\leq k-2
}}\{(\vec{u}_{i}-\vec{E}_{i})\cdot\nabla_{x}\rho_{I, j}+\vec{u}_{I, i}\cdot\nabla_{x}(\rho_{j}-A_{j})+(\rho_{j}-A_{j})\mathrm{div}_{x}\vec{u}_{I, i}\\&+\rho_{I, j}\mathrm{div}_{x}(\vec{u}_{i}-\vec{E}_{i})\}:=E_{r11}+E_{r12}+E_{r13}+E_{r14}.
\end{aligned}
\end{equation}
We can write $E_{r11}$ as 
\begin{equation}\nonumber
E_{r11}=\sum_{k=0}^{N}\epsilon^{k-i}\sum_{i=0}^{k-2}\epsilon^{i}(\vec{u}_{i}-\vec{E}_{i})\cdot\nabla_{x}\rho_{I, k-2-i}.
\end{equation}
Since $\nabla_{x}\rho_{I, k-2-i}=0$, for $i=k-1, k$, we can further write $\mathcal{L}_{121}$ as 
\begin{equation}
\begin{aligned}
E_{r11}=&\sum_{k=0}^{N}\epsilon^{k-i}\sum_{i=0}^{k}\epsilon^{i}(\vec{u}_{i}-\vec{E}_{i})\cdot\nabla_{x}\rho_{I, k-2-i}
=\sum_{k=i}^{N}\epsilon^{k-i}\sum_{i=0}^{N}\epsilon^{i}(\vec{u}_{i}-\vec{E}_{i})\cdot\nabla_{x}\rho_{I, k-2-i}.
\end{aligned}
\end{equation}
Due to the exponential decay estimates \eqref{ineq-theta00-jia1}, \eqref{ineq-theta01} and \eqref{ineq-thetak-jia}, we have
\begin{equation}\label{ex}
\|\nabla_{x}\rho_{I, k-2-i}\|_{L_{T}^{\infty}H_{\vec x}^{3}}\leq Ce^{-\frac{\sigma t}{\epsilon^{2}}}.
\end{equation}
Taylor's formula yields
\begin{equation}\label{rela-1}
\vec{u}_{p}(t, \vec x)=\sum_{q=0}^{N-p}\frac{t^{q}}{q!}\partial_{t}^{q}\vec{u}_{q}(0, \vec x)+\frac{\partial_{t}^{N-p+1}\vec{u}_{q}(t', \vec x)}{(N-p+1)!}t^{N-p+1},
\end{equation}
with $t'\in [0, t]$.
Using the above formula and \eqref{formula1}, we get
\begin{equation}\label{relation}
\begin{aligned}
\sum_{p=0}^{N}\epsilon^{p}(\vec{u}_{p}-\vec{E}_{p})=&\sum_{p=0}^{N}\epsilon^{p}\Big(\sum_{q=0}^{N-p}\frac{t^{q}}{q!}\partial_{t}^{q}\vec{u}_{q}(0, \vec x)+\frac{\partial_{t}^{N-p+1}\vec{u}_{p}(t', \vec x)}{(N-p+1)!}t^{N-p+1}\Big)\\&-\sum_{p=0}^{N}\epsilon^{p}\sum_{q=0}^{p}\epsilon^{q}\frac{\tau^{q}}{q!}\partial_{t}^{q}\vec{u}_{p-q}(0, \vec x).
\end{aligned}
\end{equation}
Using the formula
\begin{equation}\nonumber
\begin{aligned}
\sum_{p=0}^{N}\sum_{q=0}^{p}g(q, p)=\sum_{q=0}^{N}\sum_{p=q}^{N}g(q, p)=\sum_{q=0}^{N}\sum_{s=0}^{N-q}g(q, q+s)=\sum_{p=0}^{N}\sum_{q=0}^{N-p}g(p, p+q),
\end{aligned}
\end{equation}
where we have taken substitutions $s=p-q$ and $q\rightarrow p, s\rightarrow q$ in the above second and third equality. Then, taking $g(p, p+q)=\epsilon^{p}\frac{t^{q}}{q!}\partial_{t}^{q}\vec{u}_{p}(0, \vec x)$, we get
\begin{equation}\nonumber
g(q, p)=\epsilon^{q}\frac{t^{p-q}}{(p-q)!}\partial_{t}^{p-q}\vec{u}_{q}(0, \vec x)=\epsilon^{2p-q}\frac{\tau^{p-q}}{(k-q)!}\partial_{t}^{p-q}\vec{u}_{q}(0, \vec x)
\end{equation}
and so
\begin{equation}
\begin{aligned}
\sum_{p=0}^{N}\sum_{q=0}^{N-p}\epsilon^{p}\frac{t^{q}}{q!}\partial_{t}^{q}\vec{u}_{p}(0, \vec x)=\sum_{p=0}^{N}\sum_{q=0}^{N-p}\epsilon^{2p-q}\frac{\tau^{p-q}}{(p-q)!}\partial_{t}^{p-q}\vec{u}_{q}(0, \vec x).
\end{aligned}
\end{equation}
By substitutions $p\rightarrow p, p-q\rightarrow q$, we have
\begin{equation}\nonumber
\begin{aligned}
\sum_{p=0}^{N}\sum_{q=0}^{N-p}\epsilon^{2p-q}\frac{\tau^{p-q}}{(p-q)!}\partial_{t}^{p-q}\vec{u}_{q}(0, \vec x)=\sum_{p=0}^{N}\sum_{q=0}^{N}\epsilon^{p+q}\frac{\tau^{q}}{q!}\partial_{t}^{q}\vec{u}_{p-q}(0, \vec x).
\end{aligned}
\end{equation}
Taking this relation and \eqref{rela-1} into \eqref{relation} leads to
\begin{equation}\label{relation1}
\begin{aligned}
\sum_{i=0}^{N}\epsilon^{i}(\vec{u}_{i}-\vec{E}_{i})=&\sum_{i=0}^{N}\epsilon^{i}\frac{\partial_{t}^{N-i+1}\vec{u}_{i}(t', \vec x)}{(N-i+1)!}t^{N-i+1}.
\end{aligned}
\end{equation}
Combining \eqref{relation1} with \eqref{relation}, \eqref{ex} satisfies
\begin{equation}\nonumber
\begin{aligned}
\|E_{r11}\|_{L_{T}^{\infty}H_{\vec x}^{3}}\leq C\sum_{i=0}^{N}\epsilon^{i}\frac{1}{(N-i+1)!}t^{N-i+1}e^{-\frac{\sigma t}{\epsilon^{2}}},
\end{aligned}
\end{equation}
where we have used the fact that $\|\vec{u}_{0}\|_{C^{0}([0, T]; H_{\vec x}^{2N+2})},\|\vec{u}_{1}\|_{C^{0}([0, T]; H_{\vec x}^{2N+1})},\|\vec{u}_{i}\|_{C^{0}([0, T]; H_{\vec x}^{2N+4-2i})}\\\leq C$, $2\leq i\leq N$. Note that the function $h(t)=t^{N-i+1}e^{-\frac{\sigma t}{\epsilon^{2}}}$ attain its maximum at $t^{\star}=\frac{(N-i+1)\epsilon^{2}}{\sigma}$ with the maximum value $h(t^{\star})=(N-i+1)^{N-i+1}/\sigma\cdot e^{-(N-i+1)}$. Therefore,
\begin{equation}\nonumber
\begin{aligned}
\|E_{r11}\|_{L_{T}^{\infty}H_{\vec x}^{3}}\leq C\sum_{i=0}^{N}\epsilon^{2N+2-i}\frac{(N-i+1)^{N-i+1}}{\sigma^{N-i+1}(N-i+1)!}e^{-(N-i+1)}\leq C\epsilon^{N+2}(\gamma_{N}-1),
\end{aligned}
\end{equation}
where $\gamma_{N}:=\sum_{n=0}^{N+1}n^{n}/(\sigma^{n}n!)>1$ is a constant depending on $N$. Therefore
\begin{equation}\label{eq2}
\|E_{r11}\|_{L_{T}^{\infty}H_{\vec x}^{3}}\leq C\epsilon^{N+2}.
\end{equation}
Similarly, we can get
\begin{equation}\label{eq2-jia}
\|E_{r12}, E_{r13}, E_{r14}\|_{L_{T}^{\infty}H_{\vec x}^{3}}\leq C\epsilon^{N+2}.
\end{equation}
Collecting \eqref{eq1}, \eqref{eq2} and \eqref{eq2-jia}, we can derive
\begin{equation}\nonumber
\|\mathcal{L}_{1}\|_{L_{T}^{\infty}L_{\vec w}^{\infty}H_{\vec x}^{3}}\leq C\epsilon^{N+1}.
\end{equation}

Similarly, we can derive
\begin{equation}\label{omit1}
\|\mathcal{L}_{2}, \mathcal{L}_{3}, \mathcal{L}_{4}\|_{L_{T}^{\infty}L_{\vec w}^{\infty}H_{\vec x}^{3}}\leq C\epsilon^{N+1},
\end{equation}
where we omit the details of the calculations. It should be noted that the key step in proving the estimates \eqref{omit1} is furnished by the exponential decay estimates for the initial layers, together with the interior solution estimates obtained in Lemmas \ref{ini-0}-\ref{kth}.

The proof of Theorem \ref{mainthm} proceeds as follows. We first derive appropriate estimates for the associated linearized system, and then invoke the Banach fixed-point theorem to establish the existence of solutions in a neighborhood of zero. This leads to the desired convergence of $(\rho^{\epsilon}, \vec{u}^{\epsilon}, \theta^{\epsilon}, f^{\epsilon})$ to $(\rho^{N}, \vec{u}^{N}, \theta^{N}, f^{N})$ as $\epsilon\rightarrow 0$.

\subsection{Uniform estimates for the linearized equations}\label{s4.1}
We shall use the method of linearization to obtain the existence and the uniform estimates about $\epsilon$ of equations \eqref{reeq1}-\eqref{reeq4}. Consider the following linearized equations:
\begin{equation}\label{reeq1-jiajia}
\partial_{t}\rho_{r}+\vec{u}_{r}\cdot\nabla_{x}\rho^{N}+\tilde{\vec{u}}\cdot\nabla_{x}\rho_{r}+\rho_{r}\mathrm{div}\vec{u}^{N}+\tilde\rho\mathrm{div}\vec{u}_{r}=R_{1},
\end{equation}
\begin{equation}\label{reeq2-jiajia}
\begin{aligned}
&\epsilon^{2}\rho_{r}\partial_{t}\vec{u}^{N}+\epsilon^{2}\tilde\rho\partial_{t}\vec{u}_{r}+\epsilon^{2}(\tilde\rho\tilde{\vec{u}}\cdot\nabla_{x}\vec{u}_{r}
+\rho_{r}\tilde{\vec{u}}\cdot\nabla_{x}\vec{u}^{N}+\rho^{N}\vec{u}_{r}\cdot\nabla_{x}\vec{u}^{N})
\\&+\epsilon^{2}(\rho_{r}\nabla_{x}\theta^{N}+\tilde\rho\nabla_{x}\theta_{r})+\epsilon^{2}(\theta_{r}\nabla_{x}\rho^{N}+\tilde\theta\nabla_{x}\rho_{r})-\langle(\epsilon+\epsilon^{3})\vec w(f_{r}-\overline{f_{r}})\rangle
=\vec{R}_{2},
\end{aligned}
\end{equation}
\begin{equation}\label{reeq3-jiajia}
\begin{aligned}
&\epsilon^{2}(\tilde\rho\partial_{t}\theta_{r}+\rho_{r}\partial_{t}\theta^{N})+\epsilon^{2}(\tilde\rho\tilde{\vec{u}}\cdot\nabla_{x}\theta_{r}+\rho_{r}\tilde{\vec{u}}\cdot\nabla_{x}\theta^{N}+\rho^{N}\vec{u}_{r}\cdot\nabla_{x}\theta^{N})
\\&+\epsilon^{2}(\tilde\rho\tilde\theta\mathrm{div}\vec {u}_{r}+\rho_{r}\theta^{N}\mathrm{div}\vec {u}^{N}+\rho^{N}\theta_{r}\mathrm{div}\vec {u}^{N})+4(\theta^{N})^{3}\theta_{r}-\overline{f_{r}}=R_{3}+\langle R\rangle
\end{aligned}
\end{equation}
and
\begin{equation}\label{reeq4-jiajia}
\epsilon^{2}\partial_{t}f_{r}+\epsilon\vec{w}\cdot\nabla_{x}f_{r}+\epsilon^{2}(f_{r}-\overline{f_{r}})+f_{r}
-4(\theta^{N})^{3}\theta_{r}=R_{4}-R,
\end{equation}
where $R_{i}=R_{i}(t, \vec x)$, $i=1, 3$, $\vec{R}_{2}=(R_{2}^{1}, R_{2}^{2}, R_{2}^{3})$, $R=R(t, \vec x, \vec w)$, $R_{4}=R_{4}(t, \vec x, \vec w)$ and $(\tilde{\rho_{r}}, \tilde{\vec{u}_{r}}, \tilde{\theta_{r}})$ are given functions, $(\tilde\rho, \tilde{\vec{u}}, \tilde\theta)=(\rho^{N}+\tilde{\rho_{r}}, \vec{u}^{N}+\tilde{\vec{u}_{r}}, \theta^{N}+\tilde{\theta_{r}})$ 
and the initial conditions are taken to be 
\begin{equation}\label{rein-jiajia}
f_{r}(0, \vec x, \vec w)=0,\ \ \rho_{r}(0, \vec x)=0,\ \ \vec{u}_{r}(0, \vec x)=0, \ \ \theta_{r}(0, \vec x)=0,\ \ \mathrm{for}\ \ (\vec x, \vec w)\in\mathbb{T}^{3}\times\mathbb{S}^{2}.
\end{equation}

\noindent\textbf{Step 1.} We rewrite \eqref{reeq1-jiajia}-\eqref{reeq4-jiajia} as follows
\begin{equation}\label{reeq1-jia-jia}
\partial_{t}V_{r}+\sum_{j=1}^{3}\tilde{E_{j}^{1}}\partial_{j}V_{r}+\tilde{E}^{2}V_{r}=\vec{\tilde{H}},
\end{equation}
\begin{equation}\label{reeq2-jia-jia}
\partial_{t}f_{r}+\frac{1}{\epsilon}\vec{w}\cdot\nabla_{x}f_{r}+\Big(1+\frac{1}{\epsilon^{2}}\Big)(f_{r}-\overline{f_{r}})
=\frac{1}{\epsilon^{2}}(R_{4}-R)+\frac{1}{\epsilon^{2}}(4(\theta^{N})^{3}\theta_{r}-\overline{f_{r}}).
\end{equation}
Here $V_{r}=(\rho_{r}, \vec{u}_{r}, \theta_{r})$, $\vec{\tilde{H}}=(\tilde{H_{1}}, \tilde{H_{2}}, \tilde{H_{3}}, \tilde{H_{4}}, \tilde{H_{5}})$, where
\begin{equation}\nonumber
\tilde{H_{1}}=R_{1},\ \ \tilde{H_{j}}=\frac{1}{\tilde{\rho}\epsilon^{2}}R_{2}^{j}+\frac{1}{\tilde\rho}\langle\Big(\epsilon+\frac{1}{\epsilon}\Big)w_{j-1}(f_{r}-\overline{f_{r}})\rangle,\ \ j=2,3,4,
\end{equation}

\begin{equation}\nonumber
\tilde{H_{5}}=\frac{1}{\tilde{\rho}\epsilon^{2}}(R_{3}+\langle R\rangle)+\frac{1}{\tilde{\rho}\epsilon^{2}}(4(\theta^{N})^{3}\theta_{r}-f_{r}).
\end{equation}
Furthermore, $\tilde{E_{j}^{1}}(\tilde V)=\{\tilde{a}_{mn}\}_{5\times5}$ where $V_{r}=(\rho_{r}, \vec{u}_{r}, \theta_{r}), \tilde V=(\tilde\rho, \tilde{\vec u}, \tilde\theta)$, $\tilde{a}_{ii}=\tilde{u}_{j}$, $\tilde{a}_{1(j+1)}=\tilde\rho$, $\tilde{a}_{(j+1)1}=\frac{\tilde\theta}{\tilde\rho}$, $\tilde{a}_{(j+1)5}=1$, $\tilde{a}_{5(j+1)}=\tilde\theta$ for $j=1, 2, 3,$ and the rest elements of $\tilde{a}_{ij}$ are set to be 0, and 
\begin{equation}\label{erjiejuzhen}
\tilde{E^{2}}
=
\left[
\begin{array}{ccccc}
\mathrm{div}_{x}\vec{u}^{N}&\partial_{x_{1}}\rho^{N}&\partial_{x_{2}}\rho^{N}&\partial_{x_{3}}\rho^{N}&0\\
\frac{\partial_{t}u_{1}^{N}}{\tilde\rho}+\frac{\tilde{\vec u}\cdot\nabla_{x}u_{1}^{N}}{\tilde\rho}+\frac{\partial_{x_{1}}\theta^{N}}{\tilde\rho}&\rho^{N}\partial_{x_{1}}u_{1}^{N}&\rho^{N}\partial_{x_{2}}u_{1}^{N}&\rho^{N}\partial_{x_{3}}u_{1}^{N}&1\\
\frac{\partial_{t}u_{2}^{N}}{\tilde\rho}+\frac{\tilde{\vec u}\cdot\nabla_{x}u_{2}^{N}}{\tilde\rho}+\frac{\partial_{x_{2}}\theta^{N}}{\tilde\rho}&\rho^{N}\partial_{x_{1}}u_{2}^{N}&\rho^{N}\partial_{x_{2}}u_{2}^{N}&\rho^{N}\partial_{x_{3}}u_{2}^{N}&1\\
\frac{\partial_{t}u_{3}^{N}}{\tilde\rho}+\frac{\tilde{\vec u}\cdot\nabla_{x}u_{3}^{N}}{\tilde\rho}+\frac{\partial_{x_{3}}\theta^{N}}{\tilde\rho}&\rho^{N}\partial_{x_{1}}u_{3}^{N}&\rho^{N}\partial_{x_{2}}u_{3}^{N}&\rho^{N}\partial_{x_{3}}u_{3}^{N}&1\\
\frac{\partial_{t}\theta^{N}}{\tilde\rho}+\frac{\theta^{N}\mathrm{div}\vec{u}^{N}}{\tilde\rho}
&\rho^{N}\partial_{x_{1}}\theta^{N}&\rho^{N}\partial_{x_{2}}\theta^{N}&\rho^{N}\partial_{x_{3}}\theta^{N}&\rho^{N}\mathrm{div}_{x}\vec{u}^{N}\\
\end{array}
\right],
\end{equation}
$\vec{u}^{N}=(u_{1}^{N}, u_{2}^{N}, u_{3}^{N})$.

First, we symmetrize \eqref{reeq1-jia-jia} by multiplying it with the symmetrizing matrix defined by
\begin{equation}\label{xishujuzhen-1}
E^{0}(\tilde{V})
=
\left[
\begin{array}{cccc}
(\tilde{\rho})^{-1}&0&0\\
0&\frac{\tilde{\rho}}{\tilde{\theta}}\mathbb{I}_{3\times3}&0\\
0&0&\frac{\tilde{\rho}}{(\tilde{\theta})^{2}}
\end{array}
\right].
\end{equation}
Therefore, we can further rewrite \eqref{reeq1-jiajia}-\eqref{reeq4-jiajia} as follows
\begin{equation}\label{research equations-3-2-jia}
E^{0}(\tilde{V})\frac{\partial {V}_{r}}{\partial t}+\sum_{j=1}^{3}E_{j}^{1}(\tilde{V})\frac{\partial {V}_{r}}{\partial x_{j}}+E^{2}(\tilde V)V_{r}=\vec H,
\end{equation}
where $E_{j}^{1}(\tilde V):=E^{0}(\tilde V)\tilde{E}_{j}^{1}(\tilde V)$ is symmetric, $E^{2}(\tilde V)=E^{0}(\tilde V)\tilde{E}^{2}(\tilde V)$ and
\begin{equation}\nonumber
\vec H=E^{0}(\tilde V)\vec{\tilde H}=(H_{1}, H_{2}, H_{3}, H_{4}, H_{5})^{t}
\end{equation}
with
\begin{equation}\nonumber
H_{1}=\frac{1}{\tilde\rho}R_{1},\ \ H_{j}=\frac{1}{\tilde{\theta}\epsilon^{2}}R_{2}^{j}+\frac{1}{\tilde\theta}\Big\langle\Big(\epsilon+\frac{1}{\epsilon}\Big)w_{j-1}(f_{r}-\overline{f_{r}})\Big\rangle,j=2, 3, 4
\end{equation}
\begin{equation}\nonumber
H_{5}=\frac{1}{(\tilde{\theta})^{2}\epsilon^{2}}(R_{3}+\langle R\rangle)+\frac{1}{(\tilde{\theta})^{2}\epsilon^{2}}(4(\theta^{N})^{3}\theta_{r}-\overline{f_{r}}).
\end{equation}
According to \eqref{rein-jiajia}, we impose 
\begin{equation}\label{rein-jiajia-1}
f_{r}(0, \vec x, \vec w)=0,\ \ \ V_{r}(0, \vec x)=0, \ \ \mathrm{for}\ \ (\vec x, \vec w)\in\mathbb{T}^{3}\times\mathbb{S}^{2}.
\end{equation}

\noindent\textbf{Step 2.} The well-posedness of the linearized equations \eqref{research equations-3-2-jia}, \eqref{reeq4-jiajia} and \eqref{rein-jiajia-1}.
{\lem\label{le4.2}
Let $\epsilon>0$ and $(\rho^{N}, \vec{u}^{N}, \theta^{N}, f^{N})$ be the composite approximate solution constructed in Section \ref{er}. Assume $\|\tilde{\rho}_{r}, \tilde{\vec{u}}_{r}, \tilde{\theta}_{r}\|_{L_{T}^{\infty}H_{\vec x}^{3}}\leq\epsilon^{s}$, $\|\partial_{t}\tilde{\rho}_{r}, \partial_{t}\tilde{\theta}_{r}\|_{L_{T}^{\infty}H_{\vec x}^{2}}\leq 1$, where $\epsilon$ is sufficiently small, and $s>0$ is a given constant. We also assume $R_{4}, R\in L_{T}^{2}L_{\vec w}^{2}H_{\vec x}^{3}$ and $R_{i}\in L_{T}^{2}H_{\vec x}^{3}$, $i=1, 2, 3$. Then, the linearized equations \eqref{research equations-3-2-jia}, \eqref{reeq4-jiajia} and \eqref{rein-jiajia-1} have a unique solution $(V_{r}, f_{r})\in C([0, T]; H_{\vec x}^{3})\times C([0, T]; L_{\vec w}^{2}H_{\vec x}^{3})$. Moreover, the solution $(\rho_{r}, \vec{u}_{r}, \theta_{r}, f_{r})$ satisfies the following estimates
\begin{equation}\label{bd00}
\begin{aligned}
&\|V_{r}\|_{L_{t}^{\infty}H_{\vec x}^{3}}^{2}+\|f_{r}\|_{L_{t}^{\infty}L_{\vec w}^{2}H_{\vec x}^{3}}^{2}+\frac{1}{\epsilon^{2}}\|f_{r}-\overline{f_{r}}\|_{L_{t}^{2}L_{\vec w}^{2}H_{\vec x}^{3}}^{2}+\frac{1}{\epsilon^{2}}\|4(\theta^{N})^{3}\theta_{r}-\overline{f_{r}}\|_{L_{t}^{2}H_{\vec x}^{3}}^{2}
\\\leq& C(t)\Big(\frac{1}{\epsilon^{16}}\|(R_{1}, \vec{R}_{2}, R_{3}, R_{4}, R)\|_{L_{t}^{2}L_{\vec w}^{2}H_{\vec x}^{3}}^{2}\Big),
\end{aligned}
\end{equation}
where $C$ is a constant depending on $t$ but not depending on $\epsilon$.}

\noindent\textbf{Proof.} In view of the bounds $\|\tilde{\rho}_{r}, \tilde{\theta}_{r}\|_{L_{T}^{\infty}H_{\vec x}^{3}}\leq\epsilon^{s}$ and the decompositions $\tilde\rho=\rho^{N}+\tilde{\rho}_{r}$, $\tilde\theta=\theta^{N}+\tilde{\theta}_{r}$, we have $\tilde\rho, \tilde\theta\geq a/4$ for all sufficiently small $\epsilon$. For each fixed $\epsilon$, the local existence of solutions to \eqref{research equations-3-2-jia}, \eqref{reeq4-jiajia}, and \eqref{rein-jiajia-1} on an interval $[0, \tilde T]$ follows from the standard linear hyperbolic theory and Lemma \ref{lianxujieguo}. We may assume without loss of generality that $0<\tilde T<T$. To extend the solution up to $T$, it remains to derive uniform estimates in time.

Applying the operator $E^{0}(\tilde V)D_{x}^{\gamma}(E^{0}(\tilde V))^{-1}$ to the equation \eqref{research equations-3-2-jia} where $0\leq\gamma\leq 3$, we obtain
\begin{equation}\label{ene-0}
\begin{split}
E^{0}(\tilde{V})\frac{\partial D_{x}^{\gamma}V_{r}}{\partial t}+\sum_{j=1}^{3}E_{j}^{1}(\tilde{V})\frac{\partial D_{x}^{\gamma}V_{r}}{\partial x_{j}}+D_{x}^{\gamma}(E^{2}(\tilde V)V_{r})=&E^{0}(\tilde V)D_{x}^{\gamma}(E^{0}(\tilde V))^{-1}\vec H+F_{\gamma},
\end{split}
\end{equation}
where 
\begin{equation}\nonumber
F_{\gamma}=-\sum_{j=1}^{3}\sum_{0\leq\beta\leq \gamma-1}E^{0}(\tilde V)D_{x}^{\gamma-\beta}((E^{0}(\tilde V))^{-1}E_{j}^{1}(\tilde V))\frac{\partial D_{x}^{\beta}V_{r}}{\partial x_{j}}.
\end{equation}
Define
\begin{equation}\label{def Et}
E_{r}(t):=\frac{1}{2}\int_{\mathbb{T}^{3}}(E^{0}(\tilde V)D_{x}^{\gamma}V_{r})\cdot D_{x}^{\gamma}V_{r}\mathrm{d}\vec x.
\end{equation}
Then,
\begin{equation}\label{energy-remainder}
\begin{aligned}
E_{r}(t)=&\int_{0}^{t}\frac{\mathrm{d}E_{r}(s)}{\mathrm{d}s}\mathrm{d}s\\=&\int_{0}^{t}\int_{\mathbb{T}^{3}}\Big(\frac{1}{2}\partial_{s}E^{0}(\tilde V)
D_{x}^{\gamma}V_{r}\cdot D_{x}^{\gamma}V_{r}+\frac{1}{2}\sum_{j=1}^{3}\frac{\partial E_{j}^{1}(\tilde V)}{\partial x_{j}}D_{x}^{\gamma}V_{r}\cdot D_{x}^{\gamma}V_{r}\\&+E^{0}(\tilde V)D_{x}^{\gamma}(E^{0}(\tilde V))^{-1}\vec H\cdot D_{x}^{\gamma}V_{r}-D_{x}^{\gamma}(E^{2}(\tilde V)V_{r})\cdot D_{x}^{\gamma}V_{r}+F_{\gamma}\cdot D_{x}^{\gamma}V_{r}\Big)\mathrm{d}\vec x\mathrm{d}s,
\end{aligned}
\end{equation}
where $t\in [0, \tilde T]$.

By \eqref{def Et}, we have 
\begin{equation}\nonumber
E_{r}(t)\geq\lambda\int_{\mathbb{T}^{3}}|D_{x}^{\gamma}V_{r}|^{2}\mathrm{d}\vec x,
\end{equation}
for some positive constant $\lambda=\lambda(a)$.

After a straightforward computation, we obtain
\begin{equation}\label{energy-remainder-1}
\begin{aligned}
\lambda\int_{\mathbb{T}^{3}}|D_{x}^{\gamma}V_{r}|^{2}\mathrm{d}\vec x\leq& C(\epsilon)(\|V_{r}\|_{L_{t}^{2}H_{\vec x}^{3}}^{2}+\|f_{r}\|_{L_{t}^{2}L_{\vec w}^{2}H_{\vec x}^{3}}^{2}),
\end{aligned}
\end{equation}
where $C(\epsilon)$ is a positive constant depending on $\epsilon$.

Summing up from $\gamma=0$ to $\gamma=3$, we have
\begin{equation}\label{energy-remainder-1}
\begin{aligned}
\|V_{r}\|_{H_{\vec x}^{3}}^{2}(t)\leq& C(\epsilon)\lambda^{-1}(\|V_{r}\|_{L_{t}^{2}H_{\vec x}^{3}}^{2}+\|f_{r}\|_{L_{t}^{2}L_{\vec w}^{2}H_{\vec x}^{3}}^{2}+\|R_{1}, \vec{R}_{2}, R_{3}, R\|_{L_{t}^{2}L_{\vec w}^{2}H_{\vec x}^{3}}^{2}).
\end{aligned}
\end{equation}
For \eqref{reeq2-jia-jia}, we can derive
\begin{equation}\label{energy-remainder-2}
\begin{aligned}
\|f_{r}\|_{L_{\vec w}^{2}H_{\vec x}^{3}}^{2}(t)\leq& C(\epsilon)(\|V_{r}\|_{L_{t}^{2}H_{\vec x}^{3}}^{2}+\|f_{r}\|_{L_{t}^{2}L_{\vec w}^{2}H_{\vec x}^{3}}^{2}+\|R_{4}, R\|_{L_{t}^{2}L_{\vec w}^{2}H_{\vec x}^{3}}^{2}).
\end{aligned}
\end{equation}
Collecting \eqref{energy-remainder-1} and \eqref{energy-remainder-2}, we have
\begin{equation}\label{energy-remainder-3}
\begin{aligned}
\|f_{r}\|_{L_{\vec w}^{2}H_{\vec x}^{3}}^{2}(t)+\|V_{r}\|_{H_{\vec x}^{3}}^{2}(t)\leq& C(\epsilon)(\|V_{r}\|_{L_{t}^{2}H_{\vec x}^{3}}^{2}+\|f_{r}\|_{L_{t}^{2}L_{\vec w}^{2}H_{\vec x}^{3}}^{2}+\|R_{1}, \vec{R}_{2}, R_{3}, R_{4}, R\|_{L_{t}^{2}L_{\vec w}^{2}H_{\vec x}^{3}}^{2}).
\end{aligned}
\end{equation}
By the Gronwall's inequality, we have
\begin{equation}\label{energy-remainder-4}
\begin{aligned}
\|V_{r}\|_{L_{t}^{2}H_{\vec x}^{3}}^{2}+\|f_{r}\|_{L_{t}^{2}L_{\vec w}^{2}H_{\vec x}^{3}}^{2}\leq& e^{C(\epsilon)t}t(\|R_{1}, \vec{R}_{2}, R_{3}, R_{4}, R\|_{L_{t}^{2}L_{\vec w}^{2}H_{\vec x}^{3}}^{2})\\\leq& e^{C(\epsilon)T}T(\|R_{1}, \vec{R}_{2}, R_{3}, R_{4}, R\|_{L_{T}^{2}L_{\vec w}^{2}H_{\vec x}^{3}}^{2})
\end{aligned}
\end{equation}
and
\begin{equation}\label{energy-remainder-5}
\begin{aligned}
\|V_{r}\|_{L_{t}^{\infty}H_{\vec x}^{3}}^{2}+\|f_{r}\|_{L_{t}^{\infty}L_{\vec w}^{2}H_{\vec x}^{3}}^{2}\leq& e^{C(\epsilon)T}T(\|R_{1}, \vec{R}_{2}, R_{3}, R_{4}, R\|_{L_{T}^{2}L_{\vec w}^{2}H_{\vec x}^{3}}^{2}).
\end{aligned}
\end{equation}
By the fact that $t\leq\tilde T\leq T$, we have
\begin{equation}\label{energy-remainder-6}
\begin{aligned}
\|V_{r}\|_{L_{\tilde T}^{\infty}H_{\vec x}^{3}}^{2}+\|f_{r}\|_{L_{\tilde T}^{\infty}L_{\vec w}^{2}H_{\vec x}^{3}}^{2}\leq& e^{C(\epsilon)T}T(\|R_{1}, \vec{R}_{2}, R_{3}, R_{4}, R\|_{L_{T}^{2}L_{\vec w}^{2}H_{\vec x}^{3}}^{2}).
\end{aligned}
\end{equation}
Then, we can extend the existence time to $[0, T]$ by the estimate \eqref{energy-remainder-6} and the continuation method.

In the following, we give the uniform energy estimate \eqref{bd00} about $\epsilon$. 

According to \eqref{expansion of rho u epsilon} and \eqref{biaoda6}, for suitably small $\eta$ and $\epsilon$, we have 
\begin{equation}\nonumber
\begin{aligned}
(\theta^{N})^{3}=&\theta_{0}^{3}+\theta_{I, 0}^{3}+O(\epsilon)=\theta_{0}^{3}+O(\eta)+O(\epsilon)\geq C>0
\end{aligned}
\end{equation}
where $C=C(a)$ is a given constant.

Applying the operator $E^{0}(\tilde V)D_{x}^{\gamma}(E^{0}(\tilde V))^{-1}$ to the equation \eqref{research equations-3-2-jia} where $0\leq\gamma\leq 3$, and multiplying \eqref{research equations-3-2-jia} by $D_{x}^{\gamma}V_{r}$, we have
\begin{equation}\label{guji1}
\begin{aligned}
&\frac{1}{2}\int_{\mathbb{T}^{3}}(E^{0}(\tilde V)D_{x}^{\gamma}V_{r})\cdot D_{x}^{\gamma}V_{r}\mathrm{d}\vec x\\=&\int_{0}^{t}\int_{\mathbb{T}^{3}}\Big(\frac{1}{2}\partial_{s}E^{0}(\tilde V)
D_{x}^{\gamma}V_{r}\cdot D_{x}^{\gamma}V_{r}+\frac{1}{2}\sum_{j=1}^{3}\frac{\partial E_{j}^{1}(\tilde V)}{\partial x_{j}}D_{x}^{\gamma}V_{r}\cdot D_{x}^{\gamma}V_{r}\\&+E^{0}(\tilde V)D_{x}^{\gamma}(E^{0}(\tilde V))^{-1}\vec H\cdot D_{x}^{\gamma}V_{r}-D_{x}^{\gamma}(E^{2}(\tilde V)V_{r})\cdot D_{x}^{\gamma}V_{r}+F_{\gamma}\cdot D_{x}^{\gamma}V_{r}\Big)\mathrm{d}\vec x\mathrm{d}s
\\:=&J_{1}+J_{2}+J_{3}+J_{4}+J_{5},
\end{aligned}
\end{equation}
where
\begin{equation}\label{guji1-1}
J_{1}=\int_{0}^{t}\int_{\mathbb{T}^{3}}\frac{1}{2}\partial_{s}E^{0}(\tilde V)
D_{x}^{\gamma}V_{r}\cdot D_{x}^{\gamma}V_{r}\mathrm{d}\vec x\mathrm{d}s,\ \
J_{2}=\int_{0}^{t}\int_{\mathbb{T}^{3}}\frac{1}{2}\sum_{j=1}^{3}\frac{\partial E_{j}^{1}(\tilde V)}{\partial x_{j}}D_{x}^{\gamma}V_{r}\cdot D_{x}^{\gamma}V_{r}\mathrm{d}\vec x\mathrm{d}s,
\end{equation}

\begin{equation}\label{guji3}
J_{3}=\int_{0}^{t}\int_{\mathbb{T}^{3}}E^{0}(\tilde V)D_{x}^{\gamma}(E^{0}(\tilde V))^{-1}\vec H\cdot D_{x}^{\gamma}V_{r}\mathrm{d}\vec x\mathrm{d}s,\ \
J_{5}=-\int_{0}^{t}\int_{\mathbb{T}^{3}}F_{\gamma}\cdot D_{x}^{\gamma}V_{r}\mathrm{d}\vec x\mathrm{d}s.
\end{equation}
For $J_{1}, J_{2}$ and $J_{5}$, we have
\begin{equation}\label{guji6}
J_{1}+J_{2}+J_{5}\leq C(\|V_{r}\|_{L_{t}^{2}L_{\vec x}^{2}}^{2}+\|D_{x}^{\gamma}V_{r}\|_{L_{t}^{2}L_{\vec x}^{2}}^{2}),
\end{equation}
where we have used the fact that $\|\tilde{\rho}_{r}, \tilde{\vec{u}}_{r}, \tilde{\theta}_{r}\|_{L_{T}^{\infty}H_{\vec x}^{3}}\leq\epsilon^{s}$, $\|\partial_{t}\tilde{\rho}_{r}, \partial_{t}\tilde{\theta}_{r}\|_{L_{T}^{\infty}H_{\vec x}^{2}}\leq 1$.

We decompose $\tilde{E^{2}}$ into the singular part $\tilde{E^{2}_{1}}$ and nonsingular part $\tilde{E^{2}_{2}}$ as follows
\begin{equation}\label{erjiejuzhen-1}
\tilde{E^{2}_{1}}
=
\left[
\begin{array}{ccccc}
0&0&0&0&0\\
\frac{\partial_{t}u_{1}^{N}}{\tilde\rho}&0&0&0&0\\
\frac{\partial_{t}u_{2}^{N}}{\tilde\rho}&0&0&0&0\\
\frac{\partial_{t}u_{3}^{N}}{\tilde\rho}&0&0&0&0\\
\frac{\partial_{t}\theta^{N}}{\tilde\rho}&0&0&0&0\\
\end{array}
\right],
\end{equation}
and 
\begin{equation}\label{erjiejuzhen-2}
\tilde{E^{2}_{2}}
=
\left[
\begin{array}{ccccc}
\mathrm{div}_{x}\vec{u}^{N}&\partial_{x_{1}}\rho^{N}&\partial_{x_{2}}\rho^{N}&\partial_{x_{3}}\rho^{N}&0\\
\frac{\partial_{t}u_{1}^{N}}{\tilde\rho}&\rho^{N}\partial_{x_{1}}u_{1}^{N}&\rho^{N}\partial_{x_{2}}u_{1}^{N}&\rho^{N}\partial_{x_{3}}u_{1}^{N}&1\\
\frac{\tilde{\vec u}\cdot\nabla_{x}u_{2}^{N}}{\tilde\rho}+\frac{\partial_{x_{2}}\theta^{N}}{\tilde\rho}&\rho^{N}\partial_{x_{1}}u_{2}^{N}&\rho^{N}\partial_{x_{2}}u_{2}^{N}&\rho^{N}\partial_{x_{3}}u_{2}^{N}&1\\
\frac{\tilde{\vec u}\cdot\nabla_{x}u_{3}^{N}}{\tilde\rho}+\frac{\partial_{x_{3}}\theta^{N}}{\tilde\rho}&\rho^{N}\partial_{x_{1}}u_{3}^{N}&\rho^{N}\partial_{x_{2}}u_{3}^{N}&\rho^{N}\partial_{x_{3}}u_{3}^{N}&1\\
\frac{\theta^{N}\mathrm{div}\vec{u}^{N}}{\tilde\rho}
&\rho^{N}\partial_{x_{1}}\theta^{N}&\rho^{N}\partial_{x_{2}}\theta^{N}&\rho^{N}\partial_{x_{3}}\theta^{N}&\rho^{N}\mathrm{div}_{x}\vec{u}^{N}\\
\end{array}
\right].
\end{equation}
So, $E^{2}(\tilde V)=E^{2}_{1}(\tilde V)+E^{2}_{2}(\tilde V)=E^{0}(\tilde V)\tilde{E^{2}_{1}}(\tilde V)+E^{0}(\tilde V)\tilde{E^{2}_{2}}(\tilde V)$.

We write
\begin{equation}\label{guji4}
\begin{aligned}
J_{4}=&-\int_{0}^{t}\int_{\mathbb{T}^{3}}D_{x}^{\gamma}(E^{2}(\tilde V)V_{r})\cdot D_{x}^{\gamma}V_{r}\mathrm{d}\vec x\mathrm{d}s
\\=&-\int_{0}^{t}\int_{\mathbb{T}^{3}}D_{x}^{\gamma}(E^{2}_{1}(\tilde V)V_{r})\cdot D_{x}^{\gamma}V_{r}\mathrm{d}\vec x\mathrm{d}s-\int_{0}^{t}\int_{\mathbb{T}^{3}}D_{x}^{\gamma}(E^{2}_{2}(\tilde V)V_{r})\cdot D_{x}^{\gamma}V_{r}\mathrm{d}\vec x\mathrm{d}s:=J_{4}^{1}+J_{4}^{2}.
\end{aligned}
\end{equation}
In view of the expansion $E^{2}_{1}(\tilde V)\sim \epsilon\partial_{t}\vec{u}_{I, 1}+\partial_{t}\theta_{I, 0}+\epsilon\theta_{I, 1}+O(\epsilon^{2})$, combined with \eqref{ineq-theta00-jia1}, \eqref{uI1-jia}, and \eqref{ineq-theta11}, we obtain
\begin{equation}\label{guji4-1}
\begin{aligned}
J_{4}^{1}=&-\int_{0}^{t}\int_{\mathbb{T}^{3}}D_{x}^{\gamma}(E^{2}_{1}(\tilde V)V_{r})\cdot D_{x}^{\gamma}V_{r}\mathrm{d}\vec x\mathrm{d}s
\\\leq&C(\|V_{r}\|_{L_{t}^{\infty}L_{\vec x}^{2}}^{2}+\|D_{x}^{\gamma}V_{r}\|_{L_{t}^{\infty}L_{\vec x}^{2}}^{2})\int_{0}^{\frac{t}{\epsilon^{2}}}(\|(\frac{1}{\epsilon}\partial_{\tau}\vec{u}_{I, 1}, \frac{1}{\epsilon^{2}}\partial_{\tau}\theta_{I, 0}, \frac{1}{\epsilon}\partial_{\tau}\theta_{I, 1})\|_{L_{t}^{\infty}H_{\vec x}^{3}}\epsilon^{2}\mathrm{d}\tau\\&+C(\|V_{r}\|_{L_{t}^{2}L_{\vec x}^{2}}^{2}+\|D_{x}^{\gamma}V_{r}\|_{L_{t}^{2}L_{\vec x}^{2}}^{2})
\\\leq& C(\epsilon+\eta)(\|V_{r}\|_{L_{t}^{\infty}L_{\vec x}^{2}}^{2}+\|D_{x}^{\gamma}V_{r}\|_{L_{t}^{\infty}L_{\vec x}^{2}}^{2})+C(\|V_{r}\|_{L_{t}^{2}L_{\vec x}^{2}}^{2}+\|D_{x}^{\gamma}V_{r}\|_{L_{t}^{2}L_{\vec x}^{2}}^{2}).
\end{aligned}
\end{equation}
For the nonsingular term $J_{4}^{2}$, we have
\begin{equation}\label{guji4-2}
\begin{aligned}
J_{4}^{2}=&-\int_{0}^{t}\int_{\mathbb{T}^{3}}D_{x}^{\gamma}(E^{2}_{2}(\tilde V)V_{r})\cdot D_{x}^{\gamma}V_{r}\mathrm{d}\vec x\mathrm{d}s
\leq C(\|V_{r}\|_{L_{t}^{2}L_{\vec x}^{2}}^{2}+\|D_{x}^{\gamma}V_{r}\|_{L_{t}^{2}L_{\vec x}^{2}}^{2}).
\end{aligned}
\end{equation}

In order to get the uniform estimates about $\epsilon$, we need to give a detailed calculation about $J_{3}$ as follows
\begin{equation}\nonumber
\begin{aligned}
J_{3}=&\int_{0}^{t}\int_{\mathbb{T}^{3}}E^{0}(\tilde V)D_{x}^{\gamma}(E^{0}(\tilde V))^{-1}\vec H\cdot D_{x}^{\gamma}V_{r}\mathrm{d}\vec x\mathrm{d}s
\\=&\int_{0}^{t}\int_{\mathbb{T}^{3}}D_{x}^{\gamma}\vec H\cdot D_{x}^{\gamma}V_{r}\mathrm{d}\vec x\mathrm{d}s+\int_{0}^{t}\int_{\mathbb{T}^{3}}E^{0}(\tilde V)[D_{x}^{\gamma}, (E^{0}(\tilde V))^{-1}]\vec H\cdot D_{x}^{\gamma}V_{r}\mathrm{d}\vec x\mathrm{d}s:=J_{3}^{1}+J_{3}^{2},
\end{aligned}
\end{equation}
where
\begin{equation}\nonumber
\begin{aligned}
J_{3}^{1}=&\int_{0}^{t}\int_{\mathbb{T}^{3}}D_{x}^{\gamma}\vec H\cdot D_{x}^{\gamma}V_{r}\mathrm{d}\vec x\mathrm{d}s
\end{aligned}
\end{equation}
and
\begin{equation}\nonumber
\begin{aligned}
J_{3}^{2}=&\int_{0}^{t}\int_{\mathbb{T}^{3}}E^{0}(\tilde V)[D_{x}^{\gamma}, (E^{0}(\tilde V))^{-1}]\vec H\cdot D_{x}^{\gamma}V_{r}\mathrm{d}\vec x\mathrm{d}s.
\end{aligned}
\end{equation}
By the definition of $\vec H$ and $V_{r}$, we rewrite $J_{31}$ as follows
\begin{equation}\nonumber
\begin{aligned}
J_{3}^{1}=\int_{0}^{t}\int_{\mathbb{T}^{3}}D_{x}^{\gamma}\vec H\cdot D_{x}^{\gamma}V_{r}\mathrm{d}\vec x\mathrm{d}s
=&\int_{0}^{t}\int_{\mathbb{T}^{3}}(D_{x}^{\gamma}H_{1}D_{x}^{\gamma}\rho_{r}+D_{x}^{\gamma}(H_{2}, H_{3}, H_{4})\cdot D_{x}^{\gamma}\vec{u}_{r}+D_{x}^{\gamma}H_{5}D_{x}^{\gamma}\theta_{r})\mathrm{d}\vec x\mathrm{d}s
\\:=&K_{1}+K_{2}+K_{3},
\end{aligned}
\end{equation}
where
\begin{equation}\nonumber
\begin{aligned}
K_{1}=&\int_{0}^{t}\int_{\mathbb{T}^{3}}D_{x}^{\gamma}H_{1}D_{x}^{\gamma}\rho_{r}\mathrm{d}\vec x\mathrm{d}s
=&\int_{0}^{t}\int_{\mathbb{T}^{3}}D_{x}^{\gamma}\Big(\frac{1}{\tilde\rho}R_{1}\Big)D_{x}^{\gamma}\rho_{r}\mathrm{d}\vec x\mathrm{d}s
\leq& C(\|R_{1}\|_{L_{t}^{2}H_{\vec x}^{3}}^{2}+\|D_{x}^{\gamma}\rho_{r}\|_{L_{t}^{2}L_{\vec x}^{2}}^{2})
\end{aligned}
\end{equation}
and
\begin{equation}\nonumber
\begin{aligned}
K_{2}=&\int_{0}^{t}\int_{\mathbb{T}^{3}}D_{x}^{\gamma}(H_{2}, H_{3}, H_{4})\cdot D_{x}^{\gamma}\vec{u}_{r}\mathrm{d}\vec x\mathrm{d}s
\\\leq& C\Big(\frac{1}{\epsilon^{4}}\|\vec{R}_{2}\|_{L_{t}^{2}H_{\vec x}^{3}}^{2}+\frac{\delta}{\epsilon^{2}}(\|D_{x}^{\gamma}f_{r}-D_{x}^{\gamma}\overline{f_{r}}\|_{L_{t}^{2}L_{\vec x}^{2}}^{2}+\|f_{r}-\overline{f_{r}}\|_{L_{t}^{2}L_{\vec x}^{2}}^{2})+\frac{1}{\delta}\|D_{x}^{\gamma}\vec{u}_{r}\|_{L_{t}^{2}L_{\vec x}^{2}}^{2}\Big).
\end{aligned}
\end{equation}
We rewrite $J_{313}$ as follows
\begin{equation}\nonumber
\begin{aligned}
K_{3}=&\int_{0}^{t}\int_{\mathbb{T}^{3}}D_{x}^{\gamma}\theta_{r}D_{x}^{\gamma}H_{5}\mathrm{d}\vec x\mathrm{d}s
=\int_{0}^{t}\int_{\mathbb{T}^{3}}D_{x}^{\gamma}\theta_{r}D_{x}^{\gamma}\Big(\frac{1}{(\tilde{\theta})^{2}\epsilon^{2}}(R_{3}+\langle R\rangle)\Big)\mathrm{d}\vec x\mathrm{d}s\\&+\int_{0}^{t}\int_{\mathbb{T}^{3}}D_{x}^{\gamma}\theta_{r}D_{x}^{\gamma}\Big(\frac{1}{(\tilde{\theta})^{2}\epsilon^{2}}(4(\theta^{N})^{3}\theta_{r}-f_{r})\Big)\mathrm{d}\vec x\mathrm{d}s
:=L_{1}+L_{2},
\end{aligned}
\end{equation}
where
\begin{equation}\nonumber
\begin{aligned}
L_{1}=&\int_{0}^{t}\int_{\mathbb{T}^{3}}D_{x}^{\gamma}\theta_{r}D_{x}^{\gamma}\Big(\frac{1}{(\tilde{\theta})^{2}\epsilon^{2}}(R_{3}+\langle R\rangle)\Big)\mathrm{d}\vec x\mathrm{d}s
\leq& C\Big(\|D_{x}^{\gamma}\theta_{r}\|_{L_{t}^{2}L_{\vec x}^{2}}^{2}+\frac{1}{\epsilon^{4}}(\|R_{3}\|_{L_{t}^{2}H_{\vec x}^{3}}^{2}+\|R\|_{L_{t}^{2}L_{\vec w}^{2}H_{\vec x}^{3}}^{2})\Big)
\end{aligned}
\end{equation}
and
\begin{equation}\label{eqjia0}
\begin{aligned}
L_{2}=&\int_{0}^{t}\int_{\mathbb{T}^{3}}D_{x}^{\gamma}\theta_{r}D_{x}^{\gamma}\Big(\frac{1}{(\tilde{\theta})^{2}\epsilon^{2}}(4(\theta^{N})^{3}\theta_{r}-\overline{f_{r}})\Big)\mathrm{d}\vec x\mathrm{d}s.
\end{aligned}
\end{equation}
We get from \eqref{eqjia0} that
\begin{equation}\label{eqjia1}
\begin{aligned}
L_{2}=&\int_{0}^{t}\int_{\mathbb{T}^{3}}D_{x}^{\gamma}\theta_{r}D_{x}^{\gamma}\Big(\frac{1}{(\tilde{\theta})^{2}\epsilon^{2}}(4(\theta^{N})^{3}\theta_{r}-\overline{f_{r}})\Big)\mathrm{d}\vec x\mathrm{d}s
\\=&\int_{0}^{t}\int_{\mathbb{T}^{3}}\Big(\frac{1}{4(\theta^{N})^{3}}D_{x}^{\gamma}(4L_{0}\theta_{r})+[D_{x}^{\gamma}, \frac{1}{4(\theta^{N})^{3}}]4(\theta^{N})^{3}\theta_{r}\Big)\Big(\frac{1}{(\tilde{\theta})^{2}\epsilon^{2}}D_{x}^{\gamma}(4(\theta^{N})^{3}\theta_{r}-\overline{f_{r}})\\&+[D_{x}^{\gamma}, \frac{1}{(\tilde{\theta})^{2}\epsilon^{2}}](4(\theta^{N})^{3}\theta_{r}-\overline{f_{r}})\Big)\mathrm{d}\vec x\mathrm{d}s
\\=&\int_{0}^{t}\int_{\mathbb{T}^{3}}\frac{1}{4(\theta^{N})^{3}(\tilde{\theta})^{2}\epsilon^{2}}D_{x}^{\gamma}(4(\theta^{N})^{3}\theta_{r})D_{x}^{\gamma}(4(\theta^{N})^{3}\theta_{r}-\overline{f_{r}})\mathrm{d}\vec x\mathrm{d}s\\&+\int_{0}^{t}\int_{\mathbb{T}^{3}}\frac{1}{4(\theta^{N})^{3}}D_{x}^{\gamma}(4(\theta^{N})^{3}\theta_{r})[D_{x}^{\gamma}, \frac{1}{(\tilde{\theta})^{2}\epsilon^{2}}](4(\theta^{N})^{3}\theta_{r}-\overline{f_{r}})\mathrm{d}\vec x\mathrm{d}s\\&+\int_{0}^{t}\int_{\mathbb{T}^{3}}\Big([D_{x}^{\gamma}, \frac{1}{4(\theta^{N})^{3}}]4(\theta^{N})^{3}\theta_{r}\Big) \frac{1}{(\tilde{\theta})^{2}\epsilon^{2}}D_{x}^{\gamma}(4(\theta^{N})^{3}\theta_{r}-\overline{f_{r}})\mathrm{d}\vec x\mathrm{d}s
\\&+\int_{0}^{t}\int_{\mathbb{T}^{3}}\Big([D_{x}^{\gamma}, \frac{1}{4(\theta^{N})^{3}}]4(\theta^{N})^{3}\theta_{r}\Big)[D_{x}^{\gamma}, \frac{1}{(\tilde{\theta})^{2}\epsilon^{2}}](4(\theta^{N})^{3}\theta_{r}-\overline{f_{r}})\mathrm{d}\vec x\mathrm{d}s
\\:=&\int_{0}^{t}\int_{\mathbb{T}^{3}}\frac{1}{4(\theta^{N})^{3}(\tilde{\theta})^{2}\epsilon^{2}}D_{x}^{\gamma}(4(\theta^{N})^{3}\theta_{r})D_{x}^{\gamma}(4(\theta^{N})^{3}\theta_{r}-\overline{f_{r}})\mathrm{d}\vec x\mathrm{d}s+\tilde{L}_{2}
\end{aligned}
\end{equation}
where
\begin{equation}\label{eqjia2}
\begin{aligned}
\tilde{L}_{2}=&\int_{0}^{t}\int_{\mathbb{T}^{3}}\frac{1}{4(\theta^{N})^{3}}D_{x}^{\gamma}(4(\theta^{N})^{3}\theta_{r})[D_{x}^{\gamma}, \frac{1}{(\tilde{\theta})^{2}\epsilon^{2}}](4(\theta^{N})^{3}\theta_{r}-\overline{f_{r}})\mathrm{d}\vec x\mathrm{d}s
\end{aligned}
\end{equation}
\begin{equation}\nonumber
\begin{aligned}
\\&+\int_{0}^{t}\int_{\mathbb{T}^{3}}\Big([D_{x}^{\gamma}, \frac{1}{4(\theta^{N})^{3}}]4L_{0}\theta_{r}\Big) \frac{1}{(\tilde{\theta})^{2}\epsilon^{2}}D_{x}^{\gamma}(4(\theta^{N})^{3}\theta_{r}-\overline{f_{r}})\mathrm{d}\vec x\mathrm{d}s
\\&+\int_{0}^{t}\int_{\mathbb{T}^{3}}\Big([D_{x}^{\gamma}, \frac{1}{4(\theta^{N})^{3}}]\theta_{r}\Big)[D_{x}^{\gamma}, \frac{1}{(\tilde{\theta})^{2}\epsilon^{2}}](4(\theta^{N})^{3}\theta_{r}-\overline{f_{r}})\mathrm{d}\vec x\mathrm{d}s
\\\leq& C(\|\theta_{r}\|_{L_{t}^{2}L_{\vec x}^{2}}^{2}+\|D_{x}^{\gamma}\theta_{r}\|_{L_{t}^{2}L_{\vec x}^{2}}^{2}+\|f_{r}\|_{L_{t}^{2}L_{\vec x}^{2}L_{\vec w}^{2}}^{2}+\|D_{x}^{\gamma}f_{r}\|_{L_{t}^{2}L_{\vec x}^{2}L_{\vec w}^{2}}^{2})\\&+\frac{C}{\epsilon^{4}}(\|\theta_{r}\|_{L_{t}^{2}L_{\vec x}^{2}}^{2}+\|D_{x}^{\gamma-1}\theta_{r}\|_{L_{t}^{2}L_{\vec x}^{2}}^{2}+\|f_{r}\|_{L_{t}^{2}L_{\vec x}^{2}L_{\vec w}^{2}}^{2}+\|D_{x}^{\gamma-1}f_{r}\|_{L_{t}^{2}L_{\vec x}^{2}L_{\vec w}^{2}}^{2}),
\end{aligned}
\end{equation}
and $\tilde{L}_{2}=0$ when $\gamma=0$ by the properties of the communicator operator. Provided $\eta$ and $\epsilon>0$ are sufficiently small, one can choose positive constants $c_{1}$ and $c_{2}$ satisfying $c_{1}<4(\theta^{N})^{3}(\tilde{\theta})^{2}<c_{2}$. So, we can verify that
\begin{equation}\label{qingkuang1}
\begin{aligned}
&\int_{0}^{t}\int_{\mathbb{T}^{3}}\frac{1}{4(\theta^{N})^{3}(\tilde{\theta})^{2}\epsilon^{2}}D_{x}^{\gamma}(4(\theta^{N})^{3}\theta_{r})D_{x}^{\gamma}(4(\theta^{N})^{3}\theta_{r}-\overline{f_{r}})\mathrm{d}\vec x\mathrm{d}s\\\leq&\frac{1}{c_{1}}\int_{0}^{t}\int_{\mathbb{T}^{3}}\frac{1}{\epsilon^{2}}D_{x}^{\gamma}(4(\theta^{N})^{3}\theta_{r})D_{x}^{\gamma}(4(\theta^{N})^{3}\theta_{r}-\overline{f_{r}})\mathrm{d}\vec x\mathrm{d}s
\end{aligned}
\end{equation}
or
\begin{equation}\label{qingkuang2}
\begin{aligned}
&\int_{0}^{t}\int_{\mathbb{T}^{3}}\frac{1}{4(\theta^{N})^{3}(\tilde{\theta})^{2}\epsilon^{2}}D_{x}^{\gamma}(4L_{0}\theta_{r})D_{x}^{\gamma}(4(\theta^{N})^{3}\theta_{r}-\overline{f_{r}})\mathrm{d}\vec x\mathrm{d}s\\\leq&\frac{1}{c_{2}}\int_{0}^{t}\int_{\mathbb{T}^{3}}\frac{1}{\epsilon^{2}}D_{x}^{\gamma}(4(\theta^{N})^{3}\theta_{r})D_{x}^{\gamma}(4(\theta^{N})^{3}\theta_{r}-\overline{f_{r}})\mathrm{d}\vec x\mathrm{d}s.
\end{aligned}
\end{equation}
In view of \eqref{qingkuang1} and \eqref{qingkuang2}, it follows that
\begin{equation}\label{qingkuang3}
\begin{aligned}
&\int_{0}^{t}\int_{\mathbb{T}^{3}}\frac{1}{4(\theta^{N})^{3}(\tilde{\theta})^{2}\epsilon^{2}}D_{x}^{\gamma}(4(\theta^{N})^{3}\theta_{r})D_{x}^{\gamma}(4(\theta^{N})^{3}\theta_{r}-\overline{f_{r}})\mathrm{d}\vec x\mathrm{d}s\\\leq&\frac{1}{c}\int_{0}^{t}\int_{\mathbb{T}^{3}}\frac{1}{\epsilon^{2}}D_{x}^{\gamma}(4(\theta^{N})^{3}\theta_{r})D_{x}^{\gamma}(4(\theta^{N})^{3}\theta_{r}-\overline{f_{r}})\mathrm{d}\vec x\mathrm{d}s,
\end{aligned}
\end{equation}
where $c=c_{1}$ or $c=c_{2}$ and $c>0$.

For $J_{3}^{2}$, we have
\begin{equation}\nonumber
\begin{aligned}
J_{3}^{2}=&\int_{0}^{t}\int_{\mathbb{T}^{3}}E^{0}(\tilde V)[D_{x}^{\gamma}, (E^{0}(\tilde V))^{-1}]\vec H\cdot D_{x}^{\gamma}V_{r}\mathrm{d}\vec x\mathrm{d}s
\\\leq& \int_{0}^{t}\int_{\mathbb{T}^{3}}|D_{x}^{\gamma}V_{r}|^{2}\mathrm{d}\vec x\mathrm{d}s+C\int_{0}^{t}\int_{\mathbb{T}^{3}}(|\vec H|^{2}+|D_{x}^{\gamma-1}\vec H|^{2})\mathrm{d}\vec x\mathrm{d}s
\\\leq& C(\|D_{x}^{\gamma}V_{r}\|_{L_{t}^{2}L_{\vec x}^{2}}+\frac{C}{\epsilon^{4}}\|(R_{1}, R_{2}, R_{3}, R_{4}, R)\|_{L_{t}^{2}L_{\vec w}^{2}H_{\vec x}^{3}}^{2}+\frac{C}{\epsilon^{2}}(\|f_{r}-\overline{f_{r}}\|_{L_{t}^{2}L_{\vec w}^{2}L_{\vec x}^{2}}^{2}+\|D_{x}^{\gamma-1}f_{r}\\&-D_{x}^{\gamma-1}\overline{f_{r}}\|_{L_{t}^{2}L_{\vec w}^{2}L_{\vec x}^{2}}^{2})+\frac{C}{\epsilon^{4}}(\|f_{r}\|_{L_{t}^{2}L_{\vec w}^{2}L_{\vec x}^{2}}^{2}+\|D_{x}^{\gamma-1}f_{r}\|_{L_{t}^{2}L_{\vec w}^{2}L_{\vec x}^{2}}^{2}+\|\theta_{r}\|_{L_{t}^{2}L_{\vec w}^{2}}^{2}+\|D_{x}^{\gamma-1}\theta_{r}\|_{L_{t}^{2}L_{\vec x}^{2}}^{2}),
\end{aligned}
\end{equation}
where we note that $J_{32}=0$ when $\gamma=0$ by the commutator estimate.

Collecting \eqref{guji1}-\eqref{guji6} and respective estimates, we have
\begin{equation}\label{guji7}
\begin{aligned}
&\frac{\lambda}{2}\|D_{x}^{\gamma}V_{r}\|_{L_{\vec x}^{2}}^{2}\mathrm{d}\vec x\\\leq&\frac{1}{2}\int_{\mathbb{T}^{3}}(E^{0}(\tilde V)D_{x}^{\gamma}V_{r})\cdot D_{x}^{\gamma}V_{r}\mathrm{d}\vec x\\\leq&C\Big(\frac{1}{\epsilon^{4}}\|(R_{1}, R_{2}, R_{3}, R_{4}, R)\|_{L_{t}^{2}L_{\vec w}^{2}H_{\vec x}^{3}}^{2}+\frac{\delta}{\epsilon^{2}}(\|D_{x}^{\gamma}f_{r}-D_{x}^{\gamma}\overline{f_{r}}\|_{L_{t}^{2}L_{\vec x}^{2}}^{2}+\|f_{r}-\overline{f_{r}}\|_{L_{t}^{2}L_{\vec x}^{2}}^{2})\Big)
\\&+C(\|V_{r}\|_{L_{t}^{2}L_{\vec x}^{2}}+\|D_{x}^{\gamma}V_{r}\|_{L_{t}^{2}L_{\vec x}^{2}}+\|f_{r}\|_{L_{t}^{2}L_{\vec x}^{2}L_{\vec w}^{2}}^{2}+\|D_{x}^{\gamma}f_{r}\|_{L_{t}^{2}L_{\vec x}^{2}L_{\vec w}^{2}}^{2})+\frac{C}{\epsilon^{2}}(\|f_{r}-\overline{f_{r}}\|_{L_{t}^{2}L_{\vec w}^{2}L_{\vec x}^{2}}^{2}\\&+\|D_{x}^{\gamma-1}f_{r}-D_{x}^{\gamma-1}\overline{f_{r}}\|_{L_{t}^{2}L_{\vec w}^{2}L_{\vec x}^{2}}^{2})+\frac{C}{\epsilon^{4}}(\|f_{r}\|_{L_{t}^{2}L_{\vec w}^{2}L_{\vec x}^{2}}^{2}+\|D_{x}^{\gamma-1}f_{r}\|_{L_{t}^{2}L_{\vec w}^{2}L_{\vec x}^{2}}^{2}+\|\theta_{r}\|_{L_{t}^{2}L_{\vec w}^{2}}^{2}\\&+\|D_{x}^{\gamma-1}\theta_{r}\|_{L_{t}^{2}L_{\vec x}^{2}}^{2})
+C(\epsilon+\eta)(\|V_{r}\|_{L_{t}^{\infty}L_{\vec x}^{2}}^{2}+\|D_{x}^{\gamma}V_{r}\|_{L_{t}^{\infty}L_{\vec x}^{2}}^{2})\\&+\frac{1}{c\epsilon^{2}}\int_{0}^{t}\int_{\mathbb{T}^{3}}D_{x}^{\gamma}(4(\theta^{N})^{3}\theta_{r})D_{x}^{\gamma}(4(\theta^{N})^{3}\theta_{r}-\overline{f_{r}})\mathrm{d}\vec x\mathrm{d}s,
\end{aligned}
\end{equation}
specially, for $\gamma=0$,
\begin{equation}\label{guji8}
\begin{aligned}
\frac{\lambda}{2}\|V_{r}\|_{L_{\vec x}^{2}}^{2}\mathrm{d}\vec x\leq& C\Big(\frac{1}{\epsilon^{4}}\|(R_{1}, \vec{R}_{2}, R_{3}, R_{4}, R)\|_{L_{t}^{2}L_{\vec w}^{2}H_{\vec x}^{3}}^{2}+\frac{\delta}{\epsilon^{2}}\|f_{r}-\overline{f_{r}}\|_{L_{t}^{2}L_{\vec x}^{2}}^{2})\\&+C(\|V_{r}\|_{L_{t}^{2}L_{\vec x}^{2}}+\|f_{r}\|_{L_{t}^{2}L_{\vec x}^{2}L_{\vec w}^{2}}^{2})+\frac{1}{c\epsilon^{2}}\int_{0}^{t}\int_{\mathbb{T}^{3}}4(\theta^{N})^{3}\theta_{r}(4(\theta^{N})^{3}\theta_{r}-\overline{f_{r}})\mathrm{d}\vec x\mathrm{d}s
\\&+C(\epsilon+\eta)\|V_{r}\|_{L_{t}^{\infty}L_{\vec x}^{2}}^{2}.
\end{aligned}
\end{equation}

Upon applying $D_{x}^{\gamma}$, $0\leq\gamma\leq 3$ to both sides of \eqref{reeq2-jia-jia}, multiplying the resulting equation by $D_{x}^{\gamma}f_{r}$, and then integrating with respect to $s\in(0, t)$, $\vec x\in\mathbb{T}^{3}$ and $\vec w\in\mathbb{S}^{2}$, we arrive at
\begin{equation}\label{guji9}
\begin{aligned}
&\int_{\mathbb{T}^{3}}\int_{\mathbb{S}^{2}}|D_{x}^{\gamma}f_{r}|^{2}\mathrm{d}\vec w\mathrm{d}\vec x
+\Big(1+\frac{1}{\epsilon^{2}}\Big)\int_{0}^{t}\int_{\mathbb{T}^{3}}\int_{\mathbb{S}^{2}}|D_{x}^{\gamma}f_{r}-D_{x}^{\gamma}\overline{f_{r}}|^{2}
\mathrm{d}\vec w\mathrm{d}\vec x\mathrm{d}s\\=&\frac{1}{\epsilon^{2}}\int_{0}^{t}\int_{\mathbb{T}^{3}}\int_{\mathbb{S}^{2}}D_{x}^{\gamma}(R_{4}-R)D_{x}^{\gamma}f_{r}\mathrm{d}\vec w\mathrm{d}\vec x\mathrm{d}s+\frac{1}{\epsilon^{2}}\int_{0}^{t}\int_{\mathbb{T}^{3}}\int_{\mathbb{S}^{2}}D_{x}^{\gamma}(4(\theta^{N})^{3}\theta_{r}-\overline{f_{r}})D_{x}^{\gamma}f_{r}\mathrm{d}\vec w\mathrm{d}\vec x\mathrm{d}s
\\:=&G_{1}+G_{2},
\end{aligned}
\end{equation}
where
\begin{equation}\nonumber
\begin{aligned}
G_{1}=&\frac{1}{\epsilon^{2}}\int_{0}^{t}\int_{\mathbb{T}^{3}}\int_{\mathbb{S}^{2}}D_{x}^{\gamma}(R_{4}-R)D_{x}^{\gamma}f_{r}\mathrm{d}\vec w\mathrm{d}\vec x\mathrm{d}s
\\\leq& \int_{0}^{t}\int_{\mathbb{T}^{3}}\int_{\mathbb{S}^{2}}|D_{x}^{\gamma}f_{r}|^{2}\mathrm{d}\vec w\mathrm{d}\vec x\mathrm{d}s+\frac{C}{\epsilon^{4}}\int_{0}^{t}\int_{\mathbb{T}^{3}}\int_{\mathbb{S}^{2}}|D_{x}^{\gamma}(R_{4}-R)|^{2}\mathrm{d}\vec w\mathrm{d}\vec x\mathrm{d}s
\end{aligned}
\end{equation}
and
\begin{equation}\nonumber
\begin{aligned}
G_{2}=&\frac{1}{\epsilon^{2}}\int_{0}^{t}\int_{\mathbb{T}^{3}}\int_{\mathbb{S}^{2}}D_{x}^{\gamma}(4(\theta^{N})^{3}\theta_{r}-\overline{f_{r}})D_{x}^{\gamma}f_{r}\mathrm{d}\vec w\mathrm{d}\vec x\mathrm{d}s=\frac{4\pi}{\epsilon^{2}}\int_{0}^{t}\int_{\mathbb{T}^{3}}(D_{x}^{\gamma}(4(\theta^{N})^{3}\theta_{r}-\overline{f_{r}})D_{x}^{\gamma}\overline{f_{r}}\mathrm{d}\vec x\mathrm{d}s.
\end{aligned}
\end{equation}

Upon multiplying \eqref{guji7} by $c$ and \eqref{guji9} by $\frac{1}{4\pi}$, respectively, and summing the resulting estimates, we obtain
\begin{equation}\label{guji10}
\begin{aligned}
&\frac{\lambda c}{2}\|D_{x}^{\gamma}V_{r}\|_{L_{\vec x}^{2}}^{2}+\frac{1}{4\pi}\int_{\mathbb{T}^{3}}\int_{\mathbb{S}^{2}}|D_{x}^{\gamma}f_{r}|^{2}\mathrm{d}\vec w\mathrm{d}\vec x
+\frac{1}{4\pi}\Big(1+\frac{1}{\epsilon^{2}}\Big)\int_{0}^{t}\int_{\mathbb{T}^{3}}\int_{\mathbb{S}^{2}}|D_{x}^{\gamma}f_{r}-D_{x}^{\gamma}\overline{f_{r}}|^{2}
\mathrm{d}\vec w\mathrm{d}\vec x\mathrm{d}s\\&+\frac{1}{4\pi\epsilon^{2}}\int_{0}^{t}\int_{\mathbb{T}^{3}}|D_{x}^{\gamma}(4(\theta^{N})^{3}\theta_{r}-\overline{f_{r}})|^{2}\mathrm{d}\vec x\mathrm{d}s
\\\leq&Cc\Big(\frac{1}{\epsilon^{4}}\|(R_{1}, R_{2}, R_{3}, R_{4}, R)\|_{L_{t}^{2}L_{\vec w}^{2}H_{\vec x}^{3}}^{2}+\frac{\delta}{\epsilon^{2}}(\|D_{x}^{\gamma}f_{r}-D_{x}^{\gamma}\overline{f_{r}}\|_{L_{t}^{2}L_{\vec x}^{2}}^{2}+\|f_{r}-\overline{f_{r}}\|_{L_{t}^{2}L_{\vec x}^{2}}^{2})\Big)\\&+Cc(\|V_{r}\|_{L_{t}^{2}L_{\vec x}^{2}}+\|D_{x}^{\gamma}V_{r}\|_{L_{t}^{2}L_{\vec x}^{2}}+\|f_{r}\|_{L_{t}^{2}L_{\vec x}^{2}L_{\vec w}^{2}}^{2}+\|D_{x}^{\gamma}f_{r}\|_{L_{t}^{2}L_{\vec x}^{2}L_{\vec w}^{2}}^{2})+\frac{Cc}{\epsilon^{2}}(\|f_{r}-\overline{f_{r}}\|_{L_{t}^{2}L_{\vec w}^{2}L_{\vec x}^{2}}^{2}\\&+\|D_{x}^{\gamma-1}f_{r}-D_{x}^{\gamma-1}\overline{f_{r}}\|_{L_{t}^{2}L_{\vec w}^{2}L_{\vec x}^{2}}^{2})+\frac{Cc}{\epsilon^{4}}(\|f_{r}\|_{L_{t}^{2}L_{\vec w}^{2}L_{\vec x}^{2}}^{2}+\|D_{x}^{\gamma-1}f_{r}\|_{L_{t}^{2}L_{\vec w}^{2}L_{\vec x}^{2}}^{2}+\|\theta_{r}\|_{L_{t}^{2}L_{\vec w}^{2}}^{2}\\&+\|D_{x}^{\gamma-1}\theta_{r}\|_{L_{t}^{2}L_{\vec x}^{2}}^{2})
+Cc\eta(\|\theta_{r}\|_{L_{t}^{\infty}L_{\vec x}^{2}}^{2}+\|D_{x}^{\gamma}\theta_{r}\|_{L_{t}^{\infty}L_{\vec x}^{2}}^{2})+\frac{C}{\epsilon^{4}}\|(R_{2}, R)\|_{L_{t}^{2}L_{\vec w}^{2}H_{\vec x}^{3}}^{2}\\&+C\eta\|D_{x}^{\gamma}f_{r}\|_{L_{t}^{\infty}L_{\vec w}^{2}L_{\vec x}^{2}}^{2}+Cc(\epsilon+\eta)(\|V_{r}\|_{L_{t}^{\infty}L_{\vec x}^{2}}^{2}+\|D_{x}^{\gamma}V_{r}\|_{L_{t}^{\infty}L_{\vec x}^{2}}^{2}).
\end{aligned}
\end{equation}
Choose $\delta$ such that $Cc\delta=\frac{1}{8\pi}$. Then we have
\begin{equation}\label{guji11}
\begin{aligned}
&\frac{\lambda c}{2}\|D_{x}^{\gamma}V_{r}\|_{L_{\vec x}^{2}}^{2}+\frac{1}{8\pi}\int_{\mathbb{T}^{3}}\int_{\mathbb{S}^{2}}|D_{x}^{\gamma}f_{r}|^{2}\mathrm{d}\vec w\mathrm{d}\vec x
+\frac{1}{4\pi}\Big(1+\frac{1}{\epsilon^{2}}\Big)\int_{0}^{t}\int_{\mathbb{T}^{3}}\int_{\mathbb{S}^{2}}|D_{x}^{\gamma}f_{r}-D_{x}^{\gamma}\overline{f_{r}}|^{2}
\mathrm{d}\vec w\mathrm{d}\vec x\mathrm{d}s
\\&+\frac{1}{4\pi\epsilon^{2}}\int_{0}^{t}\int_{\mathbb{T}^{3}}|D_{x}^{\gamma}(4(\theta^{N})^{3}\theta_{r}-\overline{f_{r}})|^{2}\mathrm{d}\vec x\mathrm{d}s\\\leq&Cc\Big(\frac{1}{\epsilon^{4}}\|(R_{1}, R_{2}, R_{3}, R_{4}, R)\|_{L_{t}^{2}L_{\vec w}^{2}H_{\vec x}^{3}}^{2}+\frac{1}{\epsilon^{2}}\|f_{r}-\overline{f_{r}}\|_{L_{t}^{2}L_{\vec x}^{2}}^{2}\Big)\\&+Cc(\|V_{r}\|_{L_{t}^{2}L_{\vec x}^{2}}+\|D_{x}^{\gamma}V_{r}\|_{L_{t}^{2}L_{\vec x}^{2}}+\|f_{r}\|_{L_{t}^{2}L_{\vec x}^{2}L_{\vec w}^{2}}^{2}+\|D_{x}^{\gamma}f_{r}\|_{L_{t}^{2}L_{\vec x}^{2}L_{\vec w}^{2}}^{2})+\frac{Cc}{\epsilon^{2}}(\|f_{r}-\overline{f_{r}}\|_{L_{t}^{2}L_{\vec w}^{2}L_{\vec x}^{2}}^{2}\\&+\|D_{x}^{\gamma-1}f_{r}-D_{x}^{\gamma-1}\overline{f_{r}}\|_{L_{t}^{2}L_{\vec w}^{2}L_{\vec x}^{2}}^{2})+\frac{Cc}{\epsilon^{4}}(\|f_{r}\|_{L_{t}^{2}L_{\vec w}^{2}L_{\vec x}^{2}}^{2}+\|D_{x}^{\gamma-1}f_{r}\|_{L_{t}^{2}L_{\vec w}^{2}L_{\vec x}^{2}}^{2}+\|\theta_{r}\|_{L_{t}^{2}L_{\vec w}^{2}}^{2}\\&+\|D_{x}^{\gamma-1}\theta_{r}\|_{L_{t}^{2}L_{\vec x}^{2}}^{2})
+Cc\eta(\|\theta_{r}\|_{L_{t}^{\infty}L_{\vec x}^{2}}^{2}+\|D_{x}^{\gamma}\theta_{r}\|_{L_{t}^{\infty}L_{\vec x}^{2}}^{2})+\frac{C}{\epsilon^{4}}\|(R_{2}, R)\|_{L_{t}^{2}L_{\vec w}^{2}H_{\vec x}^{3}}^{2}\\&+C\eta\|D_{x}^{\gamma}f_{r}\|_{L_{t}^{\infty}L_{\vec w}^{2}L_{\vec x}^{2}}^{2}+Cc(\epsilon+\eta)(\|V_{r}\|_{L_{t}^{\infty}L_{\vec x}^{2}}^{2}+\|D_{x}^{\gamma}V_{r}\|_{L_{t}^{\infty}L_{\vec x}^{2}}^{2}),
\end{aligned}
\end{equation}
which can be further rewritten as
\begin{equation}\label{guji12}
\begin{aligned}
&\|D_{x}^{\gamma}V_{r}(t)\|_{L_{\vec x}^{2}}^{2}+\|D_{x}^{\gamma}f_{r}(t)\|_{L_{\vec w}^{2}L_{\vec x}^{2}}^{2}
+\frac{1}{\epsilon^{2}}\int_{0}^{t}\int_{\mathbb{T}^{3}}\int_{\mathbb{S}^{2}}|D_{x}^{\gamma}f_{r}-D_{x}^{\gamma}\overline{f_{r}}|^{2}
\mathrm{d}\vec w\mathrm{d}\vec x\mathrm{d}s\\&+\frac{1}{\epsilon^{2}}\int_{0}^{t}\int_{\mathbb{T}^{3}}|D_{x}^{\gamma}(4(\theta^{N})^{3}\theta_{r}-\overline{f_{r}})|^{2}\mathrm{d}\vec x\mathrm{d}s\\\leq&C\Big(\frac{1}{\epsilon^{4}}\|(R_{1}, R_{2}, R_{3}, R_{4}, R)\|_{L_{t}^{2}L_{\vec w}^{2}H_{\vec x}^{3}}^{2}+\frac{1}{\epsilon^{2}}\|f_{r}-\overline{f_{r}}\|_{L_{t}^{2}L_{\vec x}^{2}}^{2}\Big)\\&+C(\|V_{r}\|_{L_{t}^{2}L_{\vec x}^{2}}+\|D_{x}^{\gamma}V_{r}\|_{L_{t}^{2}L_{\vec x}^{2}}+\|f_{r}\|_{L_{t}^{2}L_{\vec x}^{2}L_{\vec w}^{2}}^{2}+\|D_{x}^{\gamma}f_{r}\|_{L_{t}^{2}L_{\vec x}^{2}L_{\vec w}^{2}}^{2})+\frac{C}{\epsilon^{2}}(\|f_{r}-\overline{f_{r}}\|_{L_{t}^{2}L_{\vec w}^{2}L_{\vec x}^{2}}^{2}
\\&+\|D_{x}^{\gamma-1}f_{r}-D_{x}^{\gamma-1}\overline{f_{r}}\|_{L_{t}^{2}L_{\vec w}^{2}L_{\vec x}^{2}}^{2})+\frac{C}{\epsilon^{4}}(\|f_{r}\|_{L_{t}^{2}L_{\vec w}^{2}L_{\vec x}^{2}}^{2}+\|D_{x}^{\gamma-1}f_{r}\|_{L_{t}^{2}L_{\vec w}^{2}L_{\vec x}^{2}}^{2}+\|\theta_{r}\|_{L_{t}^{2}L_{\vec w}^{2}}^{2}
\end{aligned}
\end{equation}
\begin{equation}\nonumber
\begin{aligned}
\\&+\|D_{x}^{\gamma-1}\theta_{r}\|_{L_{t}^{2}L_{\vec x}^{2}}^{2})
+C\eta(\|V_{r}\|_{L_{t}^{\infty}L_{\vec x}^{2}}^{2}+\|D_{x}^{\gamma}V_{r}\|_{L_{t}^{\infty}L_{\vec x}^{2}}^{2})+C\eta\|D_{x}^{\gamma}f_{r}\|_{L_{t}^{\infty}L_{\vec w}^{2}L_{\vec x}^{2}}^{2}.\hspace{2cm}
\end{aligned}
\end{equation}
As $t\in [0, T]$ was chosen arbitrarily, it follows that
\begin{equation}\label{guji12}
\begin{aligned}
&\|D_{x}^{\gamma}V_{r}\|_{L_{t}^{\infty}L_{\vec x}^{2}}^{2}+\|D_{x}^{\gamma}f_{r}\|_{L_{t}^{\infty}L_{\vec w}^{2}L_{\vec x}^{2}}^{2}
+\frac{1}{\epsilon^{2}}\int_{0}^{t}\int_{\mathbb{T}^{3}}\int_{\mathbb{S}^{2}}|D_{x}^{\gamma}f_{r}-D_{x}^{\gamma}\overline{f_{r}}|^{2}
\mathrm{d}\vec w\mathrm{d}\vec x\mathrm{d}s\\&+\frac{1}{\epsilon^{2}}\int_{0}^{t}\int_{\mathbb{T}^{3}}|D_{x}^{\gamma}(4(\theta^{N})^{3}\theta_{r}-\overline{f_{r}})|^{2}\mathrm{d}\vec x\mathrm{d}s\\\leq&C\Big(\frac{1}{\epsilon^{4}}\|(R_{1}, \vec{R}_{2}, R_{3}, R_{4}, R)\|_{L_{t}^{2}L_{\vec w}^{2}H_{\vec x}^{3}}^{2}+\frac{1}{\epsilon^{2}}\|f_{r}-\overline{f_{r}}\|_{L_{t}^{2}L_{\vec w}^{2}L_{\vec x}^{2}}^{2}\Big)\\&+C(\|V_{r}\|_{L_{t}^{2}L_{\vec x}^{2}}+\|D_{x}^{\gamma}V_{r}\|_{L_{t}^{2}L_{\vec x}^{2}}+\|f_{r}\|_{L_{t}^{2}L_{\vec x}^{2}L_{\vec w}^{2}}^{2}+\|D_{x}^{\gamma}f_{r}\|_{L_{t}^{2}L_{\vec w}^{2}L_{\vec x}^{2}}^{2})+\frac{C}{\epsilon^{2}}(\|f_{r}-\overline{f_{r}}\|_{L_{t}^{2}L_{\vec w}^{2}L_{\vec x}^{2}}^{2}\\&+\|D_{x}^{\gamma-1}f_{r}-D_{x}^{\gamma-1}\overline{f_{r}}\|_{L_{t}^{2}L_{\vec w}^{2}L_{\vec x}^{2}}^{2})+\frac{C}{\epsilon^{4}}(\|f_{r}\|_{L_{t}^{2}L_{\vec w}^{2}L_{\vec x}^{2}}^{2}+\|D_{x}^{\gamma-1}f_{r}\|_{L_{t}^{2}L_{\vec w}^{2}L_{\vec x}^{2}}^{2}+\|\theta_{r}\|_{L_{t}^{2}L_{\vec w}^{2}}^{2}\\&+\|D_{x}^{\gamma-1}\theta_{r}\|_{L_{t}^{2}L_{\vec x}^{2}}^{2})
+C(\eta+\epsilon)(\|V_{r}\|_{L_{t}^{\infty}L_{\vec x}^{2}}^{2}+\|D_{x}^{\gamma}V_{r}\|_{L_{t}^{\infty}L_{\vec x}^{2}}^{2})+C\eta\|D_{x}^{\gamma}f_{r}\|_{L_{t}^{\infty}L_{\vec w}^{2}L_{\vec x}^{2}}^{2}.
\end{aligned}
\end{equation}
Choose small $\eta$ and $\epsilon$ such that $C(\eta+\epsilon)=\frac{1}{2}$. Then we have
\begin{equation}\label{guji12}
\begin{aligned}
&\frac{1}{2}(\|D_{x}^{\gamma}V_{r}\|_{L_{t}^{\infty}L_{\vec x}^{2}}^{2}+\|D_{x}^{\gamma}f_{r}\|_{L_{t}^{\infty}L_{\vec w}^{2}L_{\vec x}^{2}}^{2})
+\frac{1}{\epsilon^{2}}\int_{0}^{t}\int_{\mathbb{T}^{3}}\int_{\mathbb{S}^{2}}|D_{x}^{\gamma}f_{r}-D_{x}^{\gamma}\overline{f_{r}}|^{2}
\mathrm{d}\vec w\mathrm{d}\vec x\mathrm{d}s
\\&+\frac{1}{\epsilon^{2}}\int_{0}^{t}\int_{\mathbb{T}^{3}}|D_{x}^{\gamma}(4(\theta^{N})^{3}\theta_{r}-\overline{f_{r}})|^{2}\mathrm{d}\vec x\mathrm{d}s\\\leq&C\Big(\frac{1}{\epsilon^{4}}\|(R_{1}, \vec{R}_{2}, R_{3}, R_{4}, R)\|_{L_{t}^{2}L_{\vec w}^{2}H_{\vec x}^{3}}^{2}+\frac{1}{\epsilon^{2}}\|f_{r}-\overline{f_{r}}\|_{L_{t}^{2}L_{\vec w}^{2}L_{\vec x}^{2}}^{2}\Big)+C(\|V_{r}\|_{L_{t}^{2}L_{\vec x}^{2}}+\|D_{x}^{\gamma}V_{r}\|_{L_{t}^{2}L_{\vec x}^{2}}\\&+\|f_{r}\|_{L_{t}^{2}L_{\vec w}^{2}L_{\vec x}^{2}}^{2}+\|D_{x}^{\gamma}f_{r}\|_{L_{t}^{2}L_{\vec w}^{2}L_{\vec x}^{2}}^{2})+\frac{C}{\epsilon^{2}}(\|f_{r}-\overline{f_{r}}\|_{L_{t}^{2}L_{\vec w}^{2}L_{\vec x}^{2}}^{2}+\|D_{x}^{\gamma-1}f_{r}-D_{x}^{\gamma-1}\overline{f_{r}}\|_{L_{t}^{2}L_{\vec w}^{2}L_{\vec x}^{2}}^{2})\\&+\frac{C}{\epsilon^{4}}(\|f_{r}\|_{L_{t}^{2}L_{\vec w}^{2}L_{\vec x}^{2}}^{2}+\|D_{x}^{\gamma-1}f_{r}\|_{L_{t}^{2}L_{\vec w}^{2}L_{\vec x}^{2}}^{2}+\|\theta_{r}\|_{L_{t}^{2}L_{\vec w}^{2}}^{2}+\|D_{x}^{\gamma-1}\theta_{r}\|_{L_{t}^{2}L_{\vec x}^{2}}^{2}).
\end{aligned}
\end{equation}
Furthermore, 
\begin{equation}\label{guji13}
\begin{aligned}
&\|D_{x}^{\gamma}V_{r}(t)\|_{L_{\vec x}^{2}}^{2}+\|D_{x}^{\gamma}f_{r}(t)\|_{L_{\vec w}^{2}L_{\vec x}^{2}}^{2}
+\frac{1}{\epsilon^{2}}\int_{0}^{t}\int_{\mathbb{T}^{3}}\int_{\mathbb{S}^{2}}|D_{x}^{\gamma}f_{r}-D_{x}^{\gamma}\overline{f_{r}}|^{2}
\mathrm{d}\vec w\mathrm{d}\vec x\mathrm{d}s\\&+\frac{1}{\epsilon^{2}}\int_{0}^{t}\int_{\mathbb{T}^{3}}|D_{x}^{\gamma}(4(\theta^{N})^{3}\theta_{r}-\overline{f_{r}})|^{2}\mathrm{d}\vec x\mathrm{d}s
\\\leq&C\Big(\frac{1}{\epsilon^{4}}\|(R_{1}, R_{2}, R_{3}, R_{4}, R)\|_{L_{t}^{2}L_{\vec w}^{2}H_{\vec x}^{3}}^{2}+\frac{1}{\epsilon^{2}}\|f_{r}-\overline{f_{r}}\|_{L_{t}^{2}L_{\vec w}^{2}L_{\vec x}^{2}}^{2}\Big)+C(\|V_{r}\|_{L_{t}^{2}L_{\vec x}^{2}}+\|D_{x}^{\gamma}V_{r}\|_{L_{t}^{2}L_{\vec x}^{2}}\\&+\|f_{r}\|_{L_{t}^{2}L_{\vec w}^{2}L_{\vec x}^{2}}^{2}+\|D_{x}^{\gamma}f_{r}\|_{L_{t}^{2}L_{\vec w}^{2}L_{\vec x}^{2}}^{2})+\frac{C}{\epsilon^{2}}(\|f_{r}-\overline{f_{r}}\|_{L_{t}^{2}L_{\vec w}^{2}L_{\vec x}^{2}}^{2}
+\|D_{x}^{\gamma-1}f_{r}-D_{x}^{\gamma-1}\overline{f_{r}}\|_{L_{t}^{2}L_{\vec w}^{2}L_{\vec x}^{2}}^{2})\\&+\frac{C}{\epsilon^{4}}(\|f_{r}\|_{L_{t}^{2}L_{\vec w}^{2}L_{\vec x}^{2}}^{2}+\|D_{x}^{\gamma-1}f_{r}\|_{L_{t}^{2}L_{\vec w}^{2}L_{\vec x}^{2}}^{2}+\|\theta_{r}\|_{L_{t}^{2}L_{\vec w}^{2}}^{2}+\|D_{x}^{\gamma-1}\theta_{r}\|_{L_{t}^{2}L_{\vec x}^{2}}^{2}).
\end{aligned}
\end{equation}
Specially, for $\gamma=0$, we can derive
\begin{equation}\label{guji14}
\begin{aligned}
&\|V_{r}(t)\|_{L_{\vec x}^{2}}^{2}+\|f_{r}(t)\|_{L_{\vec w}^{2}L_{\vec x}^{2}}^{2}
+\frac{1}{\epsilon^{2}}\int_{0}^{t}\int_{\mathbb{T}^{3}}\int_{\mathbb{S}^{2}}|f_{r}-\overline{f_{r}}|^{2}
\mathrm{d}\vec w\mathrm{d}\vec x\mathrm{d}s+\frac{1}{\epsilon^{2}}\int_{0}^{t}\int_{\mathbb{T}^{3}}|(4(\theta^{N})^{3}\theta_{r}-\overline{f_{r}})|^{2}\mathrm{d}\vec x\mathrm{d}s
\end{aligned}
\end{equation}
\begin{equation}\nonumber
\begin{aligned}
\\\leq&C\Big(\frac{1}{\epsilon^{4}}\|(R_{1}, R_{2}, R_{3}, R_{4}, R)\|_{L_{t}^{2}L_{\vec w}^{2}H_{\vec x}^{3}}^{2}\Big)+C(\|V_{r}\|_{L_{t}^{2}L_{\vec x}^{2}}+\|f_{r}\|_{L_{t}^{2}L_{\vec w}^{2}L_{\vec x}^{2}}^{2}).\hspace{4cm}
\end{aligned}
\end{equation}
Consequently, an application of Grönwall's inequality yields
\begin{equation}\label{guji15}
\begin{aligned}
&\|V_{r}\|_{L_{t}^{\infty}L_{\vec x}^{2}}^{2}+\|f_{r}\|_{L_{t}^{\infty}L_{\vec w}^{2}L_{\vec x}^{2}}^{2}
+\frac{1}{\epsilon^{2}}\int_{0}^{t}\int_{\mathbb{T}^{3}}\int_{\mathbb{S}^{2}}|f_{r}-\overline{f_{r}}|^{2}
\mathrm{d}\vec w\mathrm{d}\vec x\mathrm{d}s+\frac{1}{\epsilon^{2}}\int_{0}^{t}\int_{\mathbb{T}^{3}}|(4(\theta^{N})^{3}\theta_{r}-\overline{f_{r}})|^{2}\mathrm{d}\vec x\mathrm{d}s\\\leq&C(t)\Big(\frac{1}{\epsilon^{4}}\|(R_{1}, \vec{R}_{2}, R_{3}, R_{4}, R)\|_{L_{t}^{2}L_{\vec w}^{2}H_{\vec x}^{3}}^{2}\Big).
\end{aligned}
\end{equation}
From \eqref{guji15}, with $\gamma=1$, \eqref{guji13} can be written as
\begin{equation}\label{guji16}
\begin{aligned}
&\|D_{x}V_{r}(t)\|_{L_{\vec x}^{2}}^{2}+\|D_{x}f_{r}(t)\|_{L_{\vec w}^{2}L_{\vec x}^{2}}^{2}
+\frac{1}{\epsilon^{2}}\int_{0}^{t}\int_{\mathbb{T}^{3}}\int_{\mathbb{S}^{2}}|D_{x}f_{r}-D_{x}\overline{f_{r}}|^{2}
\mathrm{d}\vec w\mathrm{d}\vec x\mathrm{d}s\\&+\frac{1}{\epsilon^{2}}\int_{0}^{t}\int_{\mathbb{T}^{3}}|D_{x}(4(\theta^{N})^{3}\theta_{r}-\overline{f_{r}})|^{2}\mathrm{d}\vec x\mathrm{d}s\\\leq&C\Big(\frac{1}{\epsilon^{8}}\|(R_{1}, \vec{R}_{2}, R_{3}, R_{4}, R)\|_{L_{t}^{2}L_{\vec w}^{2}H_{\vec x}^{3}}^{2}\Big)+C(\|D_{x}V_{r}\|_{L_{t}^{2}L_{\vec x}^{2}}^{2}+\|D_{x}f_{r}\|_{L_{t}^{2}L_{\vec w}^{2}L_{\vec x}^{2}}^{2}).
\end{aligned}
\end{equation}
By the Gronwall's inequality, we have
\begin{equation}\label{guji17}
\begin{aligned}
&\|D_{x}V_{r}\|_{L_{t}^{\infty}L_{\vec x}^{2}}^{2}+\|D_{x}f_{r}\|_{L_{t}^{\infty}L_{\vec w}^{2}L_{\vec x}^{2}}^{2}
+\frac{1}{\epsilon^{2}}\int_{0}^{t}\int_{\mathbb{T}^{3}}\int_{\mathbb{S}^{2}}|D_{x}f_{r}-D_{x}\overline{f_{r}}|^{2}
\mathrm{d}\vec w\mathrm{d}\vec x\mathrm{d}s\\&+\frac{1}{\epsilon^{2}}\int_{0}^{t}\int_{\mathbb{T}^{3}}|D_{x}(4(\theta^{N})^{3}\theta_{r}-\overline{f_{r}})|^{2}\mathrm{d}\vec x\mathrm{d}s\\\leq&C(t)\Big(\frac{1}{\epsilon^{8}}\|(R_{1}, \vec{R}_{2}, R_{3}, R_{4}, R)\|_{L_{t}^{2}L_{\vec w}^{2}H_{\vec x}^{3}}^{2}\Big).
\end{aligned}
\end{equation}
Likewise, for $\gamma=2$ and $\gamma=3$, it follows that
\begin{equation}\label{guji18}
\begin{aligned}
&\|D_{x}^{2}V_{r}\|_{L_{t}^{\infty}L_{\vec x}^{2}}^{2}+\|D_{x}^{2}f_{r}\|_{L_{t}^{\infty}L_{\vec w}^{2}L_{\vec x}^{2}}^{2}
+\frac{1}{\epsilon^{2}}\int_{0}^{t}\int_{\mathbb{T}^{3}}\int_{\mathbb{S}^{2}}|D_{x}^{2}f_{r}-D_{x}^{2}\overline{f_{r}}|^{2}
\mathrm{d}\vec w\mathrm{d}\vec x\mathrm{d}s\\&+\frac{1}{\epsilon^{2}}\int_{0}^{t}\int_{\mathbb{T}^{3}}|D_{x}^{2}(4(\theta^{N})^{3}\theta_{r}-\overline{f_{r}})|^{2}\mathrm{d}\vec x\mathrm{d}s\leq C(t)\Big(\frac{1}{\epsilon^{12}}\|(R_{1}, \vec{R}_{2}, R_{3}, R_{4}, R)\|_{L_{t}^{2}L_{\vec w}^{2}H_{\vec x}^{3}}^{2}\Big)
\end{aligned}
\end{equation}
and
\begin{equation}\label{guji19}
\begin{aligned}
&\|D_{x}^{3}V_{r}\|_{L_{t}^{\infty}L_{\vec x}^{2}}^{2}+\|D_{x}^{3}f_{r}\|_{L_{t}^{\infty}L_{\vec w}^{2}L_{\vec x}^{2}}^{2}
+\frac{1}{\epsilon^{2}}\int_{0}^{t}\int_{\mathbb{T}^{3}}\int_{\mathbb{S}^{2}}|D_{x}^{3}f_{r}-D_{x}^{3}\overline{f_{r}}|^{2}
\mathrm{d}\vec w\mathrm{d}\vec x\mathrm{d}s\\&+\frac{1}{\epsilon^{2}}\int_{0}^{t}\int_{\mathbb{T}^{3}}|D_{x}^{3}(4(\theta^{N})^{3}\theta_{r}-\overline{f_{r}})|^{2}\mathrm{d}\vec x\mathrm{d}s\leq C(t)\Big(\frac{1}{\epsilon^{16}}\|(R_{1}, \vec{R}_{2}, R_{3}, R_{4}, R)\|_{L_{t}^{2}L_{\vec w}^{2}H_{\vec x}^{3}}^{2}\Big),
\end{aligned}
\end{equation}
respectively.

Collecting the estimates from $\gamma=0$ to $\gamma=3$, we get the following desired estimate
\begin{equation}\label{guji20}
\begin{aligned}
&\|V_{r}\|_{L_{t}^{\infty}H_{\vec x}^{3}}^{2}+\|f_{r}\|_{L_{t}^{\infty}L_{\vec w}^{2}H_{\vec x}^{3}}^{2}+\frac{1}{\epsilon^{2}}\|f_{r}-\overline{f_{r}}\|_{L_{t}^{2}L_{\vec w}^{2}H_{\vec x}^{3}}^{2}+\frac{1}{\epsilon^{2}}\|4(\theta^{N})^{3}\theta_{r}-\overline{f_{r}}\|_{L_{t}^{2}H_{\vec x}^{3}}^{2}
\\\leq& C(t)\Big(\frac{1}{\epsilon^{16}}\|(R_{1}, \vec{R}_{2}, R_{3}, R_{4}, R)\|_{L_{t}^{2}L_{\vec w}^{2}H_{\vec x}^{3}}^{2}\Big).
\end{aligned}
\end{equation}

\subsection{Proof of the main theorem}\label{sec4.2}

We now show the existence and uniqueness of solutions to system \eqref{research equations} around the constructed composite approximate solution $(\rho^{N}, \vec{u}^{N}, \theta^{N}, f^{N})$ and finish the proof of Theorem \ref{mainthm}. Due to the equivalence between system \eqref{reeq1}-\eqref{rein} and \eqref{research equations}, we only need to show the existence and uniqueness of solutions for \eqref{reeq1}-\eqref{rein} in the neighborhood of zero.

\noindent\textbf{Proof of Theorem \ref{mainthm}.} The proof of existence and uniqueness is obtained using the Banach fixed point theorem. We first construct a sequence of functions and then show the sequence satisfies the contraction principle. Finally we show the convergence of $(\rho_{r}, \vec{u}_{r}, \theta_{r}, f_{r})$ to zero as $\epsilon\rightarrow 0$.

\noindent\textbf{Step 1.}Construction of sequence of functions. Let $\{\rho_{r}^{0}, \vec{u}_{r}^{0}, \theta_{r}^{0}, f_{r}^{0}\}$ be zero functions
\begin{equation}\nonumber
\rho_{r}^{0}=0, \ \ \vec{u}_{r}^{0}=0, \ \ \theta_{r}^{0}=0,\ \ f_{r}^{0}=0,
\end{equation}
and for $k\geq 1,$ $\{\rho_{r}^{k+1}, \vec{u}_{r}^{k+1}, \theta_{r}^{k+1}, f_{r}^{k+1}\}$ are defined recursively by
\begin{equation}\label{reeq1-jiajia1}
\partial_{t}\rho_{r}^{k+1}+\vec{u}_{r}^{k+1}\cdot\nabla_{x}\rho^{N}+\tilde{\vec{u}}^{k}\cdot\nabla_{x}\rho_{r}^{k+1}+\rho_{r}^{k+1}\mathrm{div}\vec{u}^{N}
+\tilde{\rho}^{k}\mathrm{div}\vec{u}_{r}^{k+1}=-\mathcal{L}_{1}(\rho^{N}, \vec{u}^{N}),
\end{equation}
\begin{equation}\label{reeq2-jiajia1}
\begin{aligned}
&\epsilon^{2}\rho_{r}^{k+1}\partial_{t}\vec{u}^{N}+\epsilon^{2}\tilde{\rho}^{k}\partial_{t}\vec{u}_{r}^{k+1}+\epsilon^{2}(\tilde\rho^{k}\tilde{\vec{u}}^{k}\cdot\nabla_{x}\vec{u}_{r}^{k+1}
+\rho_{r}^{k+1}\tilde{\vec{u}}^{k}\cdot\nabla_{x}\vec{u}^{N}+\rho^{N}\vec{u}_{r}^{k+1}\cdot\nabla_{x}\vec{u}^{N})
\\&+\epsilon^{2}(\rho_{r}^{k+1}\nabla_{x}\theta^{N}+\tilde\rho^{k}\nabla_{x}\theta_{r}^{k+1})+\epsilon^{2}(\theta_{r}^{k+1}\nabla_{x}\rho^{N}+\tilde\theta^{k}\nabla_{x}\rho_{r}^{k+1})
-\langle(\epsilon+\epsilon^{3})\vec w(f_{r}^{k+1}-\overline{f_{r}^{k+1}})\rangle
\\=&-\mathcal{L}_{2}(\rho^{N}, \vec{u}^{N}, \theta^{N}, f^{N}),
\end{aligned}
\end{equation}
\begin{equation}\label{reeq3-jiajia1}
\begin{aligned}
&\epsilon^{2}(\tilde\rho^{k}\partial_{t}\theta_{r}^{k+1}+\rho_{r}^{k+1}\partial_{t}\theta^{N})+\epsilon^{2}(\tilde\rho^{k}\tilde{\vec{u}}^{k}\cdot\nabla_{x}\theta_{r}^{k+1}
+\rho_{r}^{k+1}\tilde{\vec{u}}^{k}\cdot\nabla_{x}\theta^{N}+\rho^{N}\vec{u}_{r}^{k+1}\cdot\nabla_{x}\theta^{N})
\\&+\epsilon^{2}(\tilde\rho^{k}\tilde\theta^{k}\mathrm{div}\vec {u}_{r}^{k+1}+\rho_{r}^{k+1}\theta^{N}\mathrm{div}\vec {u}^{N}+\rho^{N}\theta_{r}^{k+1}\mathrm{div}\vec {u}^{N})+4(\theta^{N})^{3}\theta_{r}^{k+1}-\overline{f_{r}^{k+1}}\\=&-\mathcal{L}_{3}(\rho^{N}, \vec{u}^{N}, \theta^{N}, f^{N})+(\theta_{r}^{k})^{4}+6(\theta_{r}^{k})^{2}(\theta^{N})^{2}+4(\theta_{r}^{k})^{3}(\theta^{N})
\end{aligned}
\end{equation}
and
\begin{equation}\label{reeq4-jiajia1}
\begin{aligned}
&\epsilon^{2}\partial_{t}f_{r}^{k+1}+\epsilon\vec{w}\cdot\nabla_{x}f_{r}^{k+1}+\epsilon^{2}(f_{r}^{k+1}-\overline{f_{r}^{k+1}})+f_{r}^{k+1}
-4(\theta^{N})^{3}\theta_{r}^{k+1}\\=&-\mathcal{L}_{4}(\rho^{N}, \theta^{N})-(\theta_{r}^{k})^{4}-6(\theta_{r}^{k})^{2}(\theta^{N})^{2}-4(\theta_{r}^{k})^{3}(\theta^{N}),
\end{aligned}
\end{equation}
with initial conditions
\begin{equation}\label{rein-jia1}
\rho_{r}^{k+1}(0, \vec x)=0,\ \ \ \vec{u}_{r}^{k+1}(0, \vec x)=0,\ \ \theta_{r}^{k+1}(0, \vec x)=0,\ \ f_{r}^{k+1}(0, \vec x, \vec w)=0,\ \ \mathrm{for}\ \ (\vec x, \vec w)\in\mathbb{T}^{3}\times\mathbb{S}^{2}.
\end{equation}
Define $V_{r}^{k}=(\rho_{r}^{k}, \vec{u}_{r}^{k}, \theta_{r}^{k})$, $\tilde V^{k}=(\tilde\rho^{k}, \tilde{\vec{u}}^{k}, \tilde\theta^{k})=(\rho^{N}+\rho_{r}^{k}, \vec{u}^{N}+\vec{u}_{r}^{k}, \theta^{N}+\theta_{r}^{k})$. Then, we can write \eqref{reeq1-jiajia1}-\eqref{reeq4-jiajia1} as follows
\begin{equation}\label{research equations-3-2}
E^{0}(\tilde{V}^{k})\frac{\partial {V}_{r}^{k+1}}{\partial t}+\sum_{j=1}^{3}E_{j}^{1}(\tilde{V^{k}})\frac{\partial {V}_{r}^{k+1}}{\partial x_{j}}+E^{2}(\tilde V^{k})V_{r}^{k+1}=\vec H^{k},
\end{equation}
and
\begin{equation}\label{reeq4-jiajia1-1}
\begin{aligned}
&\epsilon^{2}\partial_{t}f_{r}^{k+1}+\epsilon\vec{w}\cdot\nabla_{x}f_{r}^{k+1}+\epsilon^{2}(f_{r}^{k+1}-\overline{f_{r}^{k+1}})+f_{r}^{k+1}
-4(\theta^{N})^{3}\theta_{r}^{k+1}={R}_{4}(\rho^{N}, \theta^{N})-R,
\end{aligned}
\end{equation}
with initial conditions
\begin{equation}\label{rein-jia1}
V_{r}^{k+1}(0, \vec x)=0,\ \ f_{r}^{k+1}(0, \vec x, \vec w)=0,\ \ \mathrm{for}\ \ (\vec x, \vec w)\in\mathbb{T}^{3}\times\mathbb{S}^{2}.
\end{equation}
where $E^{0}(\tilde{V}^{k})$, $E_{j}^{1}(\tilde{V}^{k})$ and $E^{2}(\tilde V^{k})$ can be determined by $E^{0}(\tilde{V})$, $E_{j}^{1}(\tilde{V})$ and $E^{2}(\tilde V)$, respectively through substituting $\tilde{V}$ by $\tilde{V^{k}}$.

$\vec H^{k}$ has the similar structure of $\vec H$, defined as follows
\begin{equation}\nonumber
\vec H^{k}=(H_{1}^{k}, H_{2}^{k}, H_{3}^{k}, H_{4}^{k}, H_{5}^{k})^{t}
\end{equation} 
with
\begin{equation}\nonumber
H_{1}^{k}=\frac{1}{\tilde\rho^{k}}{R}_{1},\ \ H_{j}^{k}=\frac{1}{\tilde{\theta^{k}}\epsilon^{2}}\vec{R}_{2}^{j}+\frac{1}{\tilde{\theta}^{k}}\langle(\epsilon+\frac{1}{\epsilon})w_{j-1}(f_{r}^{k+1}-\overline{f_{r}^{k+1}})\rangle,\ \ j=2, 3, 4,
\end{equation}
\begin{equation}\nonumber
H_{5}^{k}=\frac{1}{(\tilde{\theta}^{k})^{2}\epsilon^{2}}({R}_{3}+\langle R\rangle)+\frac{1}{(\tilde{\theta}^{k})^{2}\epsilon^{2}}(4(\theta^{N})^{3}\theta_{r}^{k+1}-\overline{f_{r}^{k+1}}),
\end{equation}
where $R_{1}, \vec{R}_{2}, R_{3}, R_{4}$ are the same as before, and a little differently,
\begin{equation}\nonumber
R=(\theta_{r}^{k})^{4}+6(\theta_{r}^{k})^{2}(\theta^{N})^{2}+4(\theta_{r}^{k})^{3}(\theta^{N}).
\end{equation}

The above system defines a mapping $\mathcal{T}$ with $(V_{r}^{k+1}, f_{r}^{k+1})=\mathcal{T}((V_{r}^{k}, f_{r}^{k}))$.

\noindent\textbf{Step 2.} The contraction mapping. Define $V_{r}=(\rho_{r}, \vec{u}_{r}, \theta_{r})$. We consider the solution in the function space
\begin{equation}\nonumber
O_{q}:=\{(V_{r}, f_{r})\in L^{\infty}_{T}L^{2}_{\vec w}H^{3}_{\vec x}\times L^{\infty}_{T}H^{3}_{\vec x}: \|f_{r}\|_{L^{\infty}_{T}L^{2}_{\vec w}H^{3}_{\vec x}}+\|V_{r}\|_{L^{\infty}_{T}H^{3}_{\vec x}}\leq \epsilon^{q}, \|\partial_{t}\rho_{r}, \partial_{t}\theta_{r}\|_{L_{T}^{\infty}H_{\vec x}^{2}}\leq 1\},
\end{equation}
where $q>0$ is a constant to be chosen later.

First, we show $\mathcal{T}$ maps the space $O_{q}$ into itself. By \eqref{bd00} with $R_{1}=-\mathcal{L}_{1}(\rho^{N}, f^{N})$, $\vec{R}_{2}=-\mathcal{L}_{2}(\rho^{N}, \vec{u}^{N}, \theta^{N}, f^{N})$, ${R}_{3}=-\mathcal{L}_{3}(\rho^{N}, \vec{u}^{N}, \theta^{N}, f^{N})$, ${R}_{4}=-\mathcal{L}_{4}(\theta^{N}, f^{N})$ and
$R=-6(\theta^{N})^{2}(\theta_{r}^{k})^{2}-4\theta^{N}(\theta_{r}^{k})^{3}-(\theta_{r}^{k})^{4}$, the following estimate holds (assuming $q\geq 1$):
\begin{equation}\label{bd00-jia1}
\begin{aligned}
&\|f_{r}^{k+1}\|_{L_{T}^{\infty}H^{3}_{\vec x}L_{\vec w}^{2}}+\|V_{r}^{k+1}\|_{L_{T}^{\infty}H_{\vec x}^{3}}+\frac{1}{\epsilon^{2}}\|f_{r}^{k+1}-4(\theta^{N})^{3}\theta_{r}^{k+1}\|_{L_{T}^{2}L_{\vec w}^{2}H_{\vec x}^{3}}+\frac{1}{\epsilon}\|f_{r}^{k+1}-\overline{f_{r}^{k+1}}\|_{L_{T}^{2}L_{\vec w}^{2}H_{\vec x}^{3}}
\\\leq& \frac{C(T)}{\epsilon^{8}}\|(R_{1}, R_{2}, R_{3}, R_{4}, R)\|_{L_{T}^{2}L_{\vec w}^{2}H_{\vec x}^{3}}
\\\leq& \frac{C(T)}{\epsilon^{8}}\Big(\|6(\theta^{N})^{2}(\theta_{r}^{k})^{2}+4\theta^{N}(\theta_{r}^{k})^{3}
+(\theta_{r}^{k})^{4}\|_{L_{T}^{2}L_{\vec w}^{2}H_{\vec x}^{2}}+\|\mathcal{L}_{1}(\rho^{N}, \vec{u}^{N})\|_{L_{T}^{2}L_{\vec w}^{2}H_{\vec x}^{3}}\\&+\|\mathcal{L}_{2}(\rho^{N}, \vec{u}^{N}, \theta^{N}, f^{N})\|_{L_{T}^{2}L_{\vec w}^{2}H_{\vec x}^{3}}+\|\mathcal{L}_{3}(\rho^{N}, \vec{u}^{N}, \theta^{N}, f^{N})\|_{L_{T}^{2}L_{\vec w}^{2}H_{\vec x}^{3}}+\|\mathcal{L}_{4}(\theta^{N}, f^{N})\|_{L_{T}^{2}L_{\vec w}^{2}H_{\vec x}^{3}}\Big)
\\\leq& \frac{C(T)}{\epsilon^{8}}(\|\theta_{r}^{k}\|_{L_{T}^{\infty}H^{3}_{\vec x}}^{2}+\|\theta_{r}^{k}\|_{L_{T}^{\infty}H^{3}_{\vec x}}^{4}+\epsilon^{N+1})
\leq \frac{C(T)}{\epsilon^{8}}(\epsilon^{2q}+\epsilon^{N+1})
\leq C(T)(\epsilon^{2q-8}+\epsilon^{N-7}).
\end{aligned}
\end{equation}
Assuming $2q-8>q$ and $N-7>q$, i.e. $q>8$ and $N>7+q$, the above inequality implies
\begin{equation}\nonumber
\|f_{r}^{k+1}\|_{L_{T}^{\infty}L_{\vec w}^{2}H^{3}_{\vec x}}+\|V_{r}^{k+1}\|_{L_{T}^{\infty}H_{\vec x}^{3}}\leq C(T)(\epsilon^{2q-8}+\epsilon^{N-7})\leq \epsilon^{q},
\end{equation}
for sufficiently small $\epsilon$. For sufficiently small $\epsilon$, we can also verify that $\|\partial_{t}\rho_{r}^{k+1}, \partial_{t}\theta_{r}^{k+1}\|_{L_{T}^{\infty}H_{\vec x}^{2}}\\\leq C(\epsilon^{q}+\epsilon^{2q-2}+\epsilon^{N-1}+\epsilon^{q-2})\leq 1$. Thus we obtain that $(f_{r}^{k+1}, \theta_{r}^{k+1})\in O_{q}$ and therefore $\mathcal{T}$ maps $O_{q}$ into itself.

Next, we show the map $\mathcal{T}$ is a contraction mapping. Let $\varphi^{k+1}=f_{r}^{k+1}-f_{r}^{k}$, $h^{k+1}=V_{r}^{k+1}-V_{r}^{k}$, then they satisfy
\begin{equation}\label{reeq1-jia-new}
\begin{aligned}
&\epsilon^{2}\partial_{t}\varphi^{k+1}+\epsilon\vec{w}\cdot\nabla_{x}\varphi^{k+1}+\epsilon^{2}(\varphi^{k+1}
-\overline{\varphi^{k+1}})+\varphi^{k+1}
-4(\theta^{N})^{3}h^{k+1}\\=&6(\theta^{N})^{2}(\theta_{r}^{k}+\theta_{r}^{k-1})h^{k}
+4\theta^{N}((\theta_{r}^{k})^{2}+\theta_{r}^{k}\theta_{r}^{k-1}+(\theta_{r}^{k-1})^{2})h^{k}
+((\theta_{r}^{k})^{2}+(\theta_{r}^{k-1})^{2})(\theta_{r}^{k}+\theta_{r}^{k-1})h^{k},
\end{aligned}
\end{equation}
\begin{equation}\label{reeq2-jia-new}
\begin{aligned}
E^{0}(\tilde{V}^{k})\frac{\partial h^{k+1}}{\partial t}+\sum_{j=1}^{3}E_{j}^{1}(\tilde{V}^{k})\frac{\partial h^{k+1}}{\partial x_{j}}+E^{2}(\tilde V^{k})h^{k+1}=\vec H^{k}-\vec H^{k-1}+L,
\end{aligned}
\end{equation}
where
\begin{equation}\nonumber
L=-(E^{0}(\tilde{V}^{k})-E^{0}(\tilde{V}^{k-1}))\frac{\partial {V}_{r}^{k+1}}{\partial t}-(\sum_{j=1}^{3}E_{j}^{1}(\tilde{V}^{k})-E_{j}^{1}(\tilde{V}^{k-1}))\frac{\partial {V}_{r}^{k+1}}{\partial x_{j}}-(E^{2}(\tilde V^{k})-E^{2}(\tilde V^{k-1}))V_{r}^{k+1}
\end{equation}
with initial conditions
\begin{equation}\label{rein-jia-new}
\varphi^{k+1}(0, \vec x, \vec w)=0,\ \ h^{k+1}(0, \vec x)=0,\ \ \mathrm{for}\ \ (\vec x, \vec w)\in\mathbb{T}^{3}\times\mathbb{S}^{2}.
\end{equation}
From Sobolev's embedding theorem and Taylor's expansion theorem, we deduce that
\begin{equation}\nonumber
\|L\|_{L^{2}_{T}L^{2}_{\vec x}}\leq C(\epsilon^{N-1}+\epsilon^{2q-2})\|h^{k}\|_{L^{\infty}_{T}L^{2}_{\vec x}}
\end{equation}
and
\begin{equation}\nonumber
\|\vec H^{k}-\vec H^{k-1}\|_{L^{2}_{T}L^{2}_{\vec x}}\leq C(\epsilon^{N-1}+\epsilon^{2q-2})(\|\varphi^{k}\|_{L^{\infty}_{T}L^{2}_{\vec w}L^{2}_{\vec x}}+\|h^{k}\|_{L^{\infty}_{T}L^{2}_{\vec x}}).
\end{equation}
Multiplying \eqref{reeq1-jia-new} and \eqref{reeq2-jia-new} by $\varphi^{k+1}$ and $h^{k+1}$, respectively and integrating by parts in $L^{2}((0, T)\times\mathbb{T}^{3}\times\mathbb{S}^{2})$ and $L^{2}((0, T)\times\mathbb{T}^{3})$, we can derive
\begin{equation}\nonumber
\begin{aligned}
\|\varphi^{k+1}\|_{L_{T}^{\infty}L_{\vec w}^{2}L_{\vec x}^{2}}+\|h^{k+1}\|_{L_{T}^{\infty}L_{\vec x}^{2}}\leq C(\epsilon^{N-1}+\epsilon^{2q-2})(\|\varphi^{k}\|_{L^{\infty}_{T}L^{2}_{\vec w}L^{2}_{\vec x}}+\|h^{k}\|_{L^{\infty}_{T}L^{2}_{\vec x}}).
\end{aligned}
\end{equation}
Choosing small $\epsilon$ such that $0<C(\epsilon^{N-1}+\epsilon^{2q-2})<1$, we can derive that $\{f_{r}^{k}\}_{k=0}^{\infty}, \{V_{r}^{k}\}_{k=0}^{\infty}$ are contraction sequences. By the Banach fixed point theorem, there exist a unique fixed point $(f_{r}, V_{r})$ such that $(f_{r}, V_{r})=\mathcal{T}((f_{r}, V_{r}))$. Therefore, there exists a unique solution to \eqref{reeq1}-\eqref{rein} in $O_{q}$.

Taking $q=9$ and $N=17$, we can conclude that
\begin{equation}\nonumber
\|f_{r}\|_{L_{T}^{\infty}L_{\vec w}^{2}H^{3}_{\vec x}}+\|V_{r}\|_{L_{T}^{\infty}H^{3}_{\vec x}}\leq C\epsilon^{9}.
\end{equation}
In the following, we consider 
\begin{equation}\label{remainder equation-2}\left\{
\begin{split}
&\epsilon^{2}\partial_{t}f_{r}+\epsilon\vec{w}\cdot\nabla_{x}f_{r}+\epsilon^{2}(f_{r}-\overline{f_{r}})+f_{r}
=(\theta^{N}+\theta_{r})^{4}-(\theta^{N})^{4}-\mathcal{L}_{4}(f^{N}, \theta^{N})\ \ \mathrm{in} \ \ (0, T]\times \mathbb{T}^{3}\times \mathbb{S}^{2},\\
&f_{r}(0, \vec{x}, \vec{w})=0\ \ \mathrm{in}\ \ \mathbb{T}^{3}\times \mathbb{S}^{2}.\\
\end{split}\right.
\end{equation} 
According to Lemma \ref{main result}, we have
\begin{equation}\nonumber
\begin{aligned}
\|f_{r}\|_{L_{T}^{\infty}L_{\vec w}^{\infty}L^{\infty}_{\vec x}}\leq& C\|(\theta^{N}+\theta_{r})^{4}-(\theta^{N})^{4}-\mathcal{L}_{4}(f^{N}, \theta^{N})\|_{L_{T}^{\infty}L_{\vec w}^{\infty}L^{\infty}_{\vec x}}
\\\leq& C(\|\theta_{r}\|_{L_{T}^{\infty}L^{\infty}_{\vec x}}+\epsilon^{N+1})
\leq C(\|\theta_{r}\|_{L_{T}^{\infty}H^{3}_{\vec x}}+\epsilon^{N+1})
\leq C\epsilon^{9}.
\end{aligned}
\end{equation}
So, 
\begin{equation}\nonumber
\|f^{\epsilon}-\sum_{k=0}^{17}\epsilon^{k}(f_{k}+f_{I, k})\|_{L_{T}^{\infty}L_{\vec w}^{\infty}L^{\infty}_{\vec x}}
\\\leq C\epsilon^{9}
\end{equation}
and
\begin{equation}\nonumber
\|\theta^{\epsilon}-\sum_{k=0}^{17}\epsilon^{k}(\theta_{k}+\theta_{I, k})\|_{L_{T}^{\infty}H^{3}_{\vec x}}
\\\leq C\epsilon^{9}.
\end{equation}
Therefore, we derive
\begin{equation}\nonumber
\|f^{\epsilon}-\theta_{0}^{4}-f_{I, 0}\|_{L_{T}^{\infty}L_{\vec w}^{\infty}L^{\infty}_{\vec x}}
\\\leq C\epsilon
\end{equation}
and
\begin{equation}\nonumber
\|\theta^{\epsilon}-\theta_{0}-\theta_{I, 0}\|_{L_{T}^{\infty}H^{3}_{\vec x}}
\\\leq C\epsilon,
\end{equation}
combined with the fact $f_{0}=\overline{f_{0}}=\theta_{0}^{4}$.

We can also have 
\begin{equation}\nonumber
\|(\rho^{\epsilon}-\rho_{0}, \vec{u}^{\epsilon}-\vec{u}_{0})\|_{L_{T}^{\infty}H^{3}_{\vec x}}
\leq C\epsilon.
\end{equation}

The linear hyperbolic theory in \cite{cauchy-problem} implies $(\rho_{r}, \vec{u}_{r}, \theta_{r})\in C([0, T]; H^{3}(\mathbb{T}^{3}))$, then, we can derive \eqref{jielunbudengshi 1}, \eqref{jielunbudengshi 2} and \eqref{jielunbudengshi 3}.
\hfill$\Box$

\vskip 5mm

 \noindent{\bf Acknowledgments.} {Ju was partially supported by NSFC grants 12071044 and 12131007. Li was partially supported by the Scientic Research Project of Fuyang Normal University (No. 2024KYQD0106) and the Young Scientists Fund of Tianyuan Fund for Mathematics, NSFC grant 12526539. Zhang was partially supported by NSFC grants 12071044, 12271423 and 12671248.}

\end{document}